%% file: lmcs.tex
\documentclass{lmcs}
\pdfoutput=1

\usepackage{lastpage}
\lmcsdoi{21}{2}{26}
\lmcsheading{}{\pageref{LastPage}}{}{}%
{Aug.~31,~2022}{Jun.~23,~2025}{}

\usepackage{amssymb}
\usepackage{hyperref}
\usepackage{cleveref}
\usepackage[utf8]{inputenc}
\usepackage[T1]{fontenc}
\usepackage{microtype}
\usepackage{mathrsfs}
\usepackage{mathtools}
\usepackage[all]{xy}
\usepackage[defaultlines=3,all]{nowidow}
\usepackage{relsize}
\undef\cc
\usepackage{complexity}
\usepackage{graphicx}
\usepackage{braket}
\usepackage{xspace}
 \usepackage{calligra}
 \usepackage{comment}

\newcommand{\newdef}[2]{\hypertarget{def:#1}{\emph{#2}}}
\newcommand{\recalldef}[2]{\hyperlink{def:#1}{#2}}

\newcommand\xbar{\bar{x}}
\newcommand\bbar{\bar{b}}
\newcommand{\tCC}{\widetilde{\Cc}}
\newcommand{\Ra}{\Rightarrow}
\newcommand{\La}{\Leftarrow}

\crefname{thm}{Theorem}{Theorems}
\crefname{defi}{Definition}{Definitions}
\crefname{lem}{Lemma}{Lemmas}
\crefname{cor}{Corollary}{Corollaries}
\crefname{conj}{Conjecture}{Conjectures}
\crefname{exa}{Example}{Examples}
\crefname{prob}{Problem}{Problems}
\crefname{claim}{Claim}{Claims}
\crefname{prop}{Proposition}{Propositions}
\crefname{fact}{Fact}{Facts}
\crefname{rem}{Remark}{Remarks}
\crefname{part}{Part}{Parts}

\newcommand{\Cc}{\mathscr{C}}
\newcommand{\Dd}{\mathscr{D}}
\newcommand{\Ee}{\ensuremath{\mathscr E}}
\newcommand{\Gg}{\ensuremath{\mathscr G}}
\newcommand{\Hh}{\ensuremath{\mathscr H}}
\newcommand{\PW}{\mathscr{P\!W\!}}
\newcommand{\Ff}{\mathscr{F}}

\newcommand{\Pp}{\mathscr{P}}
\newcommand{\Cubic}{{\mathscr C}\!\!\textit{\calligra ubic}\,}
\newcommand{\Pl}{{\mathscr P}\!\!\textit{\calligra lanar}\,}

\newcommand{\Ss}{\mathscr{S}}
\newcommand{\Tt}{\mathscr{T}}
\newcommand{\TP}{\ensuremath{\mathscr{T\!\!P}}}

\newcommand{\N}{\mathbb{N}}

\newcommand{\dist}{\mathrm{dist}}

\renewcommand{\phi}{\varphi}

\renewcommand{\FO}{{\rm FO}}
\newcommand{\MSO}{{\rm MSO}}

\newcommand{\NFCP}{{\rm NFCP}}

\newcommand{\cover}{\mathrel{\mathlarger{\vartriangleleft}}}
\newcommand{\rloc}[2]{\ensuremath{\mathcal{B}_{#1}^{#2}}}

\newcommand{\Uu}{\mathcal{U}}

\newtheoremstyle{claim}
{5pt}
{0pt}
{\itshape}
{0pt}
{\itshape}
{.}
{ }
{$\vartriangleright$ \thmname{#1}\thmnumber{ #2}\textnormal{\thmnote{ (#3)}}}
\theoremstyle{definition}
\newtheorem{claim}[thm]{$\vartriangleright$ Claim} 
\newenvironment{claimproof}{\begin{proof}[Proof of the claim]
	
}{\end{proof}}
\makeatletter
\def\l@part{\@tocline{-1}{2pt plus2pt}{0pt}{}{\bfseries}}
\renewcommand{\@bibtitlestyle}{%
	\@xp\part\@xp*\@xp{\refname}%
}
\makeatother

	\DeclareMathOperator{\join}{+}
\DeclareMathOperator{\comp}{\oplus}
\DeclareMathOperator{\union}{\uplus}
\DeclareMathOperator{\Th}{\text{Th}}
\newcommand{\dnnc}{is self-copying\xspace}
\newcommand{\dset}{downset\xspace}
\newenvironment{absolutelynopagebreak}
{\par\nobreak\vfil\penalty0\vfilneg
	\vtop\bgroup}
{\par\xdef\tpd{\the\prevdepth}\egroup
	\prevdepth=\tpd}
\newcommand{\ERCagreement}{
	\vfill
	\begin{absolutelynopagebreak}
	\noindent\rule{2cm}{0.1pt}~\\[7pt]
	\begin{minipage}{.75\textwidth}
	\footnotesize
This paper is part of projects that have received funding from the European Research Council (ERC) under the European Union's Horizon 2020 research and innovation programme (grant agreements No 810115 -- {\sc Dynasnet}) and from the German
Research Foundation (DFG) with grant agreement
No 444419611. The first author was also supported by the project 21-10775S of the Czech Science Foundation (GA\v CR) and by the long-term strategic development financing of the Institute of Computer Science (RVO: 67985807).
	\end{minipage}$\qquad$
\begin{minipage}{.15\textwidth}
	\phantom{.}\hfill\includegraphics[height=9mm]{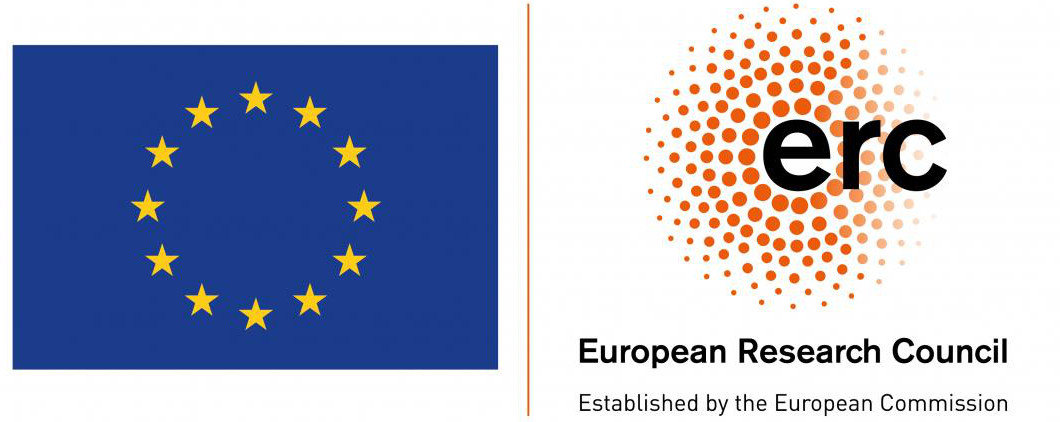}\\
	\phantom{.}\hfill\includegraphics[height=9mm]{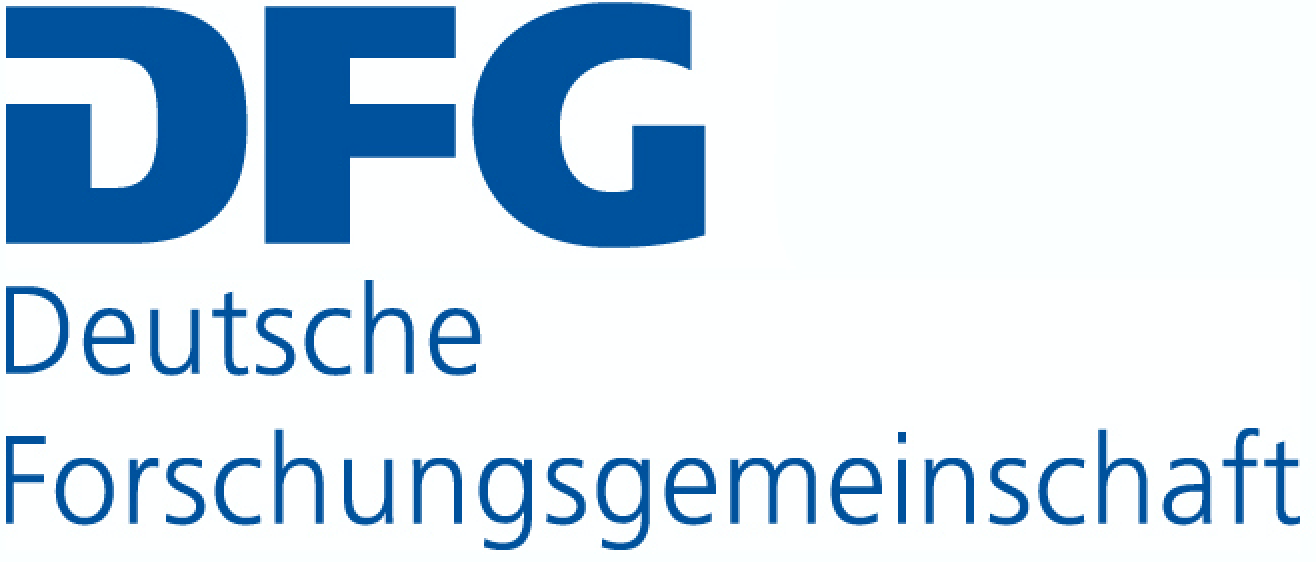}
\end{minipage}
\end{absolutelynopagebreak}
}
		
\title{On  first-order transductions 
	of  classes of graphs} 

\author[S.~Braunfeld]{Samuel Braunfeld\lmcsorcid{0000-0003-3531-9970}}[a]
\address{Computer Science Institute of Charles University (IUUK), Praha, Czech Republic, and The Czech Academy of Sciences, Institute of Computer Science, Pod Vod\'{a}renskou v\v{e}\v{z}\'{\i} 2, 182 00 Prague, Czech Republic}
\email{sbraunfeld@iuuk.mff.cuni.cz}

\author[J.~Ne\v set\v ril]{Jaroslav Ne\v set\v ril\lmcsorcid{0000-0002-5133-5586}}[b]
\address{Computer Science Institute of Charles University (IUUK), Praha, Czech Republic}
\email{nesetril@iuuk.mff.cuni.cz}

\author[P.~{Ossona de Mendez}]{Patrice {Ossona de Mendez}\lmcsorcid{0000-0003-0724-3729}}[c]
\address{Centre d'Analyse et de Math\'ematiques Sociales (CNRS, UMR 8557),
Paris, France and\\
Computer Science Institute of Charles University,
  Praha, Czech Republic}
\email{pom@ehess.fr}

\author[S.~Siebertz]{Sebastian Siebertz\lmcsorcid{0000-0002-6347-1198}}[d]
\address{University of Bremen, Bremen, Germany}
\email{siebertz@uni-bremen.de}


\begin{document}
\input{abstract-sam}
\maketitle
\ERCagreement
\setcounter{tocdepth}{1}
\newpage

\clearpage
\tableofcontents
\input{intro-sam}

\input{preliminaries}

\input{trans}
\input{examples}

\input{local}

\input{extended-dual}
\input{dense-analog}

\bibliographystyle{alphaurl}
\bibliography{ref}

\end{document}

%% file: abstract-sam.tex
\begin{abstract}
	We study various aspects of the first-order transduction quasi-order on graph classes, which provides a way of measuring the relative complexity of graph classes based on whether one can encode the other using a formula of first-order (FO) logic. In contrast with the conjectured simplicity of the transduction quasi-order for monadic second-order logic, the FO-transduction quasi-order is very complex, and many standard properties from structural graph theory and model theory naturally appear in it. We prove a local normal form for transductions among other general results and constructions, which we illustrate via several examples and via the characterizations of the transductions of some simple classes. We then turn to various aspects of the quasi-order, including the (non-)existence of minimum and maximum classes for certain properties, the strictness of the pathwidth hierarchy, the fact that the quasi-order is not a lattice, and the role of weakly sparse classes in the quasi-order.
	 	\end{abstract}

%% file: intro-sam.tex
\part*{Introduction}
A leitmotif in mathematics is to compare objects using structure-preserving maps.
In combinatorics, this has led to the study of the quasi-orders defined by homomorphisms, by the minor relation, by the topological minor relation, etc. This approach can also be applied to classes of structures (and their relationships) instead of individual structures. {\em Transductions} describe whether one class can be encoded in another using formulas from a given logic, and thus provide a way of ordering classes by their logical complexity. Perhaps surprisingly, these transduction (quasi-)orderings highlight many classes of combinatorial interest, and provide a scheme bringing together the views of model theory and structural graph theory on the complexity of graph classes.

One motivating example of this interaction is the sparsity theory for monotone graph classes (that is, classes of graphs closed under taking subgraphs), where the most general sparsity property of a monotone class is \emph{nowhere denseness} \cite{Sparsity}. This property provides a dividing line between tame and wild behavior for numerous problems; in particular, nowhere dense classes admit nearly-linear-time parameterized algorithms for many problems, and the nowhere dense/somewhere dense dichotomy matches the tractable/intractable dichotomy for the parameterized complexity of first-order model checking~\cite{grohe2014deciding}. Although originally defined combinatorially via forbidden shallow minors, nowhere denseness on monotone classes turns out to coincide with two of the best-studied properties in model theory: \emph{stability} and \emph{dependence}~\cite{Podewski1978, adler2014interpreting}. 
The notions of stability and dependence were introduced by Shelah in his seminal work on classification theory. These are dividing lines between tame and wild first-order theories defined by the presence or absence of different combinatorial configurations (See \cite[Chapters 11 and 12]{poizat2012course}).
Closely related to this, nowhere dense classes can also be characterized in terms of transductions. 
Within monotone classes, nowhere dense classes are precisely the classes that do not encode the class of all graphs by any transduction in first-order logic, and thus do not belong to the maximum class of the first-order transduction quasi-order, showing they coincide with those monotone classes that are dependent (actually {\em monadically dependent}). Additionally, (within monotone classes) the nowhere dense classes are precisely the  classes that satisfy the stronger condition of not encoding all linear orders (or equivalently, the class of all half-graphs) by any transduction in first-order logic, showing they coincide with those monotone classes that are stable (actually {\em monadically stable}). 

A recent trend in the sparsity program has focused on trying to extend nowhere denseness and related notions to the setting of hereditary classes of graphs (that is, classes closed under taking  induced subgraphs), which would yield tameness notions also capturing well-structured dense classes, such as the class of all complete graphs. The characterizations via transductions in first-order logic, and via the model-theoretic properties corresponding to these characterizations, have guided much of this work (for example, \cite{gajarsky2020new, SBE_TOCL}). 

Another motivating line for studying transductions begins with the notion of treewidth, where parallel developments involve transductions in monadic second-order logic (\MSO). This includes the initial combinatorial definition of treewidth \cite{robertson1986graph}, its algorithmic applications and Courcelle's theorem that on monotone classes bounded/unbounded treewidth (almost) corresponds to a dichotomy of fixed-parameter tractability/intractability of \MSO$_2$ model checking \cite{courcelle1990monadic}, and the characterization of bounded treewidth in terms of the incidence graph admitting an \MSO-transduction from the class of trees \cite{courcelle1992monadic}. These results were then generalized to dense classes using the property of bounded cliquewidth \cite{courcelle2000upper}, which can be characterized via admitting an \MSO-transduction from the class of trees and is conjectured  to coincide with monadic dependence in {\MSO} \cite{blumensath10}.

The quasi-order $\sqsubseteq_\MSO$ defined by {\MSO}-transductions is well studied. 
Since {\MSO} is stronger than first-order logic, it is easier for one class to encode (i.e., transduce) another in {\MSO} and so the quasi-order is simpler. In particular, it is conjectured \cite[Open Problem~9.3]{blumensath10} that this quasi-order is the chain 
$$\Ee\cover_\MSO
\Tt_2\cover_\MSO\ldots\cover_\MSO\Tt_n\cover_\MSO\ldots\sqsubseteq_\MSO\Pp\cover_\MSO\Tt\cover_\MSO~\Gg,$$
where $\cover_\MSO$ denotes the covering relation (i.e., immediate successor) of the quasi-order, $\Ee$ is the class of edgeless graphs, $\Tt_n$ is the class of forests with depth $n$ (Hence, $\Ee=\Tt_1$), $\Pp$ is the class of all paths, $\Tt$ is the class of all trees, and $\Gg$ is the class of all graphs.

In this paper, we initiate a systematic study of 
 the analogous
 quasi-order $\sqsubseteq_\FO$ defined by first-order~(\FO) transductions on infinite classes of finite graphs. This is a vastly more complex quasi-order. It clarifies the relation of various graph-theoretic properties to each other and to model-theoretic properties, and  brings new properties of  interest (such as the new notion of \emph{straightness}) to light. An overview of these graph-theoretic and model-theoretic dividing lines is presented in \Cref{fig:div_lines}, and is filled out with a skeleton of the \FO-transduction quasi-order in \Cref{fig:hierarchy}.

\begin{figure}[h]
	\centering
	\includegraphics[width=.7\textwidth]{dividing_lines}
	\caption{Prominent dividing lines and how they relate. An extended  version of this figure with some examples of classes in each region and their comparability by the transduction quasi-order $\sqsubseteq_\FO$ is given in \Cref{fig:hierarchy}.}
	\label{fig:div_lines}
\end{figure}

\begin{figure}[p]
	\hfill
	\includegraphics[width=\textwidth]{full_hierarchy-alt.pdf}
	\hfill\hfill
	\caption{Partial outline of the {\FO}-transduction quasi-order. The
		special subdivided cubic trees are those subdivisions of binary
		trees that are subgraphs of the grid.    Fat lines correspond to covers, normal
		lines correspond to strict containment $\sqsubset$, dotted lines correspond to
		containment $\sqsubseteq$ (with a possible collapse).   
	}
	\label{fig:hierarchy}
\end{figure}

This paper is split into four parts.
First, we recall some relevant notions of graph theory and model theory (\Cref{sec:prelims}). Then, we formally define first-order transductions, study their basic properties and show how they can be decomposed (\Cref{part:trans}).  Then,  in \Cref{part:examples}, we illustrate these notions and properties by several examples, some of which will be used later.
Finally, in \cref{part:applications} we give applications to the local versions of properties, to order-theoretic aspects of the transduction quasi-order, and to a recipe for generalizing properties from sparsity theory to hereditary classes.

%% file: preliminaries.tex
\part{Preliminaries}\label{sec:prelims}
In this part, we recall some basic definitions and notations from graph theory and (finite) model theory.  
We assume familiarity with first-order logic and graph theory and
refer e.g.\ to \cite{diestel2012graph,hodges1993model} for background.

\section{Graph theory}
\subsection{Basic graph theory}
\label{sec:GT}

The vertex set of a graph $G$ is
denoted as~$V(G)$ and its edge set as $E(G)$. Unless states otherwise, all graphs considered
in this paper are finite, undirected, and loopless. 
We denote by $K_t$ the complete graph on $t$
vertices and by $K_{s,t}$ the complete bipartite graph (or \emph{biclique}) with parts of size $s$ and $t$.
The \emph{clique number} 
$\omega(G)$ of a graph~$G$ is the maximum integer $t$ such that $K_t$ is a subgraph of  $G$. The \emph{maximum degree} $\Delta(G)$ of a graph $G$ is the maximum number of neighbors of vertices of $G$.  The \emph{girth} $\mathrm{girth}(G)$ of a graph $G$ is the minimum length of a cycle in $G$ (or $\infty$ if $G$ is a forest).
The {\em complement} of a
graph $G$ is the graph $\overline{G}$ with the same vertex set in
which two vertices are adjacent if they are not adjacent in $G$. The
\emph{disjoint union} of two graphs $G$ and $H$ is denoted by~$G\union H$, the disjoint union of $n$ copies of a graph $G$ is denoted by~$nG$, and the
 {\em complete join} $\overline{\overline{G}\union \overline{H}}$ of two graphs $G$ and $H$ is denoted by  $G\join H$. 
Hence, $G\join K_1$ is obtained from~$G$ by adding a 
new vertex, called an \emph{apex}, that is connected to all vertices of 
$G$. 
We write $G^k$ for the {\em $k$-th
	power} of $G$, which has the same vertex set as $G$ and two vertices
are connected if their distance in~$G$ is at most $k$. 
The {\em lexicographic product} $G\bullet H$ of two
graphs $G$ and $H$ is the graph with vertex set $V(G)\times V(H)$, in
which~$(u,v)$ is adjacent to $(u',v')$ if either $u$ is adjacent 
to~$u'$ in~$G$, or~$u=u'$ and~$v$ is adjacent to~$v'$ in $H$.  
We denote by $H\subseteq G$ the property that $H$ is a subgraph of $G$.
For a subset $S$ of vertices of a graph $G$, we denote by $G-S$ the graph obtained from $G$ by deleting the vertices in $S$. When $S$ is a singleton $\{v\}$ we use the notation $G-v$ instead of $G-\{v\}$. We further denote
by $G[A]$ the subgraph of~$G$ induced by the subset $A$ of vertices of $G$ and, more generally, 
for two (non-necessarily disjoint) subsets of vertices $A$ and $B$ of $G$, we denote by $G[A,B]$ the subgraph of $G$ with vertex set $A\cup B$, where $u\in A$ is adjacent to $v\in B$ if $u$ and $v$ are adjacent in $G$. (Note that~$G[A,A]$ is nothing but~$G[A]$.)
An \emph{interval graph} is an intersection graph of intervals of the real line.
The \newdef{pw}{pathwidth} ${\rm pw}(G)$ of a graph~$G$ is equal to one less than
the smallest clique number of an interval graph that contains $G$ as a
subgraph, that is,
\[{\rm pw}(G)=\min\{\omega(H)-1: H\text{ is an interval graph 
  with }H\supseteq G\}.\]

 Connected graphs of pathwidth $1$ are exactly the \emph{caterpillars}, i.e., trees for which removing the leaves results in a path.
A \emph{chordal graph} is a graph in which every cycle of four or more vertices has a chord.
The {\em treewidth} ${\rm tw}(G)$ of a
graph~$G$ is equal to one less than the smallest clique number of a
chordal graph that contains~$G$ as a subgraph, that is,
\[
{\rm tw}(G)=\min\{\omega(H)-1: H \text{ is a chordal graph with }H\supseteq G\}.\]

A {\em trivially perfect graph} is a comparability graph of  a forest order, i.e., a partial order where every downset is linear.
The {\em treedepth} ${\rm td}(G)$ of a graph $G$ is equal to  the smallest clique number of a
trivially perfect graph that contains~$G$ as a subgraph, that is,
\[
{\rm td}(G)=\min\{\omega(H): H \text{ is a trivially perfect graph with }H\supseteq G\}.\]

   The {\em
  bandwidth} of a graph $G$ is defined as
\[
{\rm bw}(G)=\min\{\ell:\text{there is a path } P\text{ with }P^\ell\supseteq
G \}.\]

For us, a \emph{class of graphs} (or a \emph{graph property}) means a set of unlabeled graphs, that is, of graphs considered up to isomorphism. Similarly, a \emph{class property} is a set of classes of graphs.
We extend graph operations to graph classes naturally. For example,
for a class~$\Cc$ of graphs, we denote by $\Cc\join K_1$ the 
class obtained from $\Cc$ by adding an apex to each graph of $\Cc$; also, for two classes of graphs $\Cc_1$ and $\Cc_2$, we denote by $\Cc_1+\Cc_2$ the class $\{G_1\union G_2\colon G_1\in\Cc_1, G_2\in \Cc_2\}$. 
Note that we refrain to using the notation $\Cc_1\union\Cc_2$, as this would better denote the disjoint union of the classes $\Cc_1$ and $\Cc_2$.
If $f$ is a graph parameter (such as $\omega$ or ${\rm pw}$) and $\Cc$ is a class of graphs, then 
$f(\Cc)$ denotes $\sup\{f(G)\colon G\in\Cc\}$.
A class~$\Cc$ of graphs  is \emph{monotone} if every subgraph of a graph in $\Cc$ belongs to $\Cc$; it is \emph{hereditary} if every induced subgraph of a graph in $\Cc$ belongs to $\Cc$.
The {\em monotone closure} of a class $\Cc$ is the class of all the subgraphs of graphs in $\Cc$, while the \emph{hereditary closure} of $\Cc$ is the class of all the induced subgraphs of graphs in $\Cc$.

A class of graph $\Cc$ is \emph{weakly sparse} if there exists an integer $t$ such that no graph in~$\Cc$ contains~$K_{t,t}$ as a subgraph. A class of graphs $\Cc$ is \emph{addable} if the disjoint union of any two graphs in~$\Cc$ is also in~$\Cc$, that is:
$G_1,G_2\in\Cc\Rightarrow G_1\union G_2\in\Cc$. 
Every class of graphs we consider in this paper contains the \emph{empty graph}, that is, the graph with no vertex and no edge. Thus, a class of graphs $\Cc$ is addable if and only if $\Cc=\Cc+\Cc$.

\subsection{Sparse classes}
\label{sec:sparse}
The study of the structural and algorithmic properties of classes of sparse graphs recently gained much interest. In this context, two of the authors introduced in 2006 \cite{ICGT05,Taxi_stoc06} the notion of \emph{classes with bounded expansion},  which extends the notions of classes excluding a topological minor (such as classes with bounded degree or classes excluding a minor, like planar graphs) and, soon after \cite{ECM2009,Nevsetvril2010a}, the nowhere dense vs somewhere dense dichotomy for classes of graphs. 
Although this dichotomy might look arbitrary at a first glance, it appears to be very robust, expressible in a number of different (non-obviously) equivalent ways, and to reveal both a profound structural dichotomy (roughly speaking, between branching and homogeneous structures) and a profound algorithmic complexity dichotomy (for model checking, particularly).

We recall some definitions related to graph sparsity. The interested reader is referred to the monograph \cite{Sparsity}.

A  \emph{depth-$r$ (shallow) minor} of a graph $G$ is a graph obtained from $G$ by contracting vertex-disjoints subgraphs with radius at most $r$ (i.e., containing a vertex $v$ such that all other vertices are reachable by a path of length at most $r$ from $v$) and deleting a subset of edges and vertices. For a class~$\Cc$, we denote by $\Cc\,\nabla\, r$ the class of all depth-$r$ shallow minors of graphs in $\Cc$.
Note that the classes $\Cc\,\nabla\, r$ are monotone by construction, and that $\Cc\,\nabla\,0$ is nothing but the monotone closure of $\Cc$.
A class $\Cc$ of graphs has \emph{bounded expansion}
 if there exists a function $f\colon \mathbb N\rightarrow\mathbb N$ such that, for every integer $r$, all graphs in $\Cc\,\nabla\, r$ have average degree at most~$f(r)$.

\begin{exa}
The class $\Pl$ of all planar graphs has bounded expansion as every (shallow) minor of a planar graph is planar, hence has average degree less than $6$ (by Euler's formula). 
Also, for every integer $D$, the class of all graphs $G$ with $\Delta(G)\leq D$ has bounded expansion since depth-$r$ shallow minors of graphs with maximum degree at most $D$ have average degree at most $D^{r+1}$.
\end{exa}

A class $\Cc$ of graphs is  \emph{nowhere dense}
 if there exists a function $g\colon \mathbb N\rightarrow\mathbb N$ such that, for every integer $r$, all graphs in $\Cc\,\nabla\, r$ have clique number at most $g(r)$. Equivalently, a class $\Cc$ of graphs is nowhere dense if there is no integer $r$ such that  $\Cc\,\nabla\, r$  is the class of all finite graphs.

A  \emph{depth-$r$ (shallow) topological minor} of a graph $G$ is a graph $H$ obtained from $G$ by contracting vertex-disjoints paths with length at most $2r$ and deleting a subset of edges and vertices,  that is, if a \emph{$\leq 2r$-subdivision} of $H$ is a subgraph of $G$.
For a class~$\Cc$, we denote by~$\Cc\,\widetilde\nabla\, r$ the class of all depth-$r$ shallow topological minors of graphs in $\Cc$. It follows from~\cite{POMNI} that nowhere dense classes (resp.\ bounded expansion classes) can be alternatively defined as those classes 
$\Cc$ such that for each integer $r$ the class $\Cc\,\widetilde\nabla\, r$ does not contain all finite graphs (resp.\  such that for each integer $r$ the class $\Cc\,\widetilde\nabla\, r$  has bounded average degree).

\begin{exa}
	The class of all graphs $G$ such that $\Delta(G)\leq {\rm girth}(G)$ is nowhere dense but does not have bounded expansion.
	Indeed, assume that $\Cc\,\widetilde\nabla\, r$ contains all graphs. In particular, $K_{6r+5}\in \Cc\,\widetilde\nabla\, r$. This means that there exists a graph $G\in\Cc$ such that a $\leq 2r$-subdivision of $K_{6r+5}$ is a subgraph of $G$. This implies that $G$ has maximum degree at least $6r+4$ and girth at most $3(2r+1)$, contradicting the definition of $\Cc$.
\end{exa}
\section{Model theory}
\subsection{Basic model theory}
Throughout this section, the graphs (and more general structures) we consider may be either finite or infinite, and logical notions such as formula and theory will be with respect to first-order logic. In this paper we consider either graphs or {\em $\Uu$-colored
  graphs}, that is, graphs with additional unary relations in $\Uu$
(for a set $\Uu$ of unary relation symbols). We usually denote
graphs by $G,H,\ldots$ and $\Uu$-colored graphs by
$G^+,H^+,G^*,H^*,\ldots$, but sometimes we will use $G,H,\ldots$ for
$\Uu$-colored graphs as well. We shall often use the term ``colored graph'' instead of $\Uu$-colored graph.
If $G^+$ is a $\Uu$-colored graphs and $G$ is the graph obtained from $G^+$ by forgetting the colors, we say that the graph $G$ is the \emph{reduct} of $G^+$ and that $G^+$ is a \emph{monadic expansion} of $G$.

 In formulas, the adjacency relation
will be denoted as $E(x,y)$. For each non-negative integer $r$ we can
write a formula $\delta_{\leq r}(x,y)$ such that for every graph~$G$
and all $u,v\in V(G)$ we have $G\models\delta_{\leq r}(u,v)$ if and
only if the distance between~$u$ and~$v$ in~$G$ is at most $r$. For
improved readability we write $\dist(x,y)\leq r$
for~$\delta_{\leq r}(x,y)$.  The \emph{open neighborhood} 
$N^G(v)$ of a vertex $v$ is the set of neighbors of $v$. 
For $U\subseteq V(G)$ we write~$B_r^G(U)$
for the subgraph of $G$ induced by the vertices at distance at most~$r$ from some vertex of~$U$. For the sake of
simplicity, for $v\in V(G)$ we write $B_r^G(v)$ instead of $B_r^G(\{v\})$ and, if~$G$ is clear from the context, we drop the superscript $G$. For a class~$\Cc$ and an integer $r$, we denote by
$\rloc{r}{\Cc}$ the class of all the balls of radius~$r$ of graphs in~$\Cc$: $ \rloc{r}{\Cc}=\{B_r^G(v)\mid G\in\Cc\text{ and }v\in V(G)\}$.
For a formula $\phi(x_1,\ldots, x_k)$ and a graph (or a
$\Uu$-colored graph) $G$ we define
\[\phi(G)\coloneqq \{(v_1,\ldots, v_k)\in
V(G)^k~:~G\models\phi(v_1,\ldots, v_k)\},\]
where $G\models\phi(v_1,\ldots, v_k)$ means that the formula $\phi(x_1,\dots,x_k)$ is satisfied in $G$ when interpreting the free variable $x_i$ as the vertex $v_i$. The notation $\models\psi(x_1,\dots,x_k)$ means that 
every graph satisfies $\psi(x_1,\dots,x_k)$ for every possible interpretation of $x_1,\dots,x_k$.

The \emph{(complete) theory} ${\rm Th}(G)$ of a (colored) graph $G$ (or more generally of a structure) is the set of all the sentences satisfied by $G$. We define the \emph{theory} of a class $\Cc$ of (colored) graphs as the set of all the sentences satisfied by all the graphs in $\Cc$, that is, ${\rm Th}(\Cc)=\bigcap_{G\in\Cc}{\rm Th}(G)$. A \emph{model} of~${\rm Th}(\Cc)$ is a (colored) graph that satisfies all formulas of ${\rm Th}(\Cc)$. Note that a sentence $\theta$ is satisfied by some model of ${\rm Th}(\Cc)$ if and only if it is satisfied by some graph in~$\Cc$.\footnote{
If $\theta$ is satisfied by a model $G$ of ${\rm Th}(\Cc)$, then $\neg\theta\notin{\rm Th}(\Cc)$. Hence, not every graph in $\Cc$ satisfies $\neg\theta$, that is, some graph in $\Cc$ satisfies $\theta$. Conversely, if some graph in $\Cc$ satisfies $\theta$, then $\neg\theta\notin{\rm Th}(\Cc)$. It follows that ${\rm Th}(\Cc)\cup\{\theta\}$ is consistent, thus has a (possibly infinite) model.}
\subsection{Locality}
Let $r$ be a non-negative integer.  A formula $\phi(x_1,\ldots, x_k)$
is \emph{$r$-local} if for every \mbox{(colored)} graph $G$ and all
$v_1,\ldots, v_k\in V(G)$ we have
$G\models\phi(v_1,\ldots, v_k) \iff B_r^G(\{v_1,\ldots, v_k\})\models\phi(v_1,\ldots, v_k)$.
The importance of local formulas is witnessed by the following classical result.
\begin{lem}[Gaifman's Locality Theorem~{\cite[Theorem 4.22]{Libkin04}}]\label{lem:gaifman}
	\hfill\phantom{.}\linebreak
	Every formula 
	$\phi(x_1,\ldots, x_m)$ is equivalent (for some integers $r,t$) to a Boolean
	combination of $t$-local formulas 	and so-called \emph{basic local sentences} of the form
	\[\exists y_1\ldots\exists y_k \big(\bigwedge_{1\leq i\leq k}\chi(y_i)
	\wedge\bigwedge_{1\leq i<j\leq k}\dist(y_i,y_j)>2r\big)\quad 
	\text{(where $\chi$ is $r$-local).}\] 
	Furthermore, if the quantifier-rank of $\phi$ is $q$, then 
	$r\leq 7^{q-1}$, $t\leq 7^{q-1}/2$, and $k\leq q+m$. 
\end{lem}

A stronger notion of locality for formulas is of particular interest.
An $r$-local formula $\phi(x_1,\ldots, x_k)$ is \emph{strongly
	$r$-local} if
$\models\phi(x_1,\ldots,x_k)\rightarrow \dist(x_i,x_j)\leq r$ for all
$1\leq i<j\leq k$, i.e., if $\phi$ holds for a tuple $\bar{c}$ in any graph, then the elements of $\bar{c}$ are pairwise at distance at most $r$. We will say that a formula is {\em (strongly) local} if it is (strongly) $r$-local for some $r$.

\newcommand{\loclimref}{\cite[Theorem 2]{Loclim}}
\begin{lem}[follows from \loclimref]
	\label{lem:sl}
Every local formula is equivalent to a Boolean combination of strongly local formulas.
\end{lem}

Let us give an illustration of how \Cref{lem:sl} proceeds.
\begin{exa}
	Consider the local formula $\varphi(x,y)$ asserting that $x$ and $y$ are adjacent to distinct marked vertices $u$ and $v$:
\[
\varphi(x,y):=\exists u\exists v \big(M(u)\wedge M(v)\wedge (u\neq v)\wedge E(u,x)\wedge E(v,y)\big).
\]

Now let $\psi(z)$ be the (strongly) local formula asserting that $z$ is adjacent to a marked vertex:
\[
\psi(z):=\exists w \big(M(w)\wedge E(w,z)\big).
\]

Then, the formula $\zeta(x,y):=\neg \big(\phi(x,y)\leftrightarrow (\psi(x)\wedge \psi(y))\big)$ is strongly local, since if $x$ and~$y$ are at distance more than $2$ and are each adjacent to a marked vertex, these marked vertices must be distinct.
Hence, $\phi(x,y)$ is equivalent to a Boolean combination of the strongly local formulas $\psi(x),\psi(y)$, and $\zeta(x,y)$, precisely,
$\phi(x,y)$ is equivalent to  the formula 
$\neg \zeta(x,y)\leftrightarrow(\psi(x)\wedge \psi(y))$.

\medskip
The following classical result will be useful in further strengthening these locality results.

\begin{thm}[Feferman and Vaught \cite{FV} (see also \cite{makowsky2004algorithmic})]
	\label{thm:FV}
	For every relational signature~$\mathcal L$ and every first-order formula $\phi(\bar x,\bar y)$ 	over the language of \mbox{$\mathcal L$-structures} there exists a sequence of formulas
	$(\psi_1^A(\bar x),\dots,\psi_m^A(\bar x),\psi_1^B(\bar y),\dots,\psi_m^B(\bar y))$ 	over the language of $\mathcal L$-structures and a Boolean function \mbox{$B_\phi:\{0,1\}^{2m}\rightarrow\{0,1\}$} such that, for every two disjoint $\mathcal L$-structures $\mathbf A$ and $\mathbf B$, and tuples $\bar u\in V(\mathbf{A})^{|\bar x|}$ and $\bar v\in V(\mathbf{B})^{|\bar y|}$ we have
\[
		\mathbf A \union\mathbf B\models \phi(\bar u,\bar v)\quad
		\text{iff}\quad
		B_\phi(b_1^A,\dots,b_m^A,b_1^B,\dots,b_m^B)=1,
		\]
	where $b_j^A=1$ iff $\mathbf A\models \psi_j^A(\bar u)$ and 
	$b_j^B=1$ iff $\mathbf B\models\psi_j^B(\bar v)$.
\end{thm}

Let $\phi$ be a formula and let $t$ be an integer. The \emph{$t$-localization} of $\phi$ at $x$ is the formula $\widehat\phi$  defined inductively as follows:
\begin{itemize}
	\item if $\phi$ is quantifier-free, then $\widehat\phi:=\phi$;
		\item if $\phi=\neg\psi$, then $\widehat\phi:=\neg\widehat\psi$;
	\item if $\phi=\phi_1\vee\phi_2$, then $\widehat\phi:=\widehat\phi_1\vee\widehat\phi_2$;
	\item if $\phi=\exists z\, \theta$, then $\widehat \phi:=\exists z \big({\rm dist}(x,z)\leq t\wedge\theta(z)\big)$.
\end{itemize}

It is immediate from the definition that for every first-order formula $\phi(x)$ with a single free variable $x$ and for every integer $t$, the $t$-localization $\widehat{\phi}(x)$ of $\phi(x)$ is $t$-local and that, for every $\mathcal U$-structure $G$ and every vertex $u$ of $G$ we have
\[
G\models\widehat\phi(u)\quad\iff\quad B_t^G(u)\models\phi(u).
\]

With this definition in hand, we can prove the following consequence of \Cref{thm:FV}.

\begin{cor}
\label{cor:loc_dec}
	For every integer $t$, every signature $\mathcal U$ of colored graphs,  and every $t$-local first-order formula $\phi(x,y)$, there
	exist $t$-local formulas $\zeta_1,\dots,\zeta_m$ with a single free variable and a Boolean combination $\zeta(x,y)$ of  $\zeta_1(x),\dots,\zeta_m(x),\zeta_1(y),\dots,\zeta_m(y)$ such that
	\[
	\models ({\rm dist}(x,y)>2t)\rightarrow (\phi(x,y)\leftrightarrow\zeta(x,y)).
	\]
\end{cor}

\begin{proof}
	According to \Cref{thm:FV} there exists a sequence of formulas
	$(\psi_1^0(x),\dots,\psi_m^0(x)$, $\psi_1^1(y),\dots,\psi_m^1(y))$ and a Boolean function $B_\phi:\{0,1\}^{2m}\rightarrow\{0,1\}$ such that, for every two disjoint $\mathcal U$-colored graphs~$G_0$ and $G_1$, and $u\in V(G_0)$ and $v\in V(G)_1)$, we have
	\[
	G_0 \union G_1\models \phi(u,v)\quad
	\text{iff}\quad
	B_\phi(b_1^0,\dots,b_m^0,b_1^1,\dots,b_m^1)=1,
	\]
	where $b_j^0=1$ iff $G_0\models \psi_j^0(u)$ and 
	$b_j^1=1$ iff $G_1\models\psi_j^1(v)$.
	
	Let $G$ be a $\mathcal U$-colored graph and let $u,v$ be vertices of $G$ with ${\rm dist}(u,v)>2t$.
	Then,
	\begin{align*}
		G\models \phi(u,v)\quad&\iff\quad B_t^{G}(u)\union B_t^{G}(v)\models \phi(u,v)&\text{(by  $t$-locality)}\\
		&\iff\quad
			B_\phi(b_1^0,\dots,b_m^0,b_1^1,\dots,b_m^1)=1,&
	\end{align*}
	where $b_j^0=1$ iff $B_t^{G}(u)\models \psi_j^0(u)$ and 
	$b_j^1=1$ iff $B_t^{G}(v)\models\psi_j^1(v)$.
	Now let $\zeta_j(x)$ be the $t$-localization at $x$ of the formula $\psi_j^0(x)$. Then 
	\[
	G\models\zeta_j(u)\iff B_t^G(u)\models\psi_j^0\iff b_j^0=1.
	\]

	Similarly,  if $\psi_j^1$ involves only $B$ and $\zeta_{m+j}(y)$ is the $t$-localization at $y$ of the formula obtained from $\psi_j^0$ by replacing all the occurrences of $b$ by the free variable $y$ we have
	\[
	G\models\zeta_j(v)\iff B_t^G(v)\models\psi_j^1\iff b_j^1=1.
	\]

	Thus, we can translate the Boolean combination $B_\phi$ of $b_1^0,\dots,b_m^0,b_1^1,\dots,b_m^1$ into a Boolean combination $\zeta(x,y)$ of $\zeta_1(x),\dots,\zeta_m(x),\zeta_{m+1}(y),\dots,\zeta_{2m}(y)$, and we have
	\[
	G\models \phi(u,v)\quad\iff\quad G\models\zeta(u,v).
	\]	

 As this holds for every $\mathcal U$-structure $G$ and all vertices $u,v$ of $G$ with ${\rm dist}(u,v)>2t$, and as the formula $\zeta$ does not depend on the particular choices of  $G$, $u$, and $v$, the corollary follows.
\end{proof}

\end{exa}
\subsection{Dividing lines}\label{sec:dividing-lines}
Stability theory in mathematical logic emerged in the 1960s and 1970s, focusing on classifying mathematical structures based on the complexity of their definable sets and types. 
Building on a foundational result of Morley,  Shelah developed a comprehensive classification theory. 
He introduced the notion of stability and dependence as key dividing lines between well-behaved and wild theories~\cite{shelah1990classification}. 

\begin{defi}
	Given a structure $M$ and an integer $n$, a partitioned formula $\phi(\xbar;\bar y)$ has the {\em $n$-independence property} in $M$ if there exist tuples $(\bar a_i\in M^{|\bar x|} : i\in [n]$) and 
	$(\bar b_J\in M^{|\bar y|} : J\subseteq [n]$), such that for all $i, J$
	\begin{equation}
		M\models\phi(\bar a_i;\bar b_J)\quad\iff\quad i\in J.
	\end{equation}

	The structure $M$ is {\em dependent} (or NIP) if for every partitioned formula $\phi$ there is an integer~$n_\phi$ such that $\phi$ does not have the $n_\phi$-independence property in $M$.
\end{defi}

\begin{defi}
	Given a structure $M$ and an integer $n$, a partitioned formula $\phi(\xbar;\bar y)$ has the \emph{$n$-order property} in $M$ if there exist tuples $(\bar a_i\in M^{|\bar x|} : i\in [n]$) and 
	 $(\bar b_j\in M^{|\bar y|} : j\in [n]$), such that for all $i, j$
	 \begin{equation}
	 	M\models\phi(\bar a_i;\bar b_j)\quad\iff\quad i\leq j.
	 \end{equation}

	 The structure $M$ is {\em stable} if for every partitioned formula $\phi$ there is an integer $n_\phi$ such that $\phi$ does not have the $n_\phi$-order property in $M$.
\end{defi}

\begin{defi} \label{def:halfgraph}
	The {\em half-graph} is a bipartite graph $H_n = (A, B; E)$ where $A = \{a_i : i \in [n]\}$, $B = \{b_i : i \in [n] \}$, and $E(a_i, b_j)$ holds if and only if $i \leq j$.
	\end{defi}

Thus, a partitioned formula $\phi(\bar x; \bar y)$ has the $n$-order property if and only if there are tuples on which $\phi$ defines a half-graph (only considering edges between tuples in different parts of the partition).

The derived notions of monadic dependence and monadic stability will be introduced and discussed in \Cref{sec:monadic-dependence-and-monadic-stability} in the context of \FO-transductions.

Nowhere dense classes define a natural dividing line on monotone classes of graphs, which is consistent with model theory, as shown by the next result.

\begin{thmC}[\cite{adler2014interpreting}]
	\label{thm:adler}
	Let $\Cc$ be a monotone class of graphs. The following are equivalent:
	\begin{itemize}
		\item $\Cc$ is nowhere dense,
		\item $\Cc$ is stable
		\item $\Cc$ is monadically stable
		\item $\Cc$ is dependent,
		\item $\Cc$ is monadically dependent.
	\end{itemize}
\end{thmC}
(Note that it follows from \cite{DVORAK2018143} that the assumption that $\Cc$ is monotone can be replaced by the assumption that $\Cc$ is weakly sparse.)

In this paper, we consider classes of finite (colored) graphs. We consequently adapt these notions as follows.
A hereditary class $\Cc$ of (colored) graphs is \emph{dependent} (resp.\ \emph{stable}) if every model of $\Th(\Cc)$ is dependent (resp.\ stable).
Thus, a hereditary class $\Cc$ is 
dependent (resp.\ stable) if, for every partitioned formula $\phi(\bar x;\bar y)$, there is some integer $n_\phi$ such that~$\phi$ does not have the $n$-independence property (resp.\ the $n_\phi$-order property) in every graph of~$\Cc$.

%% file: trans.tex
\part{Transductions}
\label{part:trans}

In this part, we introduce transductions and various helpful constructions and results concerning them. 
FO-transductions is a first-order counterpart of the MSO-transductions introduced by Courcelle and Engelfriet \cite{courcelle2012graph}. This notion implicitely appears in \cite{baldwin1985second}, where it is used to characterize monadic stability and monadic dependence (See Theorem~\ref{thm:BS}).

After reading the definition in the following section, the reader may wish to skip ahead to \cref{part:examples} for concrete examples and then refer back to this part as needed.
\section{Definition} \label{sec:transDef}
For a positive integer $k$, the \emph{$k$-copy operation}
$\mathsf{C}_k$ maps a graph $G$ to the graph~$\mathsf{C}_k(G)$
consisting of $k$ copies of $G$ where the copies of each vertex are
made adjacent (that is, the copies of each vertex induce a clique and
there are no other edges between the copies of~$G$). The copies of a given vertex are called \emph{clones}. 
Note that for $k=1$, $\mathsf C_k$ maps each graph $G$ to itself, thus, $\mathsf C_1$ is the identity mapping.

For a set ${\mathcal U}$ of unary relations, the \emph{coloring operation} $\Gamma_{\mathcal U}$ maps a graph $G$ to the set~$\Gamma_{\mathcal U}(G)$ of all its ${\mathcal U}$-expansions.

A \emph{simple  interpretation} $\mathsf I$ of graphs in \mbox{$\Uu$-colored}
graphs is a pair $(\nu(x), \eta(x,y))$ consisting of two formulas (in
the first-order language of $\Uu$-colored graphs), where $\eta$ is symmetric
and anti-reflexive (i.e.,
\mbox{$\models \eta(x,y)\leftrightarrow\eta(y,x)$} and
\mbox{$\models \eta(x,y)\rightarrow\neg(x=y)$}). If $G^+$ is a
$\Uu$-colored graph, then $H=\mathsf I(G^+)$ is the graph with
vertex set $V(H)=\nu(G^+)$ and edge set $E(H)=\eta(G)\cap \nu(G)^2$.
For a set $\Gamma$ of $\Uu$-colored graphs we let $\mathsf{I}(\Gamma)=\bigcup_{G^+\in \Gamma}\mathsf{I}(G^+)$. 

A \emph{transduction} $\mathsf T$ is the composition
$\mathsf I\circ\Gamma_{\mathcal U}\circ \mathsf C_k$ of a copy operation $\mathsf C_k$, a coloring operation~$\Gamma_{\mathcal U}$, and a simple interpretation $\mathsf I$ of graphs in $\Uu$-colored graphs.
In other words, for every graph $G$ we have
$\mathsf T(G)=\{\mathsf I(G^+): G^+\in\Gamma_{\mathcal U}(\mathsf C_k(G))\}$, see \Cref{fig:trans}. We remark that in $\mathsf{T}(G)$ we distinguish between isomorphic graphs because we will frequently use the identification of the vertex set of the image and a subset of the vertex set of the source. Later, when applying transductions on the class level, we will not distinguish isomorphic graphs. 

\begin{figure}[p]
	\begin{center}
		\includegraphics[width=\textwidth]{trans-fig}
		\vspace{5mm}
		
		\includegraphics[width=\textwidth]{trans-fig2}
	\end{center}
	\caption{Two examples of transductions. The top transduction $\mathsf T=\mathsf I\circ\Gamma\circ\mathsf C_2$ includes copying. The coloring transduction $\Gamma$ maps a graph to the set of all its expansions by a single unary relation (shown in red); the interpretation $\mathsf I$ complements the subset marked by the added unary relation (i.e., the subset of all red vertices) and removes the red marks. Here we have not illustrated the final step of identifying isomorphic graphs. \\
		The bottom transduction $\mathsf H$ is non-copying. Its associated interpretation extracts the subgraph induced by the red vertices. When we apply a transduction to a class, we consider the graphs up to isomorphism.}
	\label{fig:trans}
\end{figure}

A transduction $\mathsf T$ is \emph{non-copying} if it is the composition of a coloring operation and a simple interpretation, that is, if it can  be written as 
$\mathsf I\circ\Gamma_{\mathcal U}\circ \mathsf C_1$ ($=\mathsf I\circ\Gamma_{\mathcal U}$). We denote by $\mathsf{Id}$ the \emph{identity transduction}, which is such that $\mathsf{Id}(\mathbf A)=\{\mathbf A\}$ for every structure $\mathbf A$.

Let us give some easy examples of transductions:
\begin{exa}
	\label{ex:hered}
The \emph{hereditary transduction} $\mathsf H$ is the non-copying transduction extracting induced subgraphs that is defined by the interpretation $(\nu,\eta)$ with $\nu(x):=M(x)$ and $\eta(x,y):=E(x,y)$, for a fixed unary relation $M$. So, for every graph $G$, $\mathsf H(G)$ is the set of all induced subgraphs of $G$. We fix the name $\mathsf H$ for the hereditary transduction for the rest of the paper. 

The \emph{pair hereditary transduction} $\mathsf H_2$ is the non-copying transduction extracting the subgraphs $G[A,B]$ for subsets of vertices $A$ and $B$. Using new colors $A$ and $B$, it is defined by the interpretation $(\nu,\eta)$ with $\nu(x):=A(x)\vee B(x)$ and
$\eta(x):=E(x,y)\wedge 
(A(x)\vee A(y))\wedge (B(x)\vee B(y))$. So, for every graph $G$, $\mathsf H_2(G)$ is the set of all the graphs $G[A,B]$, where $A$ and $B$ are subsets of vertices of $G$. 
\end{exa}

We call a transduction $\mathsf T$ {\em immersive} if it is
non-copying and the formulas in the interpretation associated to
$\mathsf T$ are strongly local.

A {\em subset complementation} transduction is defined by the
quantifier-free interpretation on a ${\mathcal U}$-expansion (with
${\mathcal U}=\{M\}$) by
$\eta(x,y):=(x\neq y)\wedge \neg \bigl(E(x,y)\leftrightarrow(M(x)\wedge M(y)\bigr)$.
In other words, the subset complementation transduction complements
the adjacency inside the subset of the vertex set defined by~$M$. We
denote by $\comp M$ the subset complementation defined by the unary
relation~$M$. A {\em perturbation} is a composition of (a bounded
number of) subset complementations.

For a class of graphs~$\Dd$ and a transduction $\mathsf T$ we define $\mathsf T(\Dd)$ as the set of all the graphs in 
$\bigcup_{G\in\Dd}\mathsf T(G)$ up to isomorphism, and 
we say that a class~$\Cc$ is a  {\em $\mathsf T$-transduction} of $\Dd$ if $\Cc\subseteq\mathsf T(\Dd)$. 
We also say that $\mathsf{T}$ \emph{encodes} $\Cc$ in $\Dd$.

For instance, for a class $\Cc$, the class $\mathsf H(\Cc)$ is the \emph{hereditary closure} of $\Cc$, that is, the class of all the induced subgraphs of graphs in $\Cc$.
A class~$\Cc$ of graphs is a \emph{(non-copying) transduction} of a class $\Dd$ of
graphs if it is a $\mathsf T$-transduction of~$\Dd$ for some (non-copying) transduction $\mathsf T$.

We denote by $\Cc\sqsubseteq_\FO\Dd$ (resp.\ $\Cc\sqsubseteq^{\circ}_\FO\Dd$) the property that the class $\Cc$ is an FO transduction (resp.\ a non-copying FO transduction) of the class $\Dd$. Note that if $\Cc\subseteq \Dd$, then $\Cc\sqsubseteq_\FO^{\circ} \Dd$. 
It is easily checked that the composition of two (non-copying) transductions is a (non-copying) transduction (see, for instance \cite[Theorem 7.14]{courcelle2012graph}). Thus, the relations
$\Cc\sqsubseteq_\FO\Dd$ and 
$\Cc\sqsubseteq^{\circ}_\FO\Dd$ are quasi-orders
on classes of graphs.
Intuitively, if $\Cc\sqsubseteq_\FO \Dd$, then~$\Cc$ is at most as complex as~$\Dd$.
If $\Cc\sqsubseteq_\FO \Dd$ and $\Dd\sqsubseteq_\FO \Cc$, we will say that $\Cc$ and $\Dd$ are {\em transduction-equivalent}; this will be denoted by $\equiv_\FO$, and $\equiv_\FO^\circ$ is defined similarly.

\section{Transduction subsumption} \label{sec:subsump}

We say that a transduction $\mathsf T'$ \emph{subsumes} a transduction~$\mathsf T$ if for every graph $G$ we have the inclusion of graph classes
$\mathsf T'(G)\supseteq\mathsf T(G)$. We denote by
$\mathsf T'\geq\mathsf T$ the property that $\mathsf T'$ subsumes
$\mathsf T$.
We furthermore define a notion of transduction subsumption relative to a sequence of relevant classes.
Let $\Cc_0,\dots,\Cc_k$ be classes, and let $\mathsf T_1',\dots,\mathsf T_k',\mathsf T$ be transductions with $\Cc_i\subseteq\mathsf T_i'(\Cc_{i-1})$ for $i\in [k]$ and
$\Cc_k\subseteq\mathsf T(\Cc_{0})$. Note that $\mathsf T_i'$ is a transduction used to construct the graphs in $\mathscr C_i$.
We say that $\mathsf T_k'\circ\dots\circ\mathsf T_1'$   \emph{subsumes} $\mathsf T$
\emph{on the chain} $(\Cc_0,\dots,\Cc_k)$ if, for every $G_0\in\Cc_0$ and every 
$G_k\in\mathsf T(G_0)\cap\Cc_k$, there exist \mbox{$G_1\in\Cc_1,\dots,G_{k-1}\in\Cc_{k-1}$}  with 
$G_i\in \mathsf T_i'(G_{i-1})\cap\Cc_i$ for all $i\in [k]$.
Equivalently (by considering all $G_k\in\mathsf T(G_0)$)  we have that $\mathsf T_k'\circ\dots\circ\mathsf T_1'$  subsumes $\mathsf T$
on the chain $(\Cc_0,\dots,\Cc_k)$ if, for every $G_0\in\Cc_0$ we have the inclusion
\[
\mathsf T(G_0)\cap \mathscr C_k\subseteq \mathsf T_k'(\mathsf T_{k-1}'(\dots (\mathsf T_2'(\mathsf T_1'(G_0)\cap \mathscr C_1)\cap \mathscr C_2) \dots)\cap\mathscr C_{k-1})\cap \mathscr C_k.\]
This situation is graphically depicted in \Cref{fig:trsub}. 
\begin{figure}[h!]
\centering
\includegraphics[width=\textwidth]{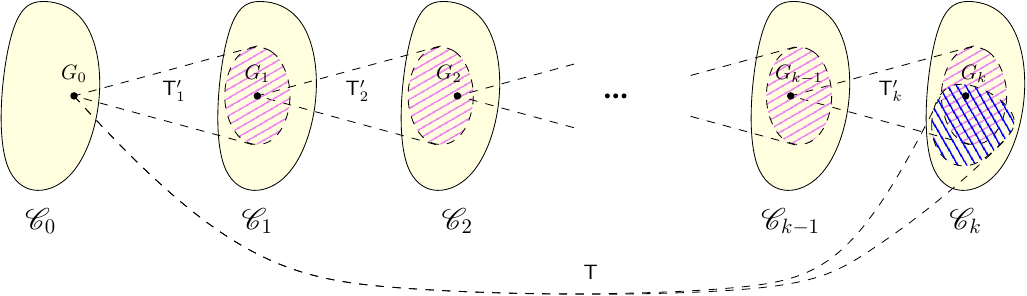}
\caption{Transduction subsumption. For every $G_0\in\Cc_0$ and every $G_k\in \mathsf T(G_0)\cap \Cc_k$, there exist $G_1\in \Cc_1,\dots, G_{k-1}\in\Cc_{k-1}$ with $G_i\in \mathsf T_i'(G_{i-1})$ for all $i=1,\dots,k$.}
\label{fig:trsub}
\end{figure}

As a particular case, a transduction $\mathsf T'$ subsumes a transduction $\mathsf T$ if it subsumes it on the chain
$(\Gg,\mathsf T(\Gg))$.

A \emph{transduction pairing} of two classes~$\mathscr C$
and~$\mathscr D$ is a pair $(\mathsf D,\mathsf C)$ of  transductions,
such that 
\[
\forall\mathbf A\in\mathscr C\quad\exists \mathbf B\in\mathsf D(\mathbf A)\cap\mathscr D\quad \mathbf A\in\mathsf C(\mathbf B)\quad\text{and}\quad
\forall\mathbf B\in\mathscr D\quad\exists \mathbf A\in\mathsf C(\mathbf B)\cap\mathscr C\quad \mathbf B\in\mathsf D(\mathbf A).
\]

In other words,  $(\mathsf D,\mathsf C)$ is a transduction pairing of $\mathscr C$
and~$\mathscr D$ if
$\mathsf C\circ \mathsf D$ subsumes $\mathsf{Id}$ on the chain $(\Cc,\Dd,\Cc)$ and $\mathsf D\circ \mathsf C$ subsumes $\mathsf{Id}$ on the chain $(\Dd,\Cc,\Dd)$.
Note that transduction pairing is a strong form of transduction equivalence.

\begin{exa}
	Let $\mathsf C_2$ be the $2$-copy transduction, let $\mathsf H$ be the hereditary closure transduction,
	let $\Cc$ be a class of graphs, and 
	let $\mathsf C_2(\Cc)=\{\mathsf C_2(G)\colon G\in\Cc\}$.
	
	Then, 
	$(\mathsf C_2,\mathsf H)$ is a transduction pairing of $\Cc$ and
	$\mathsf C_2(\Cc)$.
\end{exa}

A transduction pairing $(\mathsf D,\mathsf C)$ of two classes $\Cc$ and $\Dd$ is \emph{non-copying} if none of $\mathsf D$ and $\mathsf C$ is copying. 

\section{Local normal forms}
\label{sec:norm}
We now establish a normal form for first-order transductions that
captures the local character of first-order logic, and we further
study the properties of immersive transductions. The normal form is
based on Gaifman's Locality Theorem and uses only a copy operation, strongly local
formulas, and perturbations. 
This normal form will be one of the main tools to 
establish results in the paper. Based on Gaifman's Locality Theorem we first show that we may assume that all formulas in transductions are local. 

\begin{lem}
	\label{lem:normal0}
	Every non-copying transduction $\mathsf T$ is subsumed by a non-copying transduction~$\mathsf T'$ with associated interpretation $\mathsf I'=(M(x),\eta'(x,y))$, where $M$ is a unary relation and $\eta'$ is a local formula.
\end{lem}
\begin{proof}
	Let $\mathsf T=\mathsf I_{\mathsf T}\circ\Gamma_{{\mathcal U}_{\mathsf T}}$ be a non-copying transduction.
	We define ${\mathcal U}_{\mathsf T'}$ as the disjoint union of ${\mathcal U}_{\mathsf T}$, the singleton $\{M\}$ (where $M$ is a unary relation used to define the domain), and  a set ${\mathcal U}'=\{T_i\mid 1\leq i\leq n_1\}$ of unary relations used to define global first-order properties, for some integer $n_1$ we shall specify later.

	Let $q$ be the quantifier rank of $\eta(x,y)$.  According to
	\Cref{lem:gaifman}, $\eta$ is logically equivalent to a formula in 
	Gaifman normal form, that is, to a Boolean
	combination of $t$-local formulas and
	basic local sentences~$\theta_1,\dots,\theta_{n_1}$, for some $t<7^q$.  To each
	$\theta_i$ we associate a unary predicate $T_i\in {\mathcal U}'$. 
	The idea is to mark every vertex satisfying $\nu(x)$ by $M$  and to mark all the vertices
	by $T_i$ in graphs satisfying $\theta_i$. 
	We consider the
	formula ${\eta}'(x,y)$ obtained from the Gaifman normal form of $\eta(x,y)$ 
	by replacing the sentence $\theta_i$ by the atomic formula $T_i(x)$ (that is: by testing $T_i$ on the vertex corresponding to the first free variable of ${\eta}'(x,y)$).
	Note that, whatever the interpretation of $x$, $T_i(x)$ is equivalent to $\theta_i$ if the marking $T_i$ has been done as described above. However,  considering $T_i(x)$ instead of $(\exists z)\, T_i(z)$ allows  
	ensuring that ${\eta}'$ is $t$-local. 
\end{proof}

We will next separate the local formulas into strongly local formulas followed by perturbations. For this, we first give a practical characterization of perturbations.
For any $k\in \N$, and $\mathbf x$ and~$\mathbf y$ in $\mathbb{F}_2^k$, we denote by  $\langle\mathbf x,\mathbf y\rangle$ the inner product of  $\mathbf x$ and $\mathbf y$.
(Thus, $\langle\mathbf x,\mathbf y\rangle=1$  if $\mathbf x$ and $\mathbf y$ have an odd number of common coordinates equal to $1$, and $\langle\mathbf x,\mathbf y\rangle=0$ otherwise.) In the following lemma in a partition we allow parts to be empty.

\begin{lem}
	\label{lem:perturb}
	For every perturbation $\mathsf P=\oplus Z_1\oplus\dots\oplus Z_k$ and for every graph $G$, $H\in{\mathsf P}(G)$ if and only if there exists a partition $(V_\mathbf x)_{\mathbf x\in \mathbb{F}_2^k}$ of $V(G)$ such that $H$ is obtained from $G$ by flipping the edges $uv$ where $u\in V_{\mathbf x}, v\in V_{\mathbf y}$, and $\langle\mathbf x,\mathbf y\rangle=1$.
\end{lem}
\begin{proof}
	Let $G^+$ be an expansion of $G$ by relations $Z_1,\dots,Z_k$ and let $H$ be obtained from~$G$ by complementing the subsets $Z_1(G^+),\dots,Z_k(G^+)$.
	For $\mathbf x\in \mathbb{F}_2^k$ we define $V_{\mathbf x}$ as the subset of all vertices $v$ such that $v\in Z_i(G^+)$ if and only if $\mathbf x_i=1$. For $u\in V_{\mathbf x}$ and $v\in V_{\mathbf y}$, the number of times an edge $uv$ of $G$ is flipped when applying the subset complementations $Z_1,\dots,Z_k$ is the number of common coordinates of $\mathbf x$ and $\mathbf y$ that are equal to $1$. It follows that $H$ is obtained from $G$ by flipping the edges $uv$ where $u\in V_{\mathbf x}, v\in V_{\mathbf y}$, and $\langle\mathbf x,\mathbf y\rangle=1$, 
	see \Cref{fig:perturb}.
	
	Conversely, assume $H$ is obtained from $G$ by flipping the edges $uv$ where $u\in V_{\mathbf x}, v\in V_{\mathbf y}$, and $\langle\mathbf x,\mathbf y\rangle=1$, where $(V_\mathbf x)_{\mathbf x\in \mathbb{F}_2^k}$ is a partition of $V(G)$. Define the expansion $G^+$ of $G$ by relations $Z_1,\dots,Z_k$ where $Z_i(G^+)$ is the union of the sets $V_{\mathbf x}$ with $\mathbf x_i=1$.
	Then, it is immediate that $H$ is obtained from $G^+$ by complementing the subsets $Z_1(G^+),\dots,Z_k(G^+)$. In particular, $H\in\mathsf P(G)$.
\end{proof}
\begin{figure}[h]
	\centering
	\includegraphics[width=\textwidth]{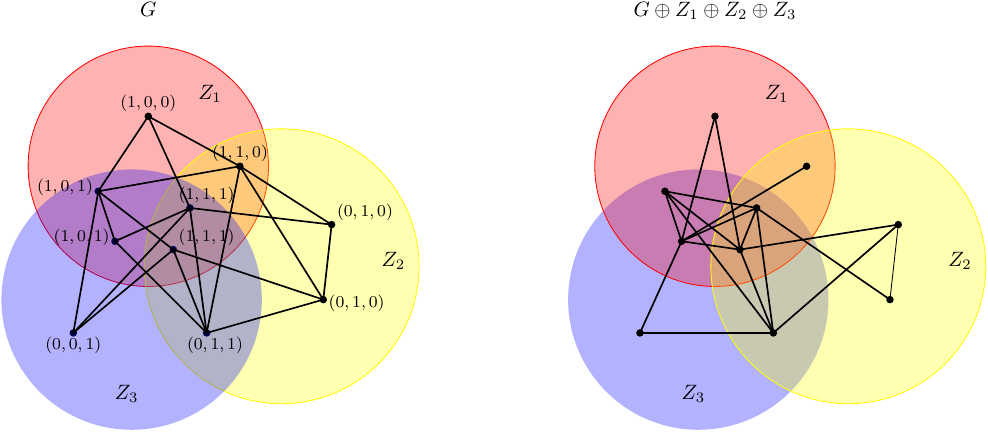}
	\caption{Illustration of the proof of \Cref{lem:perturb}}
	\label{fig:perturb}
\end{figure}

Finally, we show the main theorem of this section, which boils down to subsuming local formulas into strongly local formulas and perturbations. 

\pagebreak
\begin{thm}
	\label{thm:normal}
	Every non-copying transduction $\mathsf T$ is subsumed by the composition of an immersive transduction~$\mathsf T_{\rm imm}$ and a perturbation~$\mathsf P$, that is, $\mathsf T\leq \mathsf P\circ\mathsf T_{\rm imm}$.
	
	Consequently, every transduction $\mathsf T$ is subsumed by the composition of a
	copying operation~$\mathsf C$, an immersive transduction $\mathsf T_{\rm imm}$ and a perturbation~$\mathsf P$, that is, $\mathsf T\leq \mathsf P\circ\mathsf T_{\rm imm}\circ\mathsf C$.
\end{thm}
\begin{proof}
	Let $\mathsf T$
	be a non-copying transduction.
	According to \Cref{lem:normal0}, the transduction $\mathsf T$ is subsumed by a transduction with associated interpretation $(V(x),\widetilde{\eta}(x,y))$, where $V$ is a unary relation and $\widetilde{\eta}(x,y)$ is a $t$-local formula, for some $t\in\mathbb N$.
	
	Since every Boolean combination can be written in disjunctive normal form, according to \Cref{cor:loc_dec},  under the assumption ${\rm dist}(x,y)>2t$ the
	formula $\widetilde{\eta}$ is equivalent to a formula~$\widetilde{\phi}(x,y)$ of the form
	$\bigvee_{(i,j)\in\mathcal F}\zeta_i(x)\wedge\zeta_j(y)$, where
	$\mathcal F\subseteq [n_2]\times[n_2]$ for some integer~$n_2$ and the
	formulas $\zeta_i$ ($1\leq i\leq n_2$) are $t$-local. 
	By considering appropriate 
	Boolean combinations (or, for those familiar with model theory, by assuming that the $\zeta_i$ define local types)
	we may assume that \mbox{$\models \forall x \bigwedge_{i\neq j} \neg(\zeta_i(x)\wedge
		\zeta_j(x))$}, that is, every element of a graph satisfies at most one of the $\zeta_i$. 
	Note also that by this assumption $\mathcal F$ is symmetric, as $\eta$ (hence $\widetilde{\eta}$ and $\widetilde{\phi}$) is symmetric.
	
	We define 
	\mbox{$\psi(x,y):=\neg(\widetilde{\eta}(x,y)\leftrightarrow\widetilde{\phi}(x,y))\wedge ({\rm dist}(x,y)\leq 2t)$}, which is $2t$-strongly local, 
	and we define~$\mathsf I_{\mathsf T_\psi}$ as the interpretation of graphs in ${\mathcal U}_{\mathsf T_\psi}$-structures by using the same definitions as in $\mathsf I_{\mathsf T}$ for the domain, then defining the adjacency relation by $\psi(x,y)$.  
	To each formula~$\zeta_i$ we associate a unary 
	predicate $Z_i\in {\mathcal U}_{\mathsf P}$. We define the perturbation~$\mathsf{P}$
	as the sequence of subset complementations $\comp Z_i$ 
	(for~$(i,i)\in\mathcal F$) and of \mbox{$\comp Z_i\comp Z_j\comp (Z_i\cup Z_j)$} 
	(for~$(i,j)\in\mathcal F$ and~$i<j$). Denote by $\phi(x,y)$ the
	formula defining the edges in the interpretation~$I_\mathsf{P}$. Note that
	when the $Z_i$ are pairwise disjoint, then $\mathsf{P}$ complements
	exactly the edges of $Z_i$ or between~$Z_i$ and~$Z_j$, respectively. 
	The operation $\comp (Z_i\cup Z_j)$ complements all edges between
	$Z_i$ and $Z_j$, but also inside $Z_i$ and $Z_j$, which is undone by
	$\comp Z_j$ and~$\comp Z_i$.

	Now assume that a graph $H$ is a $\mathsf T$-transduction of a graph $G$, and let $G^+$ be a \mbox{${\mathcal U}_{\mathsf T}$-expansion} of $G$ such that $H=\mathsf I_{\mathsf T}(G^+)$.
	We define the ${\mathcal U}_\psi$-expansion $G^*$ of $G^+$ (which is thus a ${\mathcal U}_{\mathsf T_{\psi}}$-expansion of $G$)
	by defining, for each $i\in[n_1]$, $T_i(G^*)=V(G)$ if
	$G^+\models \theta_i$ and $T_i(G^*)=\emptyset$ otherwise.
	Let $K=\mathsf I_{\mathsf T_\psi}(G^*)$. We define the ${\mathcal U}_{\mathsf P}$-expansion $K^+$ of $K$ by defining, for each $j\in [n_2]$, 
	$Z_j(K^+)=\zeta_j(G^+)$. By the assumption that $\models \forall x \bigwedge_{i\neq j} \neg(\zeta_i(x)\wedge
	\zeta_j(x))$ the~$Z_j$ are pairwise disjoint. 
	Now, when $\dist(x,y)>2t$ there is no edge between~$x$ and~$y$ in~$K$, 
	hence $\phi$ on $K^+$ is equivalent to~$\widetilde{\phi}$ on $G^*$, which 
	in turn in this case is equivalent to~$\widetilde{\eta}(x,y)$ on~$G^*$. 
	On the other hand, when $\dist(x,y)\leq 2t$, then the perturbation 
	is applied to the edges defined by $\neg(\widetilde{\eta}(x,y)\leftrightarrow\widetilde{\phi}(x,y)$, which yields exactly the edges defined by $\widetilde{\eta}$
	on~$G^*$. Thus we have $\eta(G^+)=\widetilde{\eta}(G^*)=\phi(K^+)$, 
	hence    $\mathsf I_{\mathsf P}(K^+)=H$.

	It follows that the transduction $\mathsf T$ is subsumed by the composition
	of the immersive transduction $\mathsf T_\psi$ and a sequence of 
	subset complementations, the perturbation $\mathsf P$.
\end{proof}

\begin{cor}
	\label{cor:norm}
	For every immersive transduction $\mathsf T$ and every perturbation $\mathsf P$, there exist
	an immersive transduction $\mathsf T'$ and a perturbation $\mathsf P'$, such that 
	$\mathsf P'\circ\mathsf T'$ subsumes $\mathsf T\circ\mathsf P$.
\end{cor}

\Cref{thm:normal} can be strengthened if we only require a relative subsumption and if the target class $\Dd$ is weakly sparse.

\begin{thm}
	\label{thm:rel-normal}
	Let  $\mathsf T$ be a non-copying transduction and let a  weakly sparse  class $\Dd$ be a $\mathsf T$-transduction  of a class $\Cc$.
	
	Then, there exist an immersive transduction $\mathsf T_{\rm imm}$, a perturbation $\mathsf P$, and a weakly sparse (hereditary)  class $\Ss$ such that 
	$\mathsf P\circ\mathsf T_{\rm imm}$ subsumes $\mathsf T$ on the chain $(\Cc,\Ss,\Dd)$.
\end{thm}
\begin{proof}
	As the class $\Dd$ is weakly sparse, there exists an integer $t$ such that $K_{t,t}$ is not a subgraph of any graph in $\Dd$.
	
	By \Cref{thm:normal}, the transduction $\mathsf T$ is subsumed by the composition of an immersive transduction $\mathsf T_0$ and a perturbation $\mathsf P_0$.
	Thus, $\Dd\subseteq \mathsf P_0\circ\mathsf T_0(\Cc)$, which rewrites as the property that for every $H\in\Dd$ there exists $G_H\in\Cc$ and $K_H\in\mathsf T_0(G_H)$ such that
	$H\in \mathsf P_0(K_H)$. 
	The problem is that the class $\mathscr K:=\{K_H\colon H\in\Dd\}$ needs not to be weakly sparse.
	
	However, we shall prove that there exists a perturbation $\mathsf P_1$ and an immersive transduction~$\mathsf T_1$ such that, for every $H\in\Dd$, we can choose 
	$K_H'\in\mathsf P_1(K_H)\cap \mathsf T_1(K_H)$ in such a way that the class 
	$\Ss:=\{K_H'\colon H\in\Dd\}$ is weakly sparse.
	Note that $K_H'\in\mathsf P_1(K_H)$ is equivalent to $K_H\in\mathsf P_1(K_H')$, as every perturbation is undone by carrying it out again.
	Then, if we let $\mathsf P=\mathsf P_0\circ\mathsf P_1$ and $\mathsf T_{\rm imm}=\mathsf T_1\circ\mathsf T_0$, the transduction $\mathsf P$ is a perturbation, 
	the transduction $\mathsf T_{\rm imm}$ is an immersive transduction, and, for every $H\in\Dd$, we have
	$K_H'\in \mathsf T_1(K_H)\subseteq \mathsf T_1\circ\mathsf T_0(G_H)=\mathsf T_{\rm imm}(G_H)$ and
	$H\in \mathsf P_0(K_H)\subseteq \mathsf P_0\circ\mathsf P_1(K_H')=\mathsf P(K_H')$.
	Hence, $\mathsf P\circ\mathsf T_{\rm imm}$ subsumes $\mathsf T$ on the chain $(\Cc,\Ss,\Dd)$.
	
	Toward this end, consider  $H\in\Dd$ and let $G_H$ and $K_H$ be such that $G_H\in\Cc$ and 
	$K_H\in\mathsf{T}_0(G_H)\cap\mathsf P_0(H)$.
	Note that this last condition is equivalent to the conditions  
	$K_H\in\mathsf T_0(G_H)$ and $H\in\mathsf P_0(K_H)$.
	
	We analyze how the perturbation $\mathsf P_0$ transforms $K_H$ into $H$.
	According to \Cref{lem:perturb}, there exists a partition $(V_{\mathbf x})_{\mathbf x\in\mathbb F_2^k}$, such that the edges $uv$ of $K_H$ that are flipped to get $H$ are those where $u\in V_{\mathbf x}$, $v\in V_{\mathbf y}$, and $\langle\mathbf{x},\mathbf{y}\rangle=1$.
	For $\mathbf x,\mathbf y\in\mathbb{F}_2^k$, we denote by $B_{\mathbf x,\mathbf y}$ the graph with vertex set $V_{\mathbf x}\cup V_{\mathbf y}$, whose edges are the pairs $uv$ with $u\neq v$, $u\in V_{\mathbf x}$, $v\in V_{\mathbf y}$ and $uv\in E(K_H)$. We say that the pair $(\mathbf x,\mathbf y)$ is \emph{problematic} if $B_{\mathbf x,\mathbf y}$ contains $K_{t,t}$ as a subgraph. Note that if $(\mathbf x,\mathbf y)$ is problematic, then the edges of $B_{\mathbf x,\mathbf y}$  are flipped when transforming $K_H$ into $H$, as this latter graph excludes $K_{t,t}$ as a subgraph. Thus $\langle\mathbf{x},\mathbf{y}\rangle=1$ if $(\mathbf x,\mathbf y)$ is problematic.
	The next claim gives some basic structural properties of problematic pairs.
	
	\begin{claim}
		\label{claim:cc}
		If $(\mathbf x,\mathbf y)$ is problematic, then 
		$B_{\mathbf x,\mathbf y}$ has $c(\mathbf x,\mathbf y)\leq 2t$ connected components, each of diameter at most $6t$.
	\end{claim}
	\begin{claimproof}
		We first consider the case where $\mathbf y=\mathbf x$.
		If $B_{\mathbf x,\mathbf x}$ has $2t$ connected components; then by picking one vertex in each connected component, we get an independent set of size at least $2t$.
		Also, if some connected component of $B_{\mathbf x,\mathbf x}$ has diameter at least $4t$, then $B_{\mathbf x,\mathbf x}$ contains an induced path with $4t$ vertices hence (by picking one vertex over two in this path) an independent set of size $2t$. In both cases, after flipping the edges in $V_{\mathbf x}$ we get a clique of size $2t$, contradicting the assumption that $H$ does not contain $K_{t,t}$ as a subgraph. 
		
		Now consider the case where $\mathbf y\neq\mathbf x$.
		As above, no connected component of $(\mathbf x,\mathbf y)$ has diameter greater than $6t$, as if $v_0,\dots,v_{6t-3}$ is an induced path
		of $B_{\mathbf x,\mathbf y}$, then $\{v_0,v_6,\dots,v_{6t-6}\}$ and 
		$\{v_3,v_9,\dots,v_{6t-3}\}$ define, after flipping, a $K_{t,t}$ subgraph in $H$.	As $(\mathbf x,\mathbf y)$ is problematic, one connected  component $C_0$ of $B_{\mathbf x,\mathbf y}$ contains a $K_{t,t}$ subgraph. Assume for contradiction that $B_{\mathbf x,\mathbf y}$  has $2t-1$ other connected components. Then at least $t$ of them contain a vertex in a same part, say $V_{\mathbf x}$. The edges flipped between these vertices and the vertices in $C_0\cap V_{\mathbf y}$ then contain $K_{t,t}$ as a subgraph, contradicting our assumption.
		\end{claimproof}
	
	We number the connected components of $B_{\mathbf x,\mathbf y}$ from $1$ to $c(\mathbf x,\mathbf y)$ in an arbitrary way (but consistently, when  considering that $B_{\mathbf x,\mathbf y}$ and  $B_{\mathbf y,\mathbf x}$ are the same graph).
	
	We now refine the partition $(V_{\mathbf x})_{\mathbf x\in\mathbb F_2^k}$,
	by partitioning each $V_{\mathbf x}$ into at most $(2t)^{2^k}$ parts~$W_{\mathbf x,f}$, where $f$ maps each vector $\mathbf y$ such that $(\mathbf x,\mathbf y)$ is problematic to $c(\mathbf x,\mathbf y)$ and each vector $\mathbf y$ such that $(\mathbf x,\mathbf y)$ is not problematic to $0$.  
	Note that, for every pair  $(\mathbf x,\mathbf y)$ of vectors, if  
	$W_{\mathbf x,f}$ is any subpart of $V_{\mathbf x}$ and 
	$W_{\mathbf y,g}$ is any subpart of $V_{\mathbf y}$, we have
	\[
	f(\mathbf y)=0\quad\iff\quad(\mathbf x,\mathbf y)\text{ is not problematic}\quad\iff\quad g(\mathbf x)=0.\]
	
	We define a perturbation $\mathsf P_1$ that flips the edges with one end in 
	$W_{\mathbf x,f}$ and one end in~$W_{\mathbf y,g}$ if
	$\langle\mathbf{x},\mathbf{y}\rangle=1$ and
	$f(\mathbf y)\neq 0$ and $f(\mathbf y)=g(\mathbf x)$ (meaning that 	 $(\mathbf x,\mathbf y)$ is a problematic pair, but 	$W_{\mathbf x,f}$ and $W_{\mathbf y,g}$  are in the same connected component of $B_{\mathbf x,\mathbf y}$).
	
	Now, let $K_H'\in\mathsf P_1(K_H)$ be the graph obtained by applying the perturbation with the marking defined above.
	Let $B'_{(\mathbf x,f),(\mathbf y,g)}$ be the subgraph of $K_H'$
	with vertex set  $W_{\mathbf x,f}\cup W_{\mathbf y,g}$ whose edges are those edges of $K_H'$ with one endpoint in  $W_{\mathbf x,f}$ and the other endpoint in~$W_{\mathbf y,g}$.
	\begin{itemize}
		\item If $\langle\mathbf x,\mathbf y\rangle=0$, then $B'_{(\mathbf x,f),(\mathbf y,g)}$ is a subgraph of $H$, thus it does not contain $K_{t,t}$ as a subgraph;
		\item otherwise, if $f(\mathbf y)=0$, then $B'_{(\mathbf x,f),(\mathbf y,g)}$ is a subgraph of $B_{\mathbf x,\mathbf y}$, thus it does not contain $K_{t,t}$ as a subgraph since $f(\mathbf y)=0$ implies that $(\mathbf x,\mathbf y)$ is not problematic;
		\item otherwise, if $f(\mathbf y)=g(\mathbf y)$, then $B'_{(\mathbf x,f),(\mathbf y,g)}$ is a subgraph of $H$, thus it does not contain $K_{t,t}$ as a subgraph;
		\item otherwise,  $B'_{(\mathbf x,f),(\mathbf y,g)}$ is edgeless, as $W_{\mathbf x,f}$ and $W_{\mathbf y,g}$ are included in two different connected components of $B_{\mathbf x,\mathbf y}$.
	\end{itemize}
	As the number of parts is at most $(2^t)^{2^k}$, we deduce that
	$K_H'$ contains no $K_{t',t'}$ as a subgraph, where
	$t'=(2^t)^{2^k}t$.
	
	We now define the immersive transduction $\mathsf T_1$ from the perturbation $\mathsf P_1$ by conditioning the flip of an edge $uv$ to the further condition that the distance between $u$ and $v$ is at most~$6t$. Because the maximum diameter of a connected component of $B_{\mathbf x,\mathbf y}$ is at most $6t$ (by \Cref{claim:cc}), we get $K_H'\in\mathsf T_1(K_H)$ as desired.
\end{proof}

A stronger result can be obtained when we assume that the class $\Dd$ is addable, instead of weakly sparse. This will be established in \Cref{crl:slunion}.

\section{Immersive transductions}
\label{sec:immersive}

Intuitively, copying operations and perturbations are simple operations. 
The main complexity of a transduction is captured by its immersive part, which consists just of an interpretation using strongly local formulas. 
The strongly local character of immersive transductions is the key tool in 
our further analysis. It will be very useful to give another (seemingly) 
weaker property for the existence of an immersive transduction in a
class, 
which is the existence of a transduction that does not shrink the distances too much, as we prove now.

\begin{lem}
	\label{lem:epsilon}
	Assume a graph class $\Dd$ is a $\mathsf T$-transduction  of a graph class $\Cc$, where $\mathsf T$ is non-copying.
	Assume that there exists a positive real $1\geq \epsilon>0$ with the property that 
	for every $G\in\Cc$ and every $H\in \mathsf T(G)\cap\Dd$ we have
	${\rm dist}_{H}(u,v)\geq \epsilon\,{\rm dist}_G(u,v)$ (for all $u,v\in V(H)$).
	Then there exists an immersive transduction $\mathsf T'$
	that subsumes $\mathsf T$ on the chain $(\Dd,\Cc)$.
\end{lem}
\begin{proof}
	Let $\mathsf T$ be a non-copying transduction.
	According to \Cref{lem:normal0}, $\mathsf T$ is subsumed by a transduction $\mathsf T'$ with associated interpretation $\mathsf I'=(M(x),\eta'(x,y))$, where $M$ is a unary relation and $\eta'$ is a local formula. Let $G$ and $H$ be as in the statement. Note that by assumption if $\rm{dist}_G(x,y)>1/\epsilon$, then~$x$ and~$y$ are not connected in $H$. Hence, we define
	$\widehat{\mathsf I}=(M(x),\eta'(x,y)\wedge {\rm dist}(x,y)\leq 1/\epsilon)$. Then, the transduction defined by $\widehat{\mathsf I}$ is immersive and subsumes the transduction $\mathsf T$ on the chain~$(\Dd,\Cc)$.
\end{proof}

The property of an immersive transduction not to shrink distances too much implies that the sizes of balls cannot increase too much. Formally,
for a class $\mathscr C$ define the \emph{dilation function} $\Upsilon_{\mathscr C}:\mathbb N\rightarrow\mathbb N\cup\{\infty\}$ by
\[
\Upsilon_{\mathscr C}(r)=\adjustlimits\sup_{G\in \mathscr C}\max_{v\in V(G)}|B_r^G(v)|.
\] 
Then, the following fact directly follows from the strong locality of immersive transductions.
\begin{fact}
	\label{fact:dilation}
	Assume there exists an immersive  transduction encoding a class $\mathscr D$ in a class $\mathscr C$. Then, there exists a constant $c$ such that for every $r\in\mathbb N$ we have
	$\Upsilon_{\mathscr D}(r)\leq \Upsilon_{\mathscr C}(cr)$.
\end{fact}

Recall that $G\join K_1$ is obtained from $G$ by adding a 
new vertex, called an \emph{apex}, that is connected to all vertices of 
$G$. Of course, by adding an apex we shrink all distances in $G$. The
next lemma shows that when we can transduce $\Cc\join K_1$ in a 
class $\Ff$ with an immersive transduction, then we can in fact transduce 
$\Cc$ in the local balls of 
$\Ff$.

\begin{lem}
	\label{cor:apex}
	Let $\Cc,\Ff$ be graph classes, and let $\mathsf T$ be an 
	immersive transduction encoding a class $\Dd$ in $\Ff$ with
	$\Dd\supseteq \Cc\join K_1$.
	Then there exists an integer $r$ such that $\Cc\sqsubseteq_\FO^\circ\rloc{r}{\Ff}$.
\end{lem}
\begin{proof}
	Let $\mathsf T=\mathsf I\circ\Gamma_{\mathcal U}$ be an immersive transduction encoding $\Dd$ in $\Ff$, say $\mathsf{I}=(\vartheta(x), \eta(x,y))$. By \Cref{lem:normal0} we may assume that $\vartheta(x)=M(x)$. Furthermore, as $\mathsf{T}$ is immersive, $\eta(x,y)$ is strongly local, say $\models \eta(x,y)\rightarrow \mathrm{dist}(x,y)\leq r$. 
	For every graph $G\in\Cc$ there exists a graph $F\in\Ff$ such that $G\join K_1=\mathsf I(F^+)$, where~$F^+$ is a ${\mathcal U}$-expansion of $F$. Let $v$ be the apex of $G\join K_1$.
	By the strong locality of $\eta$ we get $\mathsf I(F^+)=\mathsf I(B_r^{F^+}(v))$. 
	Recall that by $\mathsf{H}$ we denote the hereditary transduction allowing to take induced subgraphs. Then $G$ can be encoded 
	in the class $\rloc{r}{\Cc}$ by the non-copying transduction $\mathsf H\circ\mathsf T$. 
\end{proof}

Finally, we show that when we encode an addable class, we do not need perturbations at all. Recall that a class $\Cc$ is addable if the disjoint union of any two graphs in $\Cc$ is also in $\Cc$.

\begin{lem}
	\label{lem:slunion}
	Let $\Cc$ be an addable  class and let $\Dd$ be a class with $\Cc\sqsubseteq_\FO^\circ\Dd$.  Then there exists an immersive transduction encoding $\Cc$ in $\Dd$.
\end{lem}
\begin{proof}
	According to \Cref{thm:normal}, the transduction of $\Cc$ in $\Dd$ is subsumed by the composition of an immersive transduction $\mathsf T$ (with associated interpretation $\mathsf I=(\nu,\eta)$) and a perturbation (with associated interpretation~$\mathsf I_P$). As $\eta$ is strongly local there exists 
	$r$ such that for all $G\in \Dd$ and ${\mathcal U}_\mathsf{T}$-expansions 
	$G^+$ and all $u,v\in \mathsf{I}(G^+)$ we have $\dist_{\mathsf{I}(G^+)}(u,v)\geq
	\dist_G(u,v)/r$. 
	Let $c$ be the number of unary relations used in the perturbation. 
	Let $H$ be a graph in~$\Cc$, let $n>3\cdot c^{|H|}$ and let $K=nH$ ($n$ disjoint copies of $H$). By assumption there exists an expansion~$G^+$ of a graph~$G$ in $\Dd$ with $K=\mathsf I_P\circ\mathsf I(G^+)$. By the choice of $n$, at least $3$ copies $H_1,H_2$, and~$H_3$ of $H$ in~$K$ satisfy the same unary predicates at the same vertices. 
	For $a\in\{1,2,3\}$ and $v\in V(H_1)$, we denote by~$\tau_a(v)$ the vertex of $H_a$ corresponding to the vertex~$v$ of $H_1$ ($\tau_1(v)$ being the 
	vertex~$v$ itself).	 
	Let $u,v$ be adjacent vertices of $H_1$. Assume that $u$ and $v$
	have distance greater than~$r$ in $G$. Then $u$ and $v$ are made
	adjacent in $K$ by the perturbation $\mathsf{P}$ (the edge cannot have
	been created by $\eta$ as it is strongly $r$-local). As $\tau_a(u)$ is not
	adjacent with $\tau_b(v)$ for $b\neq a$ there must be paths of length
	at most $r$ linking~$\tau_a(u)$ with~$\tau_b(v)$ in $G$ for $a\neq b$ (the 
	interpretation $\mathsf{I}$ must have introduced an edge that the
	perturbation removed again). This however implies that there is a 
	path of length at most $3r$ between $u$ and $v$ in $G$ (going 
	from $u$ to $\tau_2(v)$ to $\tau_3(u)$ to $v$). It follows that for all 
	$u,v\in V(K)$ we have $\dist_K(u,v)\geq \dist_G(u,v)/(3r)$. 
	Hence the transduction obtained by composing $\mathsf T$ with, the extraction of the induced subgraph $H_1$ implies the existence of an immersive transduction encoding $\Cc$ in $\Dd$, according to \Cref{lem:epsilon}. 
\end{proof}
\begin{cor}[Elimination of the perturbation]\label{crl:slunion}
	Let $\Cc$ be an addable class with  $\Cc\sqsubseteq_\FO\Dd$.  Then there exists a copy operation $\mathsf C$ and an immersive transduction~$\mathsf T_{\rm imm}$ such that $\mathsf T_{\rm imm}\circ\mathsf C$ is a transduction encoding $\Cc$ in $\Dd$. 
\end{cor}
We also deduce the following (by \Cref{fact:dilation}).
\begin{cor}\label{cor:dil_union}
	Let $\Cc$ be an addable class with  $\Cc\sqsubseteq_\FO\Dd$.  Then there exist constants $c_1,c_2$ such that for every $r\in\mathbb N$ we have
	\[
	\Upsilon_{\mathscr D}(r)\leq c_1\Upsilon_{\mathscr C}(c_2r).
	\]
\end{cor}

\begin{lem}
	\label{lem:imm_con}
	Let $\Cc$ be a class of connected graphs and let $\Dd$ be an addable class.
	Assume there exists an immersive transduction $\mathsf T$ of $\Cc$ in $\Dd$. Then, $\mathsf T(\Dd)$ includes the closure of $\Cc$ by disjoint union, that is, the class of all graphs whose connected components belong to $\Cc$.
\end{lem}
\begin{proof}
	Let $G=G_1\union\dots\union G_n$, where $G_1,\dots,G_n\in\Cc$. 
	As $\Cc\in\mathsf T(\Dd)$ there are $H_1,\dots,H_n$ in $\Dd$ with
	$G_i\in\mathsf T(H_i)$ (for every $i\in[n]$). Let $\mathsf I$ be the interpretation part of $\mathsf T$. There exist unary expansions $H_1^+,\dots,H_n^+$ of $H_1,\dots,H_n$ with $G_i=\mathsf I(H_i^+)$ 
	(for every $i\in[n]$).
	As $\mathsf T$ is immersive, $\mathsf I(H_1^+\union\dots\union H_n^+)=\mathsf I(H_1^+)\union\dots\union\mathsf I^+(H_n^+)$.  As $\Dd$ is addable and
	$H_1,\dots,H_n$ belong to $\Dd$,  so does their disjoint union $H$. As  the transduction $\mathsf T$ is immersive,  $\mathsf T(H)=\mathsf T(H_1)+\cdots+\mathsf T(H_n)$, where
	$\Cc_1+\Cc_2=\{X_1\union X_2\colon X_1\in\Cc_1\text{ and }X_2\in\Cc_2\}$.
	Hence,  $G\in\mathsf T(H)$.
\end{proof}

\section{Gluing transductions, with applications to\except{toc}{\texorpdfstring{\linebreak}{}} orientation and monotone closure}
While any graph class $\Cc$ can transduce its hereditary closure, the class of cliques witnesses that not every class can transduce its monotone closure. In \cref{lem:monotone}, we give a sufficient condition for $\Cc$ to transduce its monotone closure. We take this as an opportunity to illustrate the following useful construction of gluing, which was implicit in \cite{shrubbery}. It allows us to break the problem of defining  a 
transduction into easier subproblems, as we shall illustrate with some examples.

\begin{defi}
	Let $n$ be a positive integer, 
	let $\mathsf T_{i,j}$ be non-copying transductions indexed by $1\leq i\leq j\leq n$ defined
	by interpretations $\mathsf I_{i,j}=(\top,\eta_{i,j}(x,y))$ of graphs in  
	$\Uu_{i,j}$-colored graphs. 
	By renaming the symbols in $\Uu_{i,j}$ if necessary, we assume that these sets of symbols are all disjoint, and that they are disjoint from another set $\{V_1,\dots,V_n\}$.
	For $1\leq i\leq j\leq n$ we call  $\mathsf S_{i,j}$ the interpretation associated to the pair hereditary transduction extracting the subgraph defined by sets $A_i$ and $A_j$, where $A_k = V_k\setminus \bigcup_{k'<k}V_{k'}$.

	The \emph{gluing} of the (indexed) set $\{\mathsf T_{i,j}\colon 1\leq i\leq j\leq n\}$ of  non-copying transductions is the  non-copying transduction $\mathsf T$ defined by the interpretation $\mathsf I=(\nu(x),\eta(x,y))$ of $\widehat\Uu$-colored graphs, where 
	$\widehat\Uu:=\{V_1,\dots,V_n\}\cup\bigcup_{1\leq i\leq j\leq n} \widehat\Uu_{i,j}$, the domain
	of $\mathsf I(G)$ is defined by the formula $\bigvee_{1\leq i\leq n}V_i(x)$, and the edge set is the union of the edge sets of $\mathsf I_{i,j}\circ S_{i,j}(G)$.
\end{defi}

\begin{figure}[h!t]
	\centering\includegraphics[width=\textwidth]{gluing}
	\caption{Gluing transductions. $H[A_i,A_j]\in\mathsf T_{i,j}(G[A_i,A_j])$ for $1 \leq i \leq j \leq 3$.}
	\label{fig:gluing}
\end{figure}

It is direct from the definition that this transduction satisfies the following property: for every graph $G$, we have
\[
\mathsf T(G)=\{H\colon \exists A_1,\dots,A_n\subseteq V(G)\text{ disjoints }\quad\forall i\leq j \in [n]\quad H[A_i,A_j]\in\mathsf T_{i,j}(G[A_i,A_j])\}.
\]

As an application, we have the following property.
\begin{lem}
	\label{lem:gluing}
 Let $\Cc, \Dd$, and $\Dd'$ be classes of graphs and let $n$ be an integer.
 Assume that~$(\mathsf S,\mathsf T)$ is a  non-copying pairing of $\Dd$ and $\Dd'$, and 
that every graph $G\in\Cc$ has a vertex partition $V_1,\dots,V_n$ with
$G[V_i,V_j]\in\Dd$ for all $1\leq i\leq j\leq n$.

Then, there exists a class $\Cc'$ and a non-copying pairing $(\widehat{\mathsf S},\widehat{\mathsf T})$ of $\Cc$ and $\Cc'$, such that 
every graph $G\in\Cc'$ has a vertex partition $V_1,\dots,V_n$ with
$G[V_i,V_j]\in\Dd'$ for all $1\leq i\leq j\leq n$.
\end{lem}
\begin{proof}
	We let $\mathsf S_{i,j}:=\mathsf S$ and  $\mathsf T_{i,j}=\mathsf T$ for all $1\leq i\leq j\leq n$, and define $\widehat{\mathsf S}$ (resp.~ $\widehat{\mathsf T}$) as the gluing of the transductions
	$\mathsf S_{i,j}$ (resp.~of the transductions $\mathsf T_{i,j}$). 
\end{proof}

Because of the above construction, it will be useful to consider graph colorings, where any two classes induce a graph in an ``easy'' class. An  example of such a coloring is the star coloring.
A {\em star coloring} of a graph $G$ is  a proper coloring of $V(G)$ (i.e. with no two adjacent vertices receiving the same color) such that any two color classes induce a star forest (or, equivalently, such that no path of length three is $2$-colored). (See~\Cref{fig:starcol}.)

\begin{figure}[ht]
	\centering\includegraphics[width=.5\textwidth]{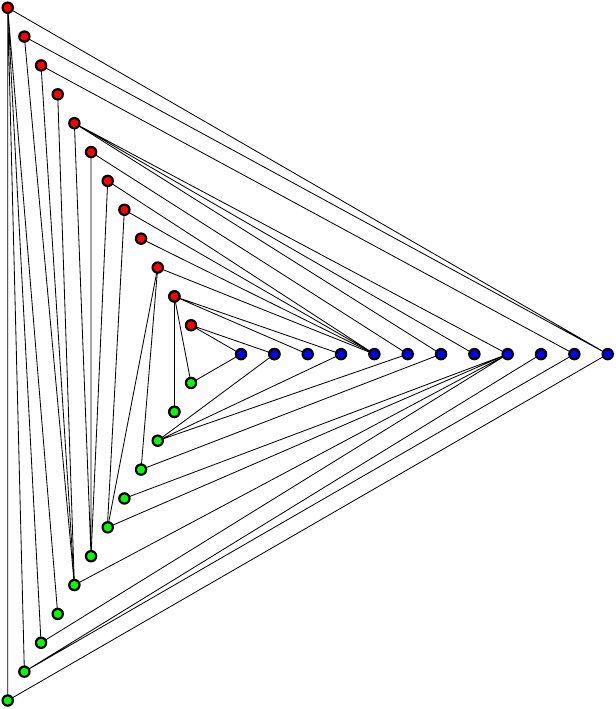}
	\caption{Example of a star coloring using $3$ colors. Every color class is independent, and every two color classes induce a star forest.}
	\label{fig:starcol}
\end{figure}

The star chromatic number $\chi_{\rm st}(G)$ of $G$ is the minimum number of colors in a star coloring of $G$~(see \cite{alon2}). The star chromatic number is bounded on classes of graphs excluding a minor~\cite{Taxi_jcolor} and, more generally, on every class with bounded expansion~\cite{nevsetvril2008grad}.

The following corollaries are simple applications of transduction gluing. A slightly more general statement (about Gaifman graphs) was proved in a much more complicated way in~\cite{TWWP-arxiv}. 

\begin{cor}
	\label{lem:monotone}
	Let $\Cc$ be a class of graphs with bounded star chromatic number. Then the
	monotone closure of $\Cc$ is a non-copying immersive transduction of $\Cc$.
\end{cor}
\begin{proof}
	According to \Cref{lem:gluing}, it is sufficient to prove the results for the class $\Tt_2$ of star forests. So let $G = (V, E) \in \Tt_2$ and let $H = (V_H, E_H) \subseteq G$ be a subgraph of $G$. In $G$, we can color the vertices in $V_H$ by the unary predicate $A$, and color the vertices in $V_H$ incident to an edge in $E_H$ by the unary predicate $B$.
	The interpretation part  $\mathsf I=(\nu(x),\eta(x,y))$ of the transduction defines the vertex set as the set of all vertices marked $A$ ($\nu(x):=A(x)$), and the edge set as the set of all the pairs of adjacent vertices in $B$ ($\eta(x,y):=B(x)\wedge B(y)\wedge E(x,y)$).
	That the transduction is non-copying and  immersive is obvious.
\end{proof}

Similarly, we have
\begin{cor}
	\label{lem:orient}
	Let $\Cc$ be a class of graphs with bounded star chromatic number. Then the class $\vec\Cc$ of all the orientations of graphs in  $\Cc$ is a non-copying immersive transduction of $\Cc$.
\end{cor}

We can also use the same techniques to color the edges in a bounded number $k$ of colors\footnote{Formally, this amounts to define $k$ adjacency relations $E_1,\dots,E_k$ partitioning $E$.}.

\begin{cor}
	\label{lem:color}
	Let $\Cc$ be a class of graphs with bounded star chromatic number. Then the class of all the edge-colorations of  graphs in  $\Cc$ (using a bounded number of colors) is a non-copying immersive transduction of $\Cc$.
\end{cor}

The next corollary is an extension of 
\Cref{lem:monotone}, as $\Cc\nabla 0$ is nothing but the monotone closure of $\Cc$.

\begin{cor}
	For every class $\Cc$ with bounded star chromatic number and every integer $r$, the class
	$\Cc\nabla r$ is a non-copying immersive transduction of $\Cc$. 
\end{cor}
\begin{proof}
	For each connected subgraph to be contracted, we color the edges in a spanning breadth-first-search tree with color $1$, we color the edges to be deleted with color $2$, and we color the set of vertices to be kept in the minor (including one vertex in each contracted subgraph) with color~$3$. 
	This can be done via a transduction, according to \Cref{lem:color}. Then, the shallow minor is obtained by an easy transduction: its vertex set is the set of all the vertices colored~$3$, and two vertices colored $3$ are adjacent in the minor if there is in $G$ a path of length at most $4r$ linking them with all its edges colored $1$ but one edge that is uncolored.
\end{proof}

A special case of gluing concerns the case where all the transductions $\mathsf T_{i,j}$ with $i\neq j$ are defined by the interpretation $(\top,\bot)$, which means that no edges will exist  between the parts in the gluing of the transductions.
We call such a gluing a \emph{trivial gluing} (see \Cref{fig:gluing_triv}).
The trivial gluing $\mathsf T$ of $\mathsf T_1,\dots, \mathsf T_n$ is such that, for every graph $H$, we have
\[
T(G)=\{H=H_1\union\dots\union H_n\colon H_i\in T_i(G[A_i])\text{ with }A_1,\dots,A_n\text{ disjoint subsets of }G\}.
\]
Note that trivial gluing obviously extends to the case where (some of) the transductions $\mathsf T_1,\dots,\mathsf T_n$ are copying.

\begin{figure}[h!t]
	\centering\includegraphics[width=\textwidth]{gluing_triv}
	\caption{Trivial gluing of transductions. $H[A_i]\in\mathsf T_{i,j}(G[A_i])$.}
	\label{fig:gluing_triv}
\end{figure}

\begin{fact}
	\label{fact:triv_gluing}
	Let $\Cc_1,\Cc_2,\Dd_1,\Dd_2$ be classes of graphs with $\Dd_1\sqsubseteq \Cc_1$ and $\Dd_2\sqsubseteq \Cc_2$, where $\sqsubseteq$ is one of $\sqsubseteq_\FO^\circ$ and $\sqsubseteq_\FO$. 
	Then, $\Dd_1+\Dd_2\sqsubseteq\Cc_1+\Cc_2$.
\end{fact}
\begin{proof}
		Let $\Dd_1$ be a $\mathsf T_1$-transduction of $\Cc_1$, let $\Dd_2$ be a $\mathsf T_2$-transduction of $\Cc_2$, and let $\mathsf T$ be the trivial gluing of $\mathsf T_1$ and $\mathsf T_2$.
		Let $H_1\in\Dd_1$ and $H_2\in \Dd_2$. By assumption, there exist $G_1\in\Cc_1$ and $G_2\in\Cc_2$ with
		 $H_1\in \mathsf T_1(G_1)$  and $H_2\in \mathsf T_2(G_2)$. By construction,  $H_1\union H_2\in \mathsf T(G_1\union G_2)$. Thus, 
		$\Dd_1+\Dd_2$ is a $\mathsf T$-transduction of $\Cc_1+\Cc_2$.
\end{proof}

In this setting, the following property,
 which generalizes the notion of addable class, will be of particular interest. 
 Let $\equiv$ be one of $\equiv_\FO^\circ$ and $\equiv_\FO$.
 A class $\Cc$ is \emph{$\equiv$-additive} (or simply \emph{additive} when  $\equiv$ is clear from the context) if $\Cc+\Cc\equiv\Cc$.
 Note that $\equiv_\FO^\circ$-additivity implies 
 $\equiv_\FO$-additivity.
As a class $\Cc$ is addable if
$\Cc+\Cc=\Cc$, we deduce that
if a class $\Cc$ is addable and a class $\Dd$ is such that $\Dd\equiv\Cc$, then $\Dd$ is $\equiv$-additive.
The significance of additivity is witnessed by the following fact.
\begin{fact}
	\label{fact:add_gluing}
	Let $\sqsubseteq$ be one of $\sqsubseteq_\FO^\circ$ and $\sqsubseteq_\FO$ and let $\equiv$ be the associated equivalence relation. 
	Let $\Cc$ be an $\equiv$-additive class and let $\Dd_1,\Dd_2$ be classes of graphs with $\Dd_1\sqsubseteq \Cc$ and $\Dd_2\sqsubseteq \Cc$.
	Then, $\Dd_1+\Dd_2\sqsubseteq\Cc$.
\end{fact}
\begin{proof}
According to \Cref{fact:triv_gluing}, we have  $\Dd_1+\Dd_2\sqsubseteq \Cc+\Cc\sqsubseteq\Cc$.
\end{proof}

\section{About copying}

The copying step of transductions makes them more involved to reason about, and is often not needed. Thus, many papers consider non-copying transductions instead. In this section, we record various results about when the copying step can be safely ignored.

\begin{fact}
	\label{fact:copy_equiv}
	For every class $\Cc$ and integer $k$ we have
	$\Cc\equiv_\FO	\mathsf C_k(\Cc)$.
		Consequently,
\[	\Cc\sqsubseteq_\FO\Dd\quad\Longrightarrow\quad \mathsf C_k(\Cc)\sqsubseteq_\FO\mathsf C_k(\Dd).\]
\end{fact}
\begin{proof}
	As $\mathsf C_k$ is a transduction, we have $\mathsf C_k(\Cc)\sqsubseteq_\FO	\mathsf\Cc$ and, as every graph $G$ is an induced subgraph of
	$\mathsf C_k(G)$ we have
	$\Cc\subseteq\mathsf H(\mathsf C_k(\Cc))$, hence $\Cc\sqsubseteq_\FO	\mathsf C_k(\Cc)$.
	
		Assume $\Cc\sqsubseteq_\FO\Dd$.
	Then,  $\mathsf C_k(\Cc)\equiv_\FO\Cc\sqsubseteq_\FO\Dd\equiv_\FO \mathsf C_k(\Dd)$.
\end{proof}

\begin{fact}
	For every class $\Cc$ and integers $k\leq\ell$ we have
	$C_k(\Cc)\sqsubseteq_\FO^\circ	\mathsf C_\ell(\Cc)$.
\end{fact}
\begin{proof}
	Indeed, for every graph $G$, $C_k(G)$ is an induced subgraph of $C_\ell(G)$. Hence, the non-copying transduction $\mathsf H$ witnesses $C_k(\Cc)\sqsubseteq_\FO^\circ	\mathsf C_\ell(\Cc)$.
\end{proof}

\begin{fact}
	\label{fact:red_copy}
	For every class $\Cc$ and integers $k,\ell$ we have
	$\mathsf C_{k\ell}(\Cc)\equiv_\FO^\circ	\mathsf C_k(\mathsf C_\ell(\Cc))$.
\end{fact}
\begin{proof}
	Let $G\in\Cc$.
	Let $v_{i,j}$ be the $j$th clone
	in $\mathsf C_k(\mathsf C_\ell(G))$ of the $i$th clone in 
	$\mathsf C_\ell(G)$ of $v\in V(G)$.
	The neighbors of $v_{i,j}$ in $\mathsf C_k(\mathsf C_\ell(G))$ are the vertices $v_{i,j'}$ (for $j'\in [k]\setminus\{j\}$), the vertices 
	$v_{i',j}$ (for $i'\in [\ell]\setminus\{i\}$), and the vertices $u_{i,j}$ for $u$ neighbor of $v$ in $G$. Note that $\mathsf C_k(\mathsf C_\ell(G))$ and 
	$\mathsf C_\ell(\mathsf C_k(G))$ are isomorphic.
	
	Consider a set $\mathcal U=\{M_{i,j}\colon (i,j)\in [\ell]\times [k]\}$ of unary relations and let $G\in \Cc$.
	Consider the non-copying transduction $\mathsf T$ defined by the formulas $\nu(x)=\top$ and 
	\[
	\phi(x,y):=E(x,y)\vee \bigvee_{((i,j),(i',j'))\in\binom{[\ell]\times[k]}{2}} 
	M_{i,j}(x)\wedge M_{i',j'}(y)\wedge \exists z 
	(M_{i,j'}(z)\wedge E(x,z)\wedge E(y,z)).
	\]
	
	By marking by $M_{i,j}$ the vertices $v_{i,j}$ of  $\mathsf C_k(\mathsf C_\ell(G))$, we get that $\mathsf C_{k\ell}(G)\in \mathsf T(\mathsf C_\ell(\mathsf C_k(G))$. Hence,
	$\mathsf C_{k\ell}(\Cc)\sqsubseteq_\FO^\circ \mathsf C_\ell(\mathsf C_k(\Cc))$.
	
	Conversely, by considering the transduction defined by
	the formulas $\nu(x)=\mathsf T$ and 
	\[
	\phi(x,y):=E(x,y)\wedge \bigvee_{i,i'\in [\ell]}\bigvee_{j,j'\in [k]}
	M_{i,j}(x)\wedge M_{i',j'}(y)\wedge ((i=i')\vee (j=j')),\] 
	we get (using an adequate marking of $\mathsf C_{k\ell}(G)$)
	that $ \mathsf C_\ell(\mathsf C_k(\Cc))\sqsubseteq_\FO^\circ \mathsf C_{k\ell}(\Cc)$.
\end{proof}

An observation is that the copying transduction $\mathsf C_k$ transports the quasi-orders $\sqsubseteq_\FO$ and~$\sqsubseteq_\FO^\circ$. (For $\sqsubseteq_\FO$ this was stated in \Cref{fact:copy_equiv}.)
\begin{fact}
	\label{fact:copymove}
	For every non-copying transduction $\mathsf T$ and every integer $k$, there is a  non-copying transduction $\mathsf T'$ with
	\[\mathsf T'\circ\mathsf C_k\geq \mathsf C_k\circ\mathsf T.\]
	
	Consequently,
\[
		\Cc\sqsubseteq_\FO^\circ\Dd\quad\Longrightarrow\quad \mathsf C_k(\Cc)\sqsubseteq_\FO^\circ\mathsf C_k(\Dd).\]
\end{fact}
\begin{proof}
	The possibility to move the copying first is not difficult to prove, though a bit technical. For a formal proof of this fact, we refer to \cite{SBE_TOCL}.
	
		Assume $\Cc\sqsubseteq_\FO^\circ\Dd$.
		Then there exists a non-copying transduction $\mathsf T$ such that $\Cc\subseteq\mathsf T(\Dd)$.
		According to the first statement there exists a non-copying transduction $\mathsf T'$ with $\mathsf T'\circ\mathsf C_k\geq \mathsf C_k\circ\mathsf T$. In particular, $\mathsf C_k(\Cc)\subseteq \mathsf T'(\mathsf C_k(\Dd))$. As $\mathsf T'$ is non-copying, we deduce
		$\mathsf C_k(\Cc)\sqsubseteq_\FO^\circ\mathsf C_k(\Dd)$.
\end{proof}

The next technical lemma is very useful. It shows that copying is only needed if the cloned vertices actually appear in final graph, and are not needed, for example, just to witness existential quantifiers in the interpretation.
\begin{lem}[Elimination of the copying]
	\label{lem:elim_copy}
	Let $\Cc,\Dd$ be graph classes, and let $\mathsf T=\mathsf I\circ\Gamma_{\mathcal U}\circ C_k$ be a transduction of $\Cc$ from $\Dd$.
	Assume that for every $G\in\Cc$ there exists a graph $H\in\Dd$ such that $G\in\mathsf T(H)$ and $V(G)$ does not contain two clones of any vertex of $H$.
	
	Then, $\Cc\sqsubseteq_\FO^\circ\Dd$.
\end{lem}

\begin{proof}
	Let $\nu(x)$ and $\eta(x,y)$ be the formulas defining the interpretation $\mathsf I$. We identify the vertex set of $K=\mathsf C_k(H)$ with $V(H)\times [k]$.  Let $K^+\in\Gamma_{\mathcal U}(K)$ be such that $G=\mathsf I(K^+)$.
	Let~${\mathcal U}'$ be the signature with, for each  unary relation $P\in{\mathcal U}$, $k$ unary relations $P_1,\dots,P_k$, as well as $k$ additional unary relations $V_1,\dots,V_k$.
	We define $H^+\in\Gamma_{{\mathcal U}'}(H)$ as follows:
	for each vertex $(v,i)\in V(K^+)$ we let 
	$H^+\models P_i(v)\ \iff\ K^+\models P((v,i))$. 
	Additionally, we let
	$H^+\models V_i(v)\ \iff\ K^+\models \nu((v,i))$. Note that, by assumption, no vertex can have marks $V_i$ and $V_j$ for $i\neq j$.
	
	We now prove by induction on the complexity of formulas that, for every formula $\phi(x_1,\dots,x_p)$ and every $i_1,\dots,i_p\in [k]$ there exists a formula $\hat\phi_{i_1,\dots,i_p}(x_1,\dots,x_p)$ such that for every $v_1,\dots,v_p\in V(H)$ we have
	\begin{equation}
		\label{eq:trsp}
		K^+\models \phi((v_1,i_1),\dots,(v_p,i_p))\quad\iff\quad H^+\models\hat\phi_{i_1,\dots,i_p}(v_1,\dots,v_p).
	\end{equation}

	For $\phi(x)=P(x)$ for any fixed unary relation $P\in{\mathcal U}$, we have
	$\hat\phi_i(x):=P_i(x)$. For $\phi(x,y)=E(x,y)$, we have
	$\hat\phi_{i,j}(x,y):=(x=y)$ if $i\neq j$ and $\hat\phi_{i,i}(x,y):=E(x,y)$.
	For $\phi(x,y)=(x=y)$ we have $\hat\phi_{i,j}:=\bot$ if $i\neq j$ and 
	$\hat\phi_{i,i}(x,y):=(x=y)$.
	
	The definitions of $\hat\phi$ easily extend when we consider Boolean combinations of formulas $\psi$ for which $\hat\psi$ is already defined.
	
	If $\phi(x_1,\dots,x_p):=\exists y\ \psi(y,x_1,\dots,x_p)$, then we define
	\[\hat\phi_{i_1,\dots,i_p}(x_1,\dots,x_p):=\bigvee_{j=1}^k \bigl(\exists y\ 
	\hat\psi_{j,i_1,\dots,i_p}(y,x_1,\dots,x_p)\bigr).\]
	
	Let $\mathsf I'$ be the simple interpretation defined by $\nu'(x):=\bigvee_{i=1}^k V_i(x)$ and 
	\[
	\eta'(x,y):=\bigvee_{i_1=1}^k\bigvee_{i_2=1}^k V_{i_1}(x)\wedge V_{i_2}(y)\wedge\hat\eta_{i_1,i_2}(x,y).
	\]

	Then, by construction, $G=\mathsf I'(H^+)$.  Considered as a transduction, 
	the simple interpretation~$\mathsf I'$ is a non-copying transduction (which we  denote $\mathsf T'$ to stress that it is a transduction) with 
	$G\in\mathsf T'(H)$. Hence,
	the  transduction $\mathsf T'$ 	 witnesses $\Cc\sqsubseteq_\FO^\circ\Dd$.
\end{proof}

\begin{lem}
	\label{lem:cancel_copy}
	For every two classes $\Cc$ and $\Dd$ and every positive integer $k$ we have
	\[
	\Cc\sqsubseteq_\FO^\circ\Dd\quad\iff\quad\mathsf C_k(\Cc)\sqsubseteq_\FO^\circ\mathsf C_k(\Dd).
	\]
\end{lem}
\begin{proof}
	The left-to-right implication follows from \Cref{fact:copymove}.
	For the right-to-left implication, assume $\mathsf T$ is a non-copying transduction with $\mathsf C_k(\Cc)\subseteq\mathsf T(\mathsf C_k(\Dd))$.
	Let $G\in\Cc$ and $H\in\Dd$ be such that $\mathsf C_k(G)\in\mathsf T(\mathsf C_k(H))$. We identify the vertex sets of $\mathsf C_k(G)$ (resp.\ $\mathsf C_k(H)$) with $V(G)\times [k]$ (resp.\ $V(H)\times [k]$). Note that there is some injection $f:V(G)\times [k]\rightarrow V(H)\times[k]$ that maps each vertex of $\mathsf C_k(G)$ to the vertex of $\mathsf C_k(H)$ it comes from.
	We define the bipartite graph $B$ with parts $V(H)$ and $V(G)$, where $u\in V(G)$ is adjacent to $v\in V(H)$ if there exist $i,j\in[k]$ with $f(u,i)=(v,j)$. For every subset $A$ of $V(G)$, we have the following bound for the size of the neighborhood $N^B(A)$ of $A$ in the bipartite graph $B$:
	\begin{align*}
	|N^B(A)|&=|\{v\in V(H)\colon \exists u\in A\ \exists i,j\in [k]\text{ with }(v,j)=f(u,i)\}|\\
	&\geq\frac{1}{k}|\{f(u,i)\colon u\in A, i\in [k]\}|=|A|&\text{(as $f$ is injective)}.
	\end{align*}

	That is, in the bipartite graph $B$, $A$ has at least $|A|$ neighbors. According to Hall's marriage theorem~\cite[Theorem 2.1.2]{diestel2012graph}, there exists a matching of $B$ saturating $V(G)$ (that is, a set of disjoint edges of $B$ that covers every vertex in $V(G)$). It follows that there exists an injection $g:V(G)\rightarrow V(H)$ and a mapping $\ell:V(G)\rightarrow[k]$ such that for every $u\in V(G)$ there exists $j\in[k]$ with $f(u,\ell(u))=(g(u),j)$. 
	
	Now, the idea is to mark by a unary relation $M$ the vertices of the form $(u,1)$ in $\mathsf C_k(G)$, and by $L$ the vertices of the form $(u,\ell(u))$. Then, each vertex in $L$ that is not in $M$ is adjacent to a unique vertex in $M$, which is its clone.
	It follows that the adjacency relation~$E$ between the vertices in $M$ can be transported to define (via the formula $\eta$ below) an adjacency relation between the vertices in $L$, so that the graph so defined on $L$ is isomorphic to $G$.
	
	Formally, we consider the transduction $\mathsf S=\mathsf I'\circ\Gamma_{{\mathcal U}'}$ with ${\mathcal U}'=\{M,L\}$, where 
	$\mathsf I'$ is defined by $(L(x),\eta(x,y))$ with 
	\[
	\begin{split}
	\eta(x,y):= \exists x'\ \exists y'\ \Bigl(\bigl((x'=x)\vee (\neg M(x)\wedge E(x',x)\bigr)
	\wedge \bigl((y'=y)\vee (\neg M(x)\wedge E(y',y)\bigr)\\
	\wedge M(x')\wedge M(y')\wedge E(x',y')\Bigr).
	\end{split}
	\]
	
	Considering the coloring operation that
 marks  in $\mathsf C_k(G)$ the vertices of the form $(u,1)$ by $M$ and the vertices of the form $(u,\ell(u))$ by $L$, 
	we get that $\mathsf S\circ\mathsf T\circ\mathsf C_k(H)$ contains a copy of $G$, where no two vertices come from  clones in $\mathsf C_k(H)$. According to \Cref{lem:elim_copy}, we infer that
	$\Cc\sqsubseteq_\FO^\circ\Dd$.
\end{proof}

We say that a class $\Cc$ \emph{\dnnc} if for every
integer $k$ the class $\mathsf{C}_k(\Cc)$ is a non-copying
transduction of $\Cc$; otherwise, we say that $\Cc$ \emph{needs copying}. For example, as the class of all matchings cannot be transduced
from the class of edgeless graphs without copying, the class of edgeless graphs needs
copying. 
\begin{exa} \label{ex:selfcopy}
	The class $\Pp$ of paths \dnnc: Let $k\in\mathbb N$.
	Consider  a set $\mathcal U$ consisting of a single unary relation symbol $M$ and the non-copying transduction $\mathsf T$ defined by the formulas
	$\nu(x):=\top$ and
	$\phi(x,y):={\rm dist}_k(x,y)\vee
	({\rm  dist}_{<k}(x,y)\wedge (M(x)\leftrightarrow M(y)))$, where
	${\rm dist}_k(x,y)$ (resp.\  ${\rm  dist}_{<k}(x,y)$) is a formula asserting that the distance between $x$ and $y$ is exactly $k$ (resp.\ less than $k$). Then, for every integer $n$, the application of the simple interpretation defined by $\nu$ and $\phi$ on a colored path $P_{nk}$ with $nk$ vertices $v_0,\dots,v_{nk-1}$ where 
	we have  $M(v_i)$ if $\lfloor i/k\rfloor$ is odd produces $\mathsf C_k(P_n)$ (See \Cref{fig:sc_path}.
	In particular, $\mathsf C_k(\Pp)\subseteq \mathsf T(\Pp)$. Thus, the class $\Pp$ \dnnc.
\end{exa}

\begin{figure}[h!t]
\centering\includegraphics[width=.5\textwidth]{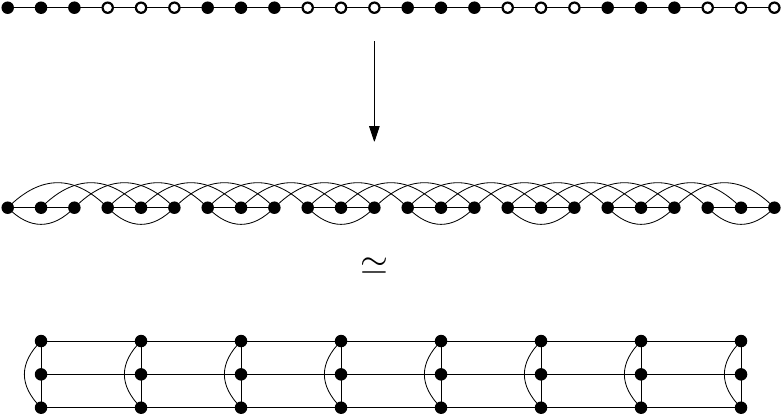}
\caption{The class of paths \dnnc: An illustration of \cref{ex:selfcopy} for $k=3$. The white vertices are those belonging to $M$.}
\label{fig:sc_path}
\end{figure}

\pagebreak
We take time for some observations.

\begin{lem}
	If $\Cc$ \dnnc and $\Cc\equiv_\FO\Dd$, then $\Cc\equiv_\FO^\circ\Dd$ and $\Dd$ \dnnc.
\end{lem}
\begin{proof}
	There is a non-copying transduction $\mathsf T$ and a copying transduction $\mathsf C_p$ with $\Dd\subseteq\mathsf T\circ\mathsf C_p(\Cc)$. 
	As $\Cc$ \dnnc, there is a non-copying transduction $\mathsf T'$ with $\mathsf C_p(\Cc)\subseteq\mathsf T'(\Cc)$. Hence, $\Dd\subseteq \mathsf T\circ\mathsf T'(\Cc)$ and thus $\Dd\sqsubseteq_\FO^\circ\Cc$. 
	
	On the other hand, there exists   a non-copying transduction $\mathsf X$ and a copying transduction~$\mathsf C_q$ with $\Cc\subseteq\mathsf X\circ\mathsf C_q(\Dd)$. It follows that for every integer $k$ we have
	$\mathsf C_k(\Cc)\sqsubseteq_\FO^\circ(\mathsf C_q(\Dd))$.
	In particular, $\mathsf C_q(\Cc)\sqsubseteq_\FO^\circ(\mathsf C_q(\Dd))$. Hence, according to \Cref{lem:cancel_copy},  $\Cc\sqsubseteq_\FO^\circ\Dd$.
	
	Consequently, $\Cc\equiv_\FO^\circ\Dd$ and, according to \Cref{lem:cancel_copy},
	for every integer $k$ we have
	$\mathsf C_k(\Cc)\equiv_\FO^\circ \mathsf C_k(\Dd)$.
	Thus, for every integer $k$ we have
	$\mathsf C_k(\Dd)\equiv_\FO^\circ \mathsf C_k(\Cc)
	\sqsubseteq_\FO^\circ \Cc\equiv_\FO^\circ\Dd$. Thus, $\Dd$ \dnnc.
\end{proof}

\begin{fact}
	A class $\Cc$ \dnnc if and only if $\mathsf C_2(\Cc)\equiv_\FO^\circ \Cc$.
\end{fact}
\begin{proof}
	By definition, if $\Cc$ \dnnc, then $\mathsf C_2(\Cc)\equiv_\FO^\circ \Cc$.
	
	Now assume $\mathsf C_2(\Cc)\equiv_\FO^\circ \Cc$.
	It is easily checked that for every positive integer $k$ there is a non-copying transduction~$\mathsf T_k$ such that $\mathsf T_k\circ\mathsf C_k\circ\mathsf C_2$ subsumes $\mathsf C_{2k}$. 	
	Thus, for every integer~$k$, if $\mathsf C_k(\Cc)\sqsubseteq_\FO^\circ \Cc$, we have
\begin{align*}
	\mathsf C_{2k}(\Cc)&\subseteq \mathsf T_k\circ\mathsf C_k\circ\mathsf C_2(\Cc)\\
	&\sqsubseteq_\FO^\circ \mathsf C_k\circ\mathsf C_2(\Cc)&\text{(as $\mathsf T_k$ is non-copying)}\\
	&\sqsubseteq_\FO^\circ \mathsf C_k(\Cc)&\text{(as $\mathsf C_2(\Cc)\equiv_\FO^\circ\Cc$)}\\
	&\sqsubseteq_\FO^\circ\Cc&\text{(as $\mathsf C_k(\Cc)\sqsubseteq_\FO^\circ \Cc$)}.
\end{align*}	

By induction over $k$ (starting from the base case $k=2$) we deduce that, for every integer~$k$ we have $\mathsf C_{2^k}(\Cc)\sqsubseteq_\FO \Cc$. According to \Cref{fact:red_copy}, we get
\[\Cc\sqsubseteq_\FO^\circ \mathsf C_k(\Cc)\sqsubseteq_\FO^\circ 
C_{2^k}(\Cc)\sqsubseteq_\FO^\circ \Cc.\]

Thus, $\mathsf C_k(\Cc)\equiv_\FO^\circ \Cc$. In other words, 
$\Cc$ \dnnc.
\end{proof}
\begin{fact}\label{fact:non-copy}
	A class $\Dd$ \dnnc if and only if, for every class $\Cc$,
	$\Cc\sqsubseteq_\FO\Dd$ is equivalent to $\Cc\sqsubseteq_\FO^\circ\Dd$.
\end{fact}
\begin{proof}
	Obviously, $\Cc\sqsubseteq_\FO^\circ\Dd$ implies $\Cc\sqsubseteq_\FO\Dd$.
	Consequently,  	$\Cc\sqsubseteq_\FO\Dd$ is equivalent to $\Cc\sqsubseteq_\FO^\circ\Dd$ if and only if $\Cc\sqsubseteq_\FO\Dd$ implies $\Cc\sqsubseteq_\FO^\circ\Dd$.

	Assume that $\Dd$ \dnnc.
	Assume that $\Cc$ is a class with $\Cc\sqsubseteq_\FO\Dd$. Then, there exists a transduction $\mathsf T=\mathsf I\circ\Gamma_{\mathcal U}\circ\mathsf C_k$ such that 
	$\Cc\subseteq \mathsf T(\Dd)$.
	Let $\mathsf T_1=\mathsf I\circ\Gamma_{\mathcal U}$.
	As $\Dd$ \dnnc there exists a non-copying transduction $\mathsf T'$ with $\mathsf C_k(\Dd)\subseteq \mathsf T'(\Dd)$.
	Hence, the non-copying transduction  $\mathsf T_1\circ\mathsf T'$ satisfies $\Cc\subseteq \mathsf T_1\circ\mathsf T'(\Dd)$, witnessing $\Cc\sqsubseteq_\FO^\circ\Dd$.
	Thus, if $\Dd$ \dnnc, then
	$\Cc\sqsubseteq_\FO^\circ\Dd$ implies $\Cc\sqsubseteq_\FO^\Dd$.
	
	Conversely, assume that for every class $\Cc$,  $\Cc\sqsubseteq_\FO\Dd$ implies 
	$\Cc\sqsubseteq_\FO^\circ\Dd$. Then, letting $\Cc=\mathsf C_k(\Dd)$, we get that $\mathsf C_k(\Dd)$ is a non-copying transduction of $\Dd$. Hence, $\Dd$ \dnnc.
\end{proof}

\begin{fact}\label{fact:non-copy-classes}
	If a class $\Cc$ is closed under adding pendant vertices (that is, if $G\in \Cc$
	and $v\in V(G)$, then $G'$, which is obtained from~$G$ by adding a new
	vertex adjacent only to $v$, is also in $\Cc$), then $\Cc$ \dnnc.
\end{fact}
\begin{proof}[Proof]
	Let $k\in\mathbb N$. Consider a set $\mathcal U=\{M_1,\dots, M_k\}$ consisting of $k$ unary relations and the non-copying transduction $\mathsf T$ defined by $\nu(x):=\bigvee_{i=1}^k M_i(x)$ and 
	\[\phi(x,y):= {\rm dist}_{2}(x,y)\vee \bigvee_{i=1}^k (M_i(x)\wedge M_i(y)\wedge {\rm dist}_3(x,y)),\]
	where ${\rm dist}_2(x,y)$ (resp.\ ${\rm dist}_3(x,y)$) is a formula asserting that the distance between~$x$ and~$y$ is exactly $2$ (resp.\ exactly $3$).
	For every $G\in\Cc$, let $G^+$ be the $\mathcal U$-colored graph obtained from~$G$ by adding  $k$ pendant vertices respectively marked $M_1,\dots,M_k$  to each vertex $v$. Then, applying the simple interpretation to $G^+$ we get $\mathsf C_k(G)$. In particular, $\mathsf C_k(\Cc)\subseteq \mathsf T(\Cc)$. Thus, $\Cc$ \dnnc.
\end{proof}
\section{Nearly $\Cc$-classes}

The results of this section will in some sense let us eliminate perturbations when considering weakly sparse classes. The main definition is inspired by structural graph theory.
Given two classes $\Cc$ and $\Dd$, we say that $\Cc$ is \emph{nearly} $\Dd$ if there exists a constant $h$ such that each graph $G \in \Cc$ can be obtained from a graph $H \in \Dd$ by adding at most $h$ vertices and connecting them in any way to each other and to $H$.

\begin{lem}
	\label{lem:perws}
	Let $\Cc$ and $\Dd$ be graph classes. If $\Cc$ is a weakly sparse perturbation of a weakly sparse class $\Dd$ and $\Dd$ is hereditary, then $\Cc$ is nearly $\Dd$.
\end{lem}
\begin{proof}
	As $\Cc$ and $\Dd$ are weakly sparse there is an integer $t$ such that $K_{t,t}$ is a subgraph of no graph in $\Cc$ or $\Dd$. According to Ramsey's theorem \cite[Theorem 9.1.]{diestel2012graph} there is an integer~$R(t)$ such that for any coloring of the edges of $K_{R(t),R(t)}$ or $K_{R(t)}$ in two colors, there is a monochromatic~$K_{t,t}$ subgraph.
	
	Let $\mathsf P=\oplus Z_1\oplus\dots\oplus Z_k$ be the perturbation that yields $\Cc$ when applied to $\Dd$.
	According to \Cref{lem:perturb}, for each graph $H\in\Cc$ there is a graph $G\in\Dd$ and a partition $(V_{\mathbf x})_{\mathbf x\in\mathbb{F}_2^k}$ such that $H$ is obtained from $G$ by flipping the edges $uv\in E(G)$ with 
	$u\in V_{\mathbf x}, v\in V_{\mathbf y}$, and $\langle\mathbf{x},\mathbf{y}\rangle=1$.
	
	Assume for contradiction that there are $\mathbf x$ and $\mathbf y$ such that $V_{\mathbf x}$ and $V_{\mathbf y}$ contain both more than $R(t)$ vertices and $\langle\mathbf{x},\mathbf{y}\rangle=1$.
	Consider a graph $B$ with vertex set $V_{\mathbf x}\cup V_{\mathbf y}$, whose edges are the pairs $uv$ with $u\neq v$, $u\in V_{\mathbf x}$ and $v\in V_{\mathbf y}$. Color $uv\in E(B)$ with color $1$ if $uv\in E(G)$ and with color $2$ otherwise, that is, if $uv\in E(H)$.
	By the definition of $R(t)$, either $G$ or $H$ contains a subgraph isomorphic to $K_{t,t}$, contradicting our choice of $t$.
	
	Let $S$ be the union of all the sets $V_{\mathbf x}$ with at most $R(t)$ vertices. Then, no edge $uv$ of $G$ with $u,v\in V(G)\setminus S$ is flipped by $\mathsf P$. Thus, $G-S=H-S$. 
	As $\Dd$ is hereditary, $G-S\in\Dd$. 
	It follows that every graph in $\Cc$ can be obtained by adding at most $2^kR(t)$ vertices to a graph in $\Dd$. Hence, $\Cc$ is nearly $\Dd$.
\end{proof}

Thus, as a corollary of \Cref{thm:rel-normal} we have
\begin{cor}
	\label{cor:rel-normal}
	Let $\Cc$ and $\Dd$ be graph classes. Assume $\Cc\sqsubseteq_\FO^\circ\Dd$ and  $\Cc$ is weakly sparse.
	Then there exists a weakly sparse hereditary  class~$\Ss$ such that 
	$\Cc$ is nearly $\Ss$ and $\Ss$ is an immersive transduction of~$\Dd$.
\end{cor}
\begin{proof}
	According to \Cref{thm:rel-normal}, 
	 there exist an immersive transduction $\mathsf T_{\rm imm}$, a perturbation~$\mathsf P$, and a weakly sparse hereditary  class $\Ss$ such that 
	$\mathsf P\circ\mathsf T_{\rm imm}$ subsumes $\mathsf T$ on the chain $(\Dd,\Ss,\Cc)$.
	In particular, $\Ss$ is an immersive transduction of $\Dd$ and, as	$\Cc$ is a weakly sparse perturbation of the weakly sparse and hereditary class $\Ss$, it follows from \Cref{lem:perws} that  $\Cc$ is nearly $\Dd$.
\end{proof}

%% file: examples.tex
\part{Examples in context}
\label{part:examples}
We now provide some examples of transductions and characterize the transductions of certain simple classes. Some of these examples will be used in \Cref{part:applications}. While the initial examples can be read knowing only the definition of a transduction, the characterizations will make use of (the statement of) the local normal form given in \cref{thm:normal}.

\section{Encoding graphs in interval graphs}
\begin{exa}
	\label{exa:interv}
	The class of all graphs is a non-copying transduction of the class of all interval graphs.  \end{exa}
\begin{proof}
	We consider  $\mathcal U$-expansions of interval graphs, where 
	$\mathcal U=\{M\}$ and $M$ is a unary relation symbol.
	Let $\mathsf T=\mathsf I\circ\Gamma_{\mathcal U}$ with $I=(M(x),\eta(x,y))$, where 
\[
\eta(x,y):=\exists z\, \biggl(E(x,z)\wedge E(y,z)\wedge
\bigl(\exists t\,(E(x,t)\wedge \neg E(z,t)\bigr)\wedge
\bigl(\exists t\,(E(y,t)\wedge \neg E(z,t)\bigr)\biggr).
\]

The interpretation $\mathsf I$ defines a graph whose vertex set is the set of all marked vertices, where two marked vertices $u$ and $v$ are adjacent if, in the source graph, they have a common neighbor $w$ with the property that both $u$ and $v$ have a  neighbor non-adjacent to $w$.

Let us prove that the class of all graphs is a $\mathsf T$-transduction of the class of all interval graphs.
Let $G$ be a graph with vertices $v_1,\dots,v_n$.
We define an interval graph $H$ defined as the intersection graph of the following family of intervals:
\begin{itemize}
	\item for each $i\in [n]$, $L_i=[4i-4,4i-3]$, $R_i=[4i-2,4i-1]$, and
	$I_i=[4i-4, 4i+3]$;
	\item for each edge $v_iv_j$ in $G$ with $i<j$, the interval $E_{i,j}=[4i-2,4j-3]$.	
\end{itemize}

The marked vertices are the intervals $I_1,\dots,I_n$. Then, every interval $E_{i,j}$ will define an edge linking $I_i$ and $I_j$ as 
$E_{i,j}$ is adjacent to both $I_i$ and $I_j$, $I_i$ is adjacent to $L_i$ (which is not adjacent to $E_{i,j}$) and $I_j$ is adjacent to 
$R_j$ (which is not adjacent to $E_{i,j}$); see \Cref{fig:interval} for an illustration.
\end{proof}

\begin{figure}[ht]
	\centering
	\includegraphics[width=.75\textwidth]{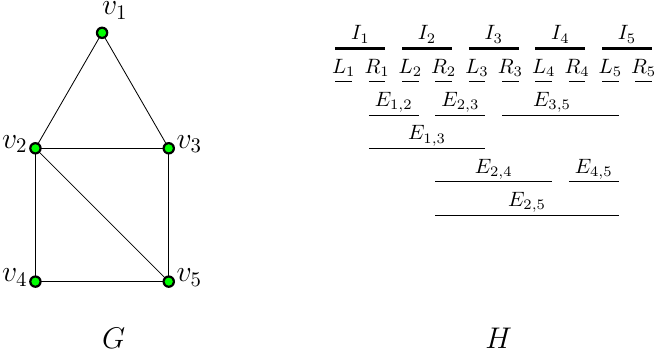}
	\caption{Encoding a graph $G$ in an interval graph $H$. The marked vertices are the intervals $I_1,\dots,I_5$.}
	\label{fig:interval}
\end{figure}
\section{Encoding grids in unit interval graphs}
\begin{exa}
	\label{exa:gr2ui}
	The class of all grids is a non-copying transduction of the class of all unit interval graphs.
\end{exa}
\begin{proof}
	Denote by $H_{n,m}$ the graph with $nm$ vertices that can be partitioned into $n$ cliques
	$V_1 =\{v_{0,0},\dots,v_{0,m-1}\}, \dots, V_n =\{v_{n,1},\dots,v_{n,m}\}$
	so that for each $i=1,\dots,n-1$ and for each $j=1,\dots,m$, vertex $v_{i,j}$ is adjacent to vertices $v_{i+1, 1}, v_{i+1, 2},\dots, v_{i+1, j}$ and there are no other edges in the graph between the classes $V_1,\dots,V_n$. An example of the graph $H_{5,5}$ is given in~\Cref{fig:H55}. 
	
	We consider a set $\mathcal U=\{M_0,M_1,M_2\}$ of $3$ unary relations and the transduction $\mathsf T$ defined by the formula $\nu(x):=\mathsf T$ and the formula $\phi(x,y)$ asserting that 
	\begin{itemize}
		\item either $x$ and $y$ are in some $M_i$ and the neighborhood in $M_{i-1}$ differ by exactly one vertex (where $i-1$ is meant modulo $3$),
		\item or $x$ is in $M_i$, $y$ is in $M_j$ with $i\neq j$, and there exists a vertex $x'$ in $M_i$ such that the symmetric difference between the neighborhoods of $x$ and $x'$ in $M_j$ is exactly $\{y\}$.
	\end{itemize}
	Then marking cyclically $V_1,\dots,V_n$ by $M_0,M_1,M_2$ witnesses that the $n\times m$ grid belongs to~$\mathsf T(H_{n,m})$.
\end{proof}

\begin{figure}[ht]
	\centering
	\includegraphics[width=\textwidth]{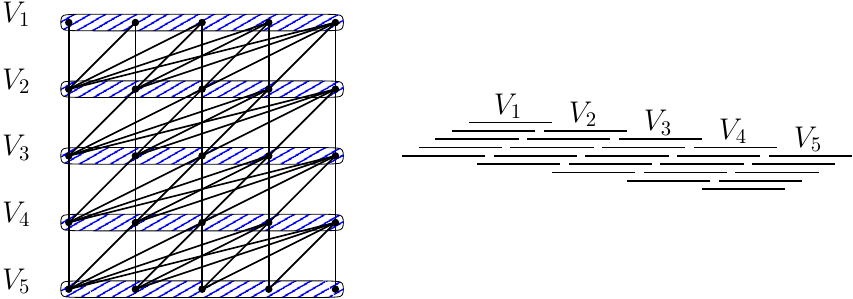}
	\caption{The unit interval graph $H_{5,5}$ and an interval representation of it.}
	\label{fig:H55}
\end{figure}
\section{Encoding classes with bounded pathwidth in the class of planar graphs}\label{sec:encoding-classes-with-bounded-pathwidth-in-the-class-of-planar-graphs}

Our last example is more involved.
Not only will this example show that a first-order transduction can be based on a non-trivial encoding, but it will also be used later to prove that the first-order transduction quasi-order is not a lattice (\Cref{thm:nolattice}).

\begin{lem}
	\label{lem:pl2pw}
	Every class of graphs with bounded \recalldef{pw}{pathwidth} is a non-copying transduction of the class  of planar graphs.
\end{lem}

\begin{proof}
	Before describing the transduction, we explain how we associate a host planar graph~$H$ to a graph $G$ with bounded pathwidth.
	
	Let $G$ be a graph with $n$ vertices and pathwidth $t$. By definition, $G$ is a spanning subgraph of an interval graph $K$ with $\omega(K)=t$.
	
	Consider an interval representation of $K$ where the endpoints of all the intervals are distinct and have coordinates $1,\dots,2n$.
	To each interval $[i,j]$ we associate a V-shape whose extremities are the endpoints of the interval. In the drawing formed by all the V-shapes, we put a black vertex at each intersection point of two V-shapes and a white vertex at the bottom of each V-shape. 
	Hence, each vertex $u$ of $K$ corresponds to a V-shape $\vee_u$ and a white vertex $\hat u$.
	We also add a segment joining two white vertices $\hat u$ and $\hat v$ if the V-shape $\vee_u$ encloses the V-shape $\vee_v$ and the segment joining $\hat u$ and $\hat v$ does not cross any V-shape. This drawing defines a planar graph $H$, whose vertex set is the set of all black and white vertices of the drawing, which we call a host planar graph of $G$. (See  \Cref{fig:pl2pw1} for an illustration of the construction.)

\begin{figure}[h!t]
	\centering
	\includegraphics[width=.7\textwidth]{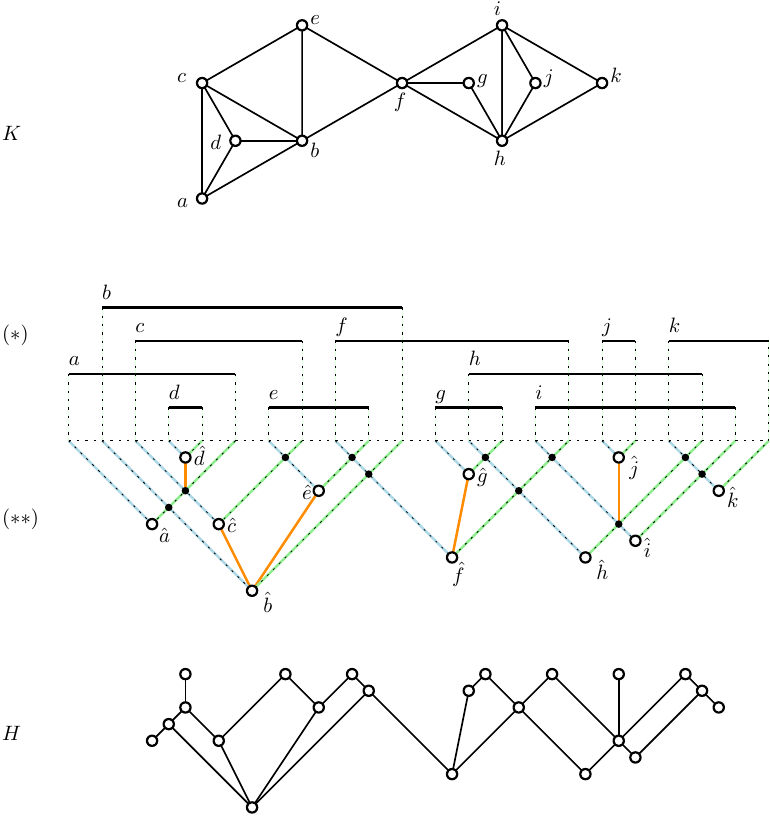}
	\caption{Host planar graph $H$ of a graph $G$ with bounded pathwidth, from an interval supergraph $K$ of $G$ with $\omega(K)={\rm pw}(G)$. The figure presents the computation of~$H$ from the interval graph $K$.}
	\label{fig:pl2pw1}
\end{figure}

We shall now describe a transduction $\mathsf T_t$ with the property that if $H$ is a host planar graph of a graph $G$ with pathwidth at most $t$, then $G\in\mathsf T_t(H)$.
By assumption, $H$ has been obtained from a drawing derived from an interval representation of an interval graph~$K$, which is a supergraph of $G$ with clique-number $t$. Because of this, the vertices of $H$ are naturally organized in layers, depending on the distance from the top of the V-shape drawing. These layers we mark by unary relations $M_0,\dots,M_{s}$ (with $s<2t$) and we further mark the vertices that were white in the V-shape drawing.
We denote by $H^+$ the obtained monadic expansion of $H$.
We say that a path in $H^+$ is monotone if the sequence of the layers of the vertices in the path (ordered naturally) is strictly monotone. It is easily checked that two vertices of $H^+$ corresponding to white vertices in the drawing (that is, to vertices of $K$) are adjacent in $K$ if and only if they are either joined by a monotone path, or  both joined to a third vertex by a monotone path.
(See \Cref{fig:pl2pw2} for an illustration.)

\begin{figure}[h!t]
	\centering
	\includegraphics[width=.7\textwidth]{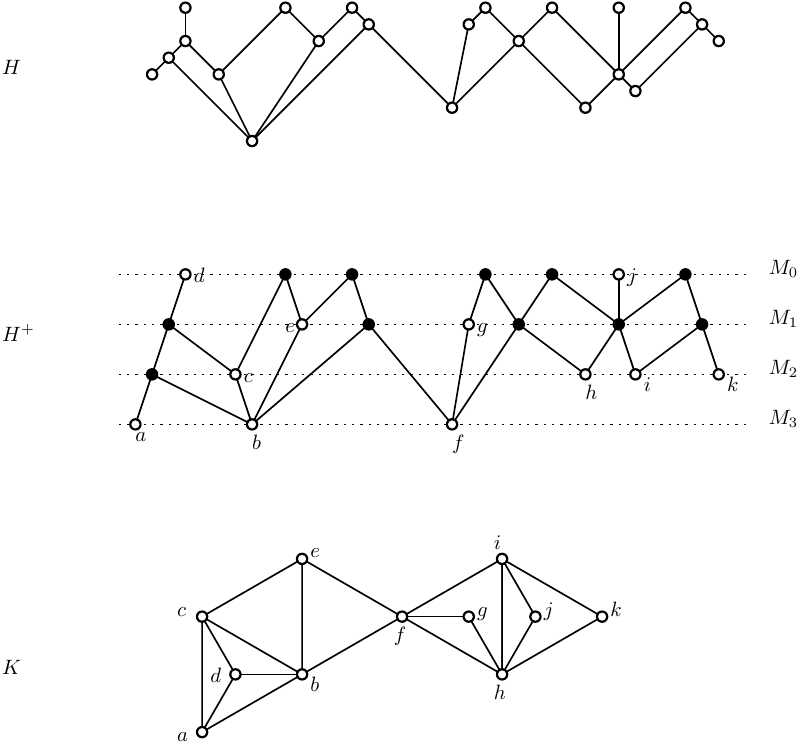}
	\caption{Encoding an interval graph $K$ with bounded clique number  in its host planar graph. Two vertices are adjacent if either they are joined by a monotone path, or they are both joined to a third vertex by a monotone path.}
	\label{fig:pl2pw2}
\end{figure}

Thus, there is a transduction $\mathsf P_t$, which allows obtaining, from a host planar graph of a graph $G$ with pathwidth $t$, a supergraph of $G$ that is an interval graph $K$ with clique-number~$t$. As interval graphs with clique-number $t$ have pathwidth $t$ and as the class of graphs with pathwidth at most $t$ has bounded expansion (hence bounded star chromatic number), we deduce from  \Cref{lem:monotone} that there exists a transduction $\mathsf S_t$ such that if $K$ has pathwidth at most $t$ then $\mathsf S_t(K)$ includes all the subgraphs of $K$. Thus, $\mathsf T_t=\mathsf S_t\circ\mathsf P_t$ is a transduction that witnesses that  the class of all graphs with pathwidth at most $t$ is a non-copying transduction of the class of all planar graphs.
\end{proof}

\input{lattice}

%% file: lattice.tex
\section{The transductions of edgeless graphs}
\label{symb:Ee}
We now turn to characterizing the transductions of certain simple classes. The first class we consider is the class $\Ee$ of edgeless graphs. It is clear that this class is the minimum of the transduction quasi-order (as we consider only infinite classes of graphs). Note that this class is addable.

\begin{prop}
	\label{lem:Ee}
	A class $\Cc$ is a transduction of $\Ee$ (or, equivalently, $\Cc\equiv_\FO\Ee$) if and only if $\Cc$ is a perturbation of a class whose members have connected components of bounded size. 
	Furthermore, a class $\Cc$ is a non-copying transduction of $\Ee$ if and only if $\Cc$ is a perturbation of $\Ee$.
\end{prop}
\begin{proof}
	Assume that $\Cc$ is a perturbation of a class $\Dd$ and that all graphs in $\Dd$ have connected components of size at most $n$. Let $\Dd'$ be the class of all graphs whose connected components have exactly $n$ vertices.
	Note that every graph in $\Dd$ is an induced subgraph of a graph in~$\Dd'$, so that $\Dd\subseteq\mathsf H(\Dd')$, where $\mathsf H$ is the hereditary transduction. 
	Let $\mathscr K$ be the class of all graphs whose connected components are cliques of order $n$. Note that $\mathscr K=\mathsf C_n(\Ee)$, where~$\mathsf C_n$ is the $n$-copy transduction.
	Let $F_1,\dots,F_N$ be an enumeration of all the (non-isomorphic) connected graphs with vertex set $[n]$.
	Let $\mathcal U=\{M_{i,j}\colon i\in [N], j\in [n]\}$.
	We define the transduction 
	$\mathsf T_0$ by its associated interpretation $\mathsf I_0=(\top,\eta(x,y))$, where 
	\[
		\eta(x,y):=\adjustlimits\bigvee_{i\in N}\bigvee_{(j,k)\in E(F_i)}M_{i,j}(x)\wedge M_{i,k}(y)\wedge E(x,y).
	\]

	Then, if the vertices of a clique of a graph in $\mathscr K$ are marked $M_{i,1},\dots,M_{i,n}$, the interpretation will construct a copy of $F_i$ on these vertices. It follows that 
	$\Dd'\subseteq \mathsf T_0(\mathscr K)$. Hence, $\Cc$ is a $\mathsf T$-transduction of $\Ee$, where $\mathsf T$ is the composition of $\mathsf C$, $\mathsf T_0$, the hereditary transduction $\mathsf H$, and a perturbation.

	Conversely, if~$\Cc$ is a transduction of $\Ee$ then, according to \Cref{thm:normal}, $\Cc$ is a perturbation of a class $\Dd\subseteq\mathsf T_{\rm imm}\circ \mathsf C_k(\Ee)$, where $\mathsf T_{\rm imm}$ is immersive and $\mathsf C_k$ is a $k$-copy transduction for some number~$k$. 
	The connected components of the graphs in $\mathsf C_k(\Ee)$ are cliques of size at most $k$ (as $\mathsf C_k$ is a $k$-copy transduction). Hence, 
	by the strong locality of the interpretation associated to $\mathsf T_{\rm imm}$, no connected component of a graph in $\Dd$ can have size greater than $k$.
	The last statement is obvious, as every immersive transduction of $\Ee$ is included in $\Ee$.
\end{proof}
\section{The transductions of cubic graphs}
Recall that a \emph{cubic graph} is a $3$-regular graph, i.e., a graph whose vertices all have degree $3$.
Cubic graphs form an important class of graphs, which is both very simple (having bounded degree) and structurally complex (having unbounded treewidth). 
The class of all cubic graphs  is addable
(the disjoint union of two cubic graphs is cubic). 

A  structural characterization of the graph classes that are \FO-transductions of classes of 
graphs of bounded degree is given in \cite{gajarsky2020new}.  The following proposition easily follows.
\begin{prop}
\label{prop:transcubic}
The following are equivalent for a class $\Cc$:  
\begin{enumerate}
	\item $\Cc$ is an \FO-transduction of the class $\Cubic$ of all cubic graphs,  
	\item $\Cc$ is an \FO-transduction of a class with bounded degree, 
	\item $\Cc$ is a perturbation of a class with bounded degree,
\end{enumerate}
\end{prop}
\begin{proof}
	(1) implies (2) as the class of cubic graphs is a class with bounded degree.
	
	Assume (2), that is,  that $\Cc$ is a $\mathsf T$-transduction of a class $\Dd$ with bounded degree.
	We denote by $D$ the maximum degree of the graphs in $\Cc$, that is, $D=\sup_{G\in\Dd}\Delta(G)$.
	According to  \Cref{thm:normal}, $\mathsf T$ is subsumed by the composition of a $k$-copy transduction $\mathsf C_k$, an immersive transduction $\mathsf T_{\rm imm}$, and a perturbation $P$.
	According to the definition of a $k$-copy transduction, the maximum degree of the graphs in $\mathsf C(\Dd)$ is at most $D+k-1$.
	Let $\mathsf I$ be the interpretation associated with the  transduction $\mathsf T_{\rm imm}$. Since $\mathsf T_{\rm imm}$ is immersive, $\mathsf I$ is strongly local, thus there exists an integer $r$ such that (for every colored graph $G$) two vertices adjacent in $\mathsf I(G)$ are at distance at most $r$ in $G$. Consequently, $\Delta(\mathsf I(G))\leq \max_{v\in V(G)} |B_r^G(v)|\leq \Delta(G)^{r+1}$.
	Thus, every graph $H\in\mathsf T_{\rm imm}\circ\mathsf C_k(\Dd)$ has maximum degree at most $(D+k-1)^{r+1}$. Hence, since $\Cc\subseteq \mathsf T(\Dd)\subseteq \mathsf P\circ\mathsf T_{\rm Imm}\circ\mathsf C_k(\Dd)$, the class $\Cc$ is a perturbation of the class
	$\mathsf T_{\rm Imm}\circ\mathsf C_k(\Dd)$, which has bounded degree. Hence, (3) holds. Thus, (2) implies (3).
	
	Now assume (3), that is, that we have
	$\Cc\subseteq \mathsf P(\Dd)$, for some perturbation $\mathsf P$ and some class $\Dd$ with bounded degree.
	Let $D=\sup_{G\in\Dd}\Delta(G)$.
	We now show how to associate a  cubic graph $H_G$ with every graph $G$ with maximum degree $D$.
	Let $D'\geq D$ be of the form~$3\cdot 2^{p-1}$ for some integer $p$ and let $Y$ be the  tree with height $p$
	and $\lfloor 3\cdot 2^{p-2}\rfloor$ leaves\footnote{The floor function is used to take care of the special case where $p=1$, in which $Y$ is a single vertex tree.}, whose internal vertices have degree $3$.
	Let $G$ be a graph with maximum degree $D$.  As $D'\geq D$, it is known (see e.g. \cite{erdos1963minimal})   that there exists a $D'$-regular graph $G'$ having $G$ as an induced subgraph. The graph $H_G$ is obtained by replacing each vertex $v$ of $G'$ by a copy of~$Y_v$ of $Y$, and by distributing the edges of $G'$ as between the leaves of the trees $Y_v$ as follows:
	if $u$ and $v$ are adjacent in $G$, then some leaf of $T_u$ is made adjacent to a leaf of $T_v$, while keeping the all degrees bounded by $3$. 
	The graph $H_G$ is cubic.  (See illustration in \Cref{fig:2cubic}.)
	\begin{figure}[ht]
		\centering
		\includegraphics[width=\textwidth]{2cubic}
		\caption{Construction of a cubic graph $H_G$ such that $G$ is an induced subgraph of  $G^{2p-1}$.}
		\label{fig:2cubic}
	\end{figure}
	
	Now remark that two vertices $u$ and $v$ of $G$ are adjacent in $G$ if and only if they are linked by some path of length $(2p-1)$ in $H_G$. In other words, $G$ is an induced subgraph of the $(2p-1)$-power of $H_G$.
	It follows that $\Dd$ is a transduction of the class of all cubic graphs. As $\Cc$ is transduction of $\Dd$, we deduce that (1) holds. Hence, (3) implies (1).
\end{proof}

Some further equivalences are given in \Cref{sec:dual}.
\section{The transductions of paths}
\label{sec:paths}

We now consider the graph class $\Pp$ of all paths. This class is  equivalent 
(by $\equiv_\FO^\circ$)
to the class~$\mathscr{L\!\!F}$ of all linear forests (that is: forests whose connected components are paths).
Indeed,  linear forest being exactly the induced subgraphs of paths, we have $\Pp\subseteq \mathscr{L\!\!F}=\mathsf H(\Pp)$.
Note that~$\mathscr{L\!\!F}$  is addable (that is, closed under taking disjoint unions). 
\begin{prop}
	\label{thm:transpath}
	A class $\Cc$ is an {\FO}-transduction of  the class  $\Pp$ of all paths if and only if it is a perturbation of a class with bounded bandwidth. 
	An \FO-transduction $\Cc$ of $\Pp$  is \FO-equivalent to $\Pp$ if and only if it is a perturbation of a class with bounded bandwidth that contains graphs with arbitrarily large connected components.
\end{prop}
\begin{proof}
	Assume 
	$\mathsf{T}$ is a transduction of~$\Cc$ in $\Pp$. 
	According to \Cref{thm:normal}, $\mathsf{T}\leq \mathsf P\circ\mathsf T_{\rm imm}\circ\mathsf C_k$, where $k\geq 1$, $\mathsf T_{\rm imm}$ is immersive, and $\mathsf P$ is a perturbation. 
	Observe first that $\mathsf{C}_k(\Pp)$ is included in the class of all 
	subgraphs of the~$(k+1)$-power of paths (See \Cref{fig:pathpower}). 
	
\begin{figure}[ht]
	\centering\includegraphics[width=\textwidth]{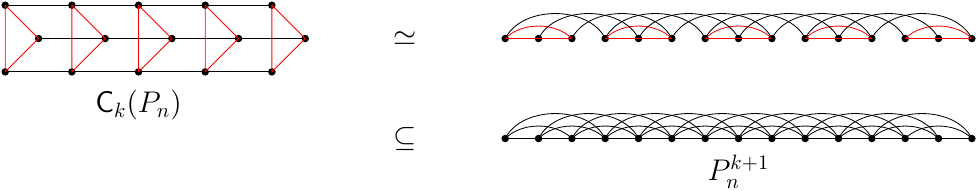}
	\caption{$\mathsf C_k(P_n)$ is a subgraph of $P_n^{k+1}$. The shown example corresponds to the case where $n=5$ and $k=3$.}
	\label{fig:pathpower}
\end{figure}	
	
	By the strong locality property of immersive transductions, 
	every class obtained from $\Pp$ by the composition of a copy operation and an immersive transduction has its image included in the class of all the  subgraphs of the $\ell$-power of paths, for some integer $\ell$ depending only on the transduction, hence, in a class of bounded bandwidth. 
	Conversely, assume that $\Cc$ is a perturbation of a class $\Dd$ containing graphs with bandwidth at most $\ell$.
	Then $\Dd$ is a subclass of the monotone closure (containing all subgraphs of the class) of the class $\mathcal P^\ell$ of $\ell$-powers of paths, which has bounded star chromatic number.
	By Corollary \ref{lem:monotone} and the 
	observation that taking the $\ell$-power is obviously a transduction, 
	we get that $\Cc\sqsubseteq_\FO\Pp$. 
	
	Assume $\Cc$ is a perturbation of a class  $\Dd$  with bandwidth at most $\ell$ that contains graphs with arbitrarily large connected components. As $\Dd$ has bounded degree thus contain graphs with arbitrarily large diameter, thus with arbitrarily long induced paths. We infer that $\Pp$ is an \FO-transduction of $\Cc$.
	
	However, if $\Cc$ is a perturbation of a class $\Dd$  with bandwidth at most $\ell$ that contains graphs with connected components with at most $k$ vertices; then $\Cc$ is a transduction of the class $\Ee$ of edgeless graphs, and it is known that $\Pp$ is not a transduction of $\Ee$ (even not an \MSO-transduction, see for instance \cite{blumensath10}).
\end{proof}

\section{The transductions of  cubic trees}
We next show that the class of all cubic trees is non-copying transduction-equivalent to the (addable) class of all subcubic forests. 

\begin{prop}
	\label{lem:BT}
	A class $\Cc$ is a transduction of  the class of cubic trees if and only if it is 
	 a perturbation of  a class with both bounded treewidth and bounded maximum degree.
\end{prop}
\begin{proof}
	Assume that $\Cc$  is a transduction of  the class of cubic trees.
According to \Cref{prop:transcubic}, the class $\Cc$ is a 
perturbation of a class $\Dd$ with bounded degree. 
Let $\mathsf P$ be this perturbation. 
If necessary, we restrict $\Dd$ to $\mathsf P(\Cc)$, so that 
$(\mathsf P,\mathsf P)$ is a pairing of $\Cc$ and $\Dd$.
The class $\Cc$, being a class of trees, has clique-width at most $3$ (folklore).  As clique-width boundedness is preserved by transductions \cite{colcombet2007combinatorial}, the class $\Dd$ has bounded clique-width. Having bounded degree, the class $\Dd$ is weakly sparse. As weakly sparse classes with bounded clique-width have bounded treewidth \cite{Gurski2000}, the class $\Dd$ has bounded treewidth.
Thus, $\Cc$ is a perturbation of a class having both bounded degree and bounded treewidth.
	
	Conversely, assume $\Cc$ is  a perturbation of a class $\Dd$ with both bounded treewidth and bounded maximum degree. 
	In \cite{DO95} it is proved that every graph $G$ with treewidth at most $w\geq 3$ and maximum degree $\Delta\geq 1$ is a subgraph of the intersection graph $I_G$ of a family $\{X_v\colon v\in V(G)\}$ of non-empty connected subtrees of a cubic tree $T_G$, where no more than   $k=48w\Delta$ subtrees $X_v$ contain a common vertex  and where all the trees $X_v$ have diameter at most $D=6\log w+12\log\Delta+30$.
	We subdivide all the edges of $T_G$ and of all the $X_v$ once, thus obtaining a subcubic tree $T_G'$ and a family of subtrees $X_v'$ with the property that if $X_u'$ and $X_v'$ contain adjacent vertices, then they contain a common vertex. Note that $I_G$ is the intersection graph of the $X_v'$, that no vertex of $G$ belongs to more than $k$ subtrees~$X_v'$, and that the diameter of all the $X_v'$ is bounded by $2D$.
	We now prove that there is a transduction that allows to obtain each $G\in\Cc$ from its associated cubic tree $T_G'$, thus proving that $\Cc$ is a transduction of the class of all subcubic trees. This transduction will be the composition of a $k$-copy transduction $\mathsf C_k$ and a transduction $\mathsf T$.
	It follows directly from the definition of treewidth that $I_G$ has treewidth at most $k-1$, hence is $k$-colorable. Let $c: V(I_G)\rightarrow [k]$ be a proper coloring of $I_G$. 
	Let $\mathsf C_k$ be the $k$-copy transduction.
	We identify the vertex set of~$\mathsf C_k(T_G')$ with $V(T_G')\times [k]$ in the natural way. We put a mark $M_i$ on $X_v\times\{i\}$ for each vertex $v$ of~$G$ with $c(v)=i$ and we mark $V$ exactly one vertex in each $X_v\times\{c(v)\}$  (for~$v$ of $G$).
   The interpretation part of the transduction $\mathsf T$ connects two vertices $u$ and $v$, for some $u$  marked $V$ and $M_i$ and some $v$ marked $V$ and $M_j$ if $j\neq i$ and there  exists 
   a path of length at most $4D$ from $u$ to $v$, where the vertices of the paths are first marked $M_i$ then~$M_j$. (Note that the last vertex marked $M_i$ has to be a clone of the first vertex marked $M_j$.)  Clearly, $I_G\in\mathsf T\circ\mathsf C_k(T_G')$.
	To end the proof, we remark that both cubic trees and graphs with treewidth at most $k$ have bounded star coloring number. Thus, according to \Cref{lem:monotone}, the class of subcubic trees is a transduction of the class of cubic trees (which we denote by $\mathsf S_1$) and the class $\Cc$ is a transduction of the class $\{I_G\colon G\in\Cc\}$ (which we denote by $\mathsf S_2$).
	Thus, $\Cc$ is a $\mathsf S_2\circ\mathsf T\circ\mathsf C_k\circ\mathsf S_1$-transduction of the class of all cubic trees.
\end{proof}

\section{The transductions of caterpillars}
A \emph{caterpillar} is a tree in which of all the vertices are within distance $1$ of a central path.
These trees have a special interest, witnessed by the property that these are exactly the connected graphs with pathwidth $1$ \cite{proskurowski1999classes}.
As every forest of caterpillars is obviously an induced subgraph of some caterpillar, the class of all caterpillars is (non-copying) transduction-equivalent to the (addable) class $\PW_1$ of all graphs with pathwidth $1$. 

The determination of which classes of graphs are transductions of $\PW_1$ is the occasion of introducing a standard model-theoretic tool, the Ehrenfeucht-Fra{\"\i}ss\'e game.

The \emph{Ehrenfeucht-Fra{\"\i}ss{\'e} game} is a well-known technique for determining whether two structures are ``equivalent''.
Suppose that we are given two graphs $G$ and $H$ and a natural number $q$. We can then define the Ehrenfeucht-Fra{\"\i}ss{\'e} game $\Game_q(G,H)$ to be a game between two players, Spoiler and Duplicator, played as follows:
Start with $A_0=B_0=\emptyset$ and let $\pi_0$ be the (empty) mapping from $A_0$ to $B_0$. Notice that $\pi_0$ is an isomorphism from $G[A_0]$ to $G[B_0]$.
For each $1\leq i\leq q$, Spoiler picks either a vertex $a$ in $G$ or a vertex $b$ in $H$.
In the first case, the Duplicator chooses a vertex $b$ in $H$; in the second case, he chooses a vertex $a$ in $G$.
Let $A_i=A_{i-1}\cup\{a\}$ and $B_i=B_{i-1}\cup\{b\}$.
If no isomorphism $\pi_i:G[A_i]\rightarrow G[B_i]$ extending $\pi_{i-1}$ exists such that $\pi(a)=b$, then Spoiler wins
the game. Otherwise, the game continues until $i=q$ (notice that $\pi_i$ is uniquely determined by $\pi_{i-1}$ and the vertices $a$ and $b$). If $i$ reaches $q$ and $\pi_q$ is an isomorphism from $G[A_q]$ to $H[B_q]$ then Duplicator wins the game.

If Duplicator has a winning strategy for $q$ then we write $G\equiv^q H$, and we
say that $G$ and $H$ are \emph{$q$-back-and-forth
	equivalent}.
The meaning of the $q$-back-and-forth equivalence is clarified by the following classical
result.

	\begin{thmC}[{\cite[Theorem 3.5]{Libkin04}}]
		Two graphs (and more generally two possibly infinite relational structures) are $q$-back and forth equivalent if and only if they satisfy the same first order sentences of quantifier rank $q$.
	\end{thmC}

Let $\mathcal U=\{M_1,\dots,M_p\}$ be a finite set of unary relations.
A \emph{compressed} $\mathcal U$-colored caterpillar is
a pair $(P,f)$, where $P$ is a $\mathcal U$-colored path
and $f: V(P)\times \mathcal P([p])\rightarrow\mathbb N$. The $\mathcal U$-colored caterpillar represented by $(P,f)$ is the caterpillar $\Upsilon(P;f)$ obtained by $P$ by adding at each  vertex $v$ of $P$, for each $I\subseteq [p]$, $f(v,I)$ vertices that belong exactly to the relations~$M_i$ with $i\in I$ (see \Cref{fig:compressed}); these vertices are called the \emph{$I$-children} of $v$ in $\Upsilon(P,f)$. We  further denote
by $L_{P,f}(v,I)$ the set of all the $I$-children of $v$ in $\Upsilon(P,f)$. 

\begin{figure}[ht]
	\centering\includegraphics[width=\textwidth]{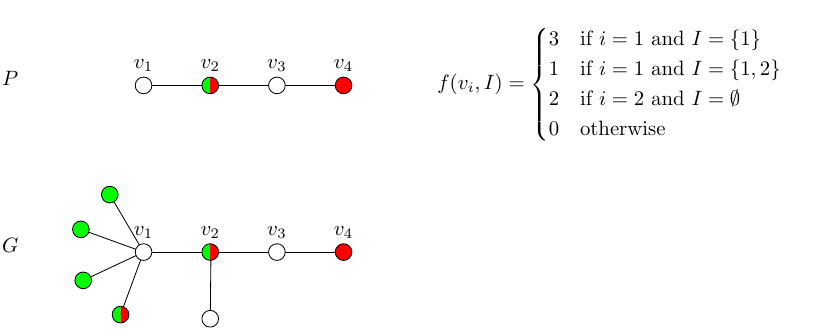}
\caption{A compressed $\{M_1,M_2\}$-colored caterpillar and 
	the caterpillar $G$ it represents. (White vertices are in no unary relations, green ones are in $M_1$, red ones are in $M_2$ and bicolored ones are both in $M_1$ and $M_2$.)}
	\label{fig:compressed}
\end{figure}

\begin{fact}
	\label{fact:auto}
		Let $(P,f)$ be a compressed caterpillar, and let $H=\Upsilon(P,f)$.
		Then, for every $v\in V(P)$ and every $I\subseteq[p]$, 
		every permutation of $L_{P,f}(v,I)$ defines an automorphism of $H$.
\end{fact}

\begin{cor}
	Let $\mathcal U=\{M_1,\dots,M_p\}$ be a set of unary relations,
	let $\mathsf I$ be an interpretation of graphs in $\mathcal U$-colored graphs, let
	$(P,f)$ be a compressed caterpillar, and let $H=\Upsilon(P,f)$.
	
	Then, for every $v\in V(P)$ and every $I\subseteq[p]$, 
	the vertices of $L_{P,f}(v,I)$ are clones in $\mathsf I(H)$.
	In particular, $L_{P,f}(v,I)$ is either an independent set or a clique of $\mathsf I(H)$.
\end{cor}
\begin{proof}
	According to \Cref{fact:auto}, every permutation of f $L_{P,f}(v,I)$ defines an automorphism of~$H$, hence an automorphism of $\mathsf I(H)$.
\end{proof}

Let $\mathcal U=\{M_1,\dots,M_p\}$ be a set of unary relations.

\begin{prop}
	\label{fact:cat_path}
	Let $\Delta$ be an integer. The class of $\mathcal U$-colored caterpillars with maximum degree $\Delta$ is a non-copying transduction of the class $\Pp$ of paths.
\end{prop}
\begin{proof}
	Let $\mathcal U'=\mathcal U\cup\{S,T\}$, where $S$ and $T$ are two unary symbols (not in $\mathcal U$).
	Let $G$ be a $\mathcal U$-colored caterpillar, and 
	let $(P,f)$ be a compressed caterpillar with $\Upsilon(P,f)=G$.
	Let $v_1,\dots,v_n$ be the vertices of $P$ (in order).
	For each $i\in[n]$ we associate with $v_i$ a path  $\mathcal U'$-colored path $Q_i$ consisting of a vertex marked $S$, a vertex marked the same way as $v_i$,
	for each $I\subseteq [p]$, $f(v_i,I)$ vertices marked by all the $M_i$ with $i\in I$, then a vertex marked $T$.
	We consider the interpretation $\mathsf I=(\neg S(x)\vee \neg T(x),\eta(x,y)$, where $\eta(x,y)$ expresses that
	 $x$ and $y$ are at distance at most $\Delta+2$, and 
	\begin{itemize}
		\item either one of $x$ and $y$ is adjacent to a vertex marked $S$ and no vertex  between $x$ and $y$ is marked $S$,
		\item or both $x$ and $y$ are adjacent to a vertex marked $S$ and there is only one vertex marked $S$ between $x$ and $y$.
	\end{itemize}
	Let $Q$ be the path obtained by concatenating the paths $Q_1,\dots,Q_n$. Then, $\mathsf I(Q)=G$ (cf \Cref{fig:cat_path}).
	It follows that the non-copying transduction defined by $\mathsf I$ witnesses that the class of $\mathcal U$-colored caterpillars with maximum degree $\Delta$ is a non-copying transduction of $\Pp$.
\end{proof}

\begin{figure}[h!t]
\centering\includegraphics[width=.75\textwidth]{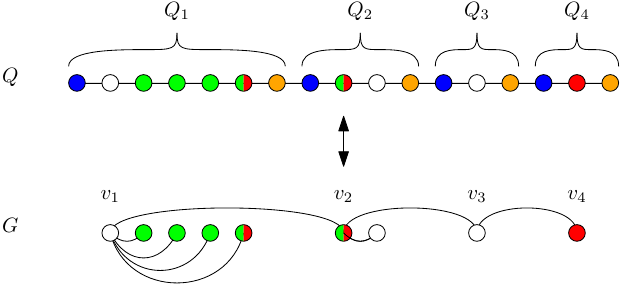}
\caption{Encoding a caterpillar with bounded degree in a path. Mark $S$ is blue, mark $T$ is orange.}
\label{fig:cat_path}
\end{figure}

		Let $(P,f_1)$ and $(P,f_2)$ be two compressed caterpillars 
based on  the same $\mathcal U$-colored path such that $\min(f_1(v,I),q)=\min(f_2(v,I),q)$ and $f_1(v,I)\leq f_2(v,I)$ for all $v\in V(P)$ and all 
$I\subseteq [p]$, and let $H_1=\Upsilon(P,f_1)$ and $H_2=\Upsilon(P,f_1)$, where $H_1$ is naturally identified with an induced subgraph of $H_2$.

\begin{lem}
	For every formula $\varphi(x,y)$ with quantifier rank at most $q-2$ and for every two vertices $u,v$ in $H_1$ we have 
	
	\[
	H_1\models \varphi(u,v)\qquad\iff\qquad H_2\models \varphi(u,v).\]
\end{lem}
\begin{proof}
	This follows from a standard argument based on an Ehrenfeucht-Fra{\"\i}ss\'e game.
\end{proof}

\begin{cor}
	\label{cor:cat_ind}
	For every interpretation $\mathsf I=(M_1(x),\eta(x,y))$ where $\eta$ has quantifier rank at most $q-2$, the graph 
	$\mathsf I(H_1)$ is an induced subgraph of $\mathsf I(H_2)$.
\end{cor}

\begin{lem} 
	\label{lem:cat_tech}
	Let $\mathcal U=\{M_1,\dots,M_p\}$ be a set of unary relations,
	let $\mathsf I=(M_1(x),\eta(x,y)$ be an interpretation of graphs in $\mathcal U$-colored graphs, let $q$ be at least the quantifier rank of~$\eta$ plus two, 
	and let  $G=\Upsilon(P,f)$.
	Define the function $\widehat{f}:V(P)\times\mathcal P([p])\rightarrow \{0,\dots,q\}$ by
	$\widehat f(v,I)=\min(f(v,I),q)$, and let $\widehat G=\Upsilon(P,\widehat f)$.
	
	Then, the graph $\mathsf I(G)$ 
	is obtained from $\mathsf I(\widehat G)$ as follows:
	for each $v\in V(P)$ and each $I\subseteq [p]$ with $1\in 1$ and $f(v,I)>q$, select arbitrarily a vertex $\ell(v,I)\in L_{P,\widehat f}(v,I)$ and
	blow the vertex $\ell(v,I)$ in $\mathsf I(\widehat{G})$ into a clique or an independent set of size $f(v,I)-\widehat{f}(v,I)+1$ (depending on whether $L_{P,\widehat f}(v,I)$ is, in $\mathsf I(\widehat{G})$, a clique or an independent set).
\end{lem}

\begin{proof}
	We prove the statement of the lemma by induction over 
	\[W(f):=\sum_{v\in V(P)}\sum_{I\subseteq [p]} |f(v,I)-\widehat f(v,I)|.\]
	If $W(f)=0$, then $f=\widehat f$, $\widehat{G}=G$ and $\mathsf{I}(\widehat{G})=\mathsf I(G)$. Hence, the base case obviously holds.
	Assume that the statement holds for $W(f)=k\geq 0$ and 
	let $f$ be such that $W(f)=k+1$. Then, there exists a vertex $v_0\in V(P)$ and a subset $I_0\subseteq [p]$ with 
	$f(v_0,I_0)>0$.  Let $f(v,I)=f(v,I)-\delta_{v,v_0}\cdot\delta_{I,I_0}$, where $\delta$ is Kronecker's symbol\footnote{By definition, $\delta_{x,y}$ is $0$ if $x\neq y$ and $1$ if $x=y$.}.
	Note that $G'=G-v_0$ is the $\mathcal U$-colored caterpillar represented by $(P,f')$.
	If $1\notin I$, it directly follows from \Cref{cor:cat_ind} that
	$\mathsf I(G')=\mathsf I(G)$. Thus, the statement of the theorem holds by the induction hypothesis.
	Otherwise, by~\Cref{cor:cat_ind}, we have $\mathsf I(G')=\mathsf I(G)-v_0$. By construction, there exists a vertex $u_0\in V(P)$ such that $v_0$ is an $I$-child of $u_0$ in $G$.  As the set of $I_0$-children of $u_0$ in~$G'$ has size at least two and are clones in $\mathsf I(G')$, there is no choice about how to define the adjacencies of $v_0$, and the statement of the lemma holds.
\end{proof}

\begin{thm}
	\label{thm:trans_cat}
	Let $\Cc$ be a (non-copying) transduction of the class of all caterpillars. If the class $\Cc$ has bounded maximum degree, then $\Cc$ is a transduction of the class of all paths. Otherwise, there exists a class $\Cc'$ with bounded degree that is a transduction of the class of all caterpillars, such that the graphs in $\Cc$ are obtained from graphs in $\Cc'$ by blowing each vertex either into an independent set or into a complete graph (of arbitrary size).
\end{thm}
\begin{proof}
	Let $\Cc$ be a  $\mathsf T$-transduction of the hereditary closure  $\Dd$ of the class of all caterpillars\footnote{as the hereditary closure of a class is non-copying transduction equivalent to it, we can consider the hereditary closure of the class of all the caterpillars instead of the class of all the caterpillars.}, where $\mathsf T$ is non-copying.  Let $\mathsf I$ be the transduction defining $\mathsf T$. 
	According to \Cref{lem:normal0}, we may assume that $\mathsf I=(M_1(x),\eta(x,y))$, where $M_1$ is a unary relation and $\eta$ is a local formula. Let $q$ be the quantifier-rank of $\eta$.
	Let $\Cc'$ be the class containing, for each graph $G\in\Cc$, the induced subgraph of $G$ obtained by deleting all the isolated vertices.
	Note that $\Cc'\subseteq \mathsf T(\Dd)$.
	
	Assume that $\Cc$ has bounded degree, and let $\Delta$ be the maximum degree of the graphs in~$\Cc$. Let $G\in\Cc'$ and let $H$ be a graph with a  minimum number of vertices such that $H\in\Dd$ and $G\in\mathsf T(H)$.
	As $G\in\mathsf T(H)$, there exists a unary expansion $H^+$ of $H$ with $G=\mathsf I(H^+)$. 
	Let $(P,f)$ be a compressed $\mathcal U$-colored caterpillar representing $H^+$, let $\widehat f$ be defined from $f$ as in  \Cref{lem:cat_tech}  and let $\widehat{H}^+$ be the $\mathcal U$-colored caterpillar represented by~$(P,\widehat f)$.
	According to \Cref{lem:cat_tech}  we get that $\mathsf I(G)$ contains a subset $X$ of clones of $G=\mathsf I(H^+)$ of size $\max_{v\in V(P)}\max_{1\in I\subseteq [p]} f(v,I)$. As $G$ contains no isolated vertex, any neighbor of some $v\in X$ has degree at least $|X|-1$. Hence, $\max_{v\in V(P)}\max_{1\in I\subseteq [p]} f(v,I)\leq \Delta+1$.
	It follows that if we choose $q\geq\Delta+1$, we have $\mathsf I(\widehat{H}^+)=\mathsf I(H^+)=G$.
	Note that $\widehat{H}^+$ is a $\mathcal U$-colored caterpillar with maximum degree at most $\Delta'=2^p (\Delta+1)+2$.
	 It follows that $\Cc'$ is a non-copying transduction of the class  $\Dd'$  of all the caterpillars with maximum degree~$\Delta'$.
Hence, according to \Cref{fact:cat_path}, the class $\Dd'$ is a non-copying transduction of the class $\Pp$ of all paths.	
As $\Ee$ is also a non-copying transduction of $\Pp$ (consider the transduction defined by the interpretation $(\top,\bot)$ that delete all the edges) and as $\Dd\subseteq\Dd'+\Ee$ (every graph in $\Dd$ is the disjoint union of a graph in $\Dd'$ and isolated vertices), we deduce from \Cref{fact:add_gluing} that $\Dd$ is a non-copying transduction of $\Pp$.

We now consider the case where $\Cc$ does not have bounded degree. Then, the conclusion directly follows from  \Cref{lem:cat_tech} where we choose $q$ to be the quantifier rank of $\eta$ plus two.
\end{proof}
\section{A distributive lattice among some low classes}

Recall that a \emph{lattice} is a partially ordered set in which every pair of elements has a unique supremum (also called a least upper bound or join) and a unique infimum (also called a greatest lower bound or meet). 
A lattice is \emph{distibutive} if the operations of join and meet distribute over each other. 
We shall see in \cref{thm:nolattice} that the transduction quasi-order is not a lattice. However, we now give some non-trivial examples of the existence of the greatest lower bound of two incomparable classes.

\begin{exa}
	\label{ex:BT}
	The class of all cubic trees is the greatest lower bound of the class of all trees and of the class of all cubic graphs.
\end{exa}
\begin{proof}
	That the class of all cubic trees is a lower bound of the class of all trees and the class of all cubic graphs is obvious, as it is included in both classes.
	
	Let  $\Cc$ be a lower bound of the class of all trees and the class of all cubic graphs. 
	Being a transduction of the class of all cubic graphs, $\Cc$ is a perturbation of a class $\Dd$ with bounded degree (by~\Cref{prop:transcubic}). Trees have cliquewidth $3$ and the property of a class of having bounded cliquewidth is preserved by transduction (see \cite{colcombet2007combinatorial}, for instance). Thus, 
	being a transduction of the class of all trees,  the class $\Dd$ has bounded cliquewidth. As the class $\Dd$ is also weakly sparse, it has bounded treewidth  \cite{Gurski2000}. 
	Thus, $\Cc$ is a perturbation of a class with both bounded degree and bounded treewidth. Hence, according to \Cref{lem:BT}, $\Cc$ is a transduction of the class of all cubic trees.
\end{proof}

Also, we have:
\begin{exa}
	\label{ex:P}
	The class of all paths is the greatest lower bound of the class of all caterpillars  and of the class of all cubic graphs.
\end{exa}
\begin{proof}
	That the class of all paths is a lower bound of the class of all caterpillars  and of the class of all cubic graphs is obvious, as it is included in both classes.
	
	Let  $\Cc$ be a lower bound  of the class of all caterpillars  and of the class of all cubic graphs.
	Being a transduction of the class of all cubic graphs, $\Cc$ is a perturbation of a class~$\Dd$ with bounded degree (by~\Cref{prop:transcubic}). Note that it follows that $\Dd$ is a perturbation of $\Cc$.
	As $\Cc$ is a transduction of the class of all caterpillars, so is $\Dd$.
	According to \Cref{thm:trans_cat}, $\Dd$~(and hence $\Cc$)  is a non-copying transduction of  the class of all paths.
\end{proof}
Note that in the above, the class of cubic graphs can be replaced by any class of bounded degree transducing all paths, for instance the class of all cubic trees or the class of all grids.

The lattice structure shown in~\Cref{fig:lattice} is effectively present in the quasi-order, in the sense that the meets and joins shown in the figure are actually meets and joins in the general quasi-order.

\begin{figure}[ht]
$$\xy
\xymatrix@!C=1cm{
	&\text{\scriptsize tree} \cup\text{\scriptsize cubic}\ar@{-}[dl]\ar@{-}[dr]\\
	\text{\scriptsize\bf tree}\ar@{-}[dr]&&\text{\scriptsize caterpillar}\cup\text{\scriptsize cubic}\ar@{-}[dl]\ar@{-}[dr]\\
	&\text{\scriptsize caterpillar} \cup\text{\scriptsize cubic tree}\ar@{-}[dl]\ar@{-}[dr]&&\text{\scriptsize star forest}\cup\text{\scriptsize cubic}\ar@{-}[dl]\ar@{-}[dr]\\
	\text{\scriptsize\bf caterpillar}\ar@{-}[dr]&&\text{\scriptsize star forest}\cup\text{\scriptsize cubic tree}\ar@{-}[dl]\ar@{-}[dr]&&\text{\scriptsize\bf cubic}\ar@{-}[dl]\\
	&\text{\scriptsize star forest}\cup\text{\scriptsize path}\ar@{-}[dl]\ar@{-}[dr]&&\text{\scriptsize\bf cubic tree}\ar@{-}[dl]\\
	\text{\scriptsize\bf  star forest}\ar@{-}[dr]&&\text{\scriptsize\bf path}\ar@{-}[dl]\\
	&\text{\scriptsize\bf edgeless}
}
\endxy$$
\caption{The lattice structure induced by the classes of all trees, by the class of all caterpillars (i.e. the class of all connected graphs with pathwidth $1$ \cite{proskurowski1999classes}), by the class of all star forests, by the class of edgeless graphs, by the class of all paths, by the class of all cubic trees, and by the class of all cubic graphs.
This hierarchy takes place in the lower part of Fig~\ref{fig:hierarchy}.}
\label{fig:lattice}
\end{figure}

%% file: local.tex
\part{Applications} \label{part:applications}
Our applications are divided into three parts. In structural graph theory, it is common to weaken properties by not insisting that they hold globally but only at every local scale, as in \cite{grohe2003local}. As might be expected, this interacts well with the locality of first-order logic, and we give some applications in the first section. The next section is concerned with the minimum and maximum classes for properties $\Pi$ of the transduction quasi-order, which serve as canonical obstructions for $\Pi$ or capture all the complexity of $\Pi$, respectively. This leads to transduction dualities, which are a framework for dichotomy statements akin to the Grid Minor Theorem \cite{robertson1986graph} stating that a class has bounded tree-width if and only if it does not contain arbitrarily large grids as minors (and which can be seen as an MSO-transduction duality for weakly sparse classes between the class of trees and the class of grids). Finally, we turn to dense analogues, which uses the transduction quasi-order to generalize various properties of weakly sparse classes to general hereditary classes.

\section{Local properties}\label{sec:monadic-dependence-and-monadic-stability}
We begin with yet another application of the local normal form, showing that all transductions of certain graph classes that are locally well-behaved can be obtained by perturbations of such classes. 
For a transduction \dset~$\Pi$ (a class property $\Pi$ is a \emph{transduction \dset} --- that is, a downset in the transduction quasi-order --- if it is preserved by transductions), the set of classes of graphs that are \emph{locally} $\Pi$, denoted $\mbox{\sf loc-}\Pi$, is the set of all classes~$\mathscr C$ such that for every integer $d$, the class $\rloc{d}{\Cc}$ of all balls of radius~$d$ of graphs in $\mathscr C$ belongs to $\Pi$, and we denote by $\overline{\mbox{\sf loc-}\Pi}$  the set of all the classes that are perturbations of a class in $\mbox{\sf loc-}\Pi$.
\begin{thm}
	\label{thm:local-prop}
	Let $\Pi$ be a transduction \dset. Then $\overline{\mbox{\sf loc-}\Pi}$  is a transduction \dset.
\end{thm}	
\begin{proof}
	Let $\mathscr C\in \overline{\mbox{\sf loc-}\Pi}$ and let $\mathsf T$ be a transduction.
	By definition, there exists a perturbation~$\mathsf P_0$ and a class $\mathscr C'\in\mbox{\sf loc-}\Pi$ with $\mathscr C\subseteq\mathsf P_0(\mathscr C')$. According to Theorem~\ref{thm:normal}, the transduction $\mathsf T\circ \mathsf P_0$ is subsumed by the composition $\mathsf P\circ\mathsf T_{\rm imm}\circ\mathsf C$ of a copying transduction, an immersive transduction, and a perturbation. Consider $H\in\mathsf T_{\rm imm}\circ\mathsf C(G)$, for $G\in\mathscr C'$. For every integer $d$, there exists an integer $d'$ such that for every vertex $v\in V(H)$ there is a vertex $v'\in V(G)$ (the projection of $v$) with the property that $B_d(H,v)$ is an induced subgraph of $\mathsf T_{\rm imm}\circ\mathsf C(B_{d'}(G,v'))$, thus in a fixed transduction of $B_{d'}(G,v')$.
	As $\mathscr C'$ is locally~$\Pi$, the class $\mathscr C_{d'}'=\{B_{d'}(G,v')\colon G\in\mathscr C', v'\in V(G)\}$ belongs to $\Pi$.
	As $\Pi$ is a transduction \dset, the class $\{B_d(H,v)\colon H\in \mathsf T_{\rm imm}\circ\mathsf C(\mathscr C')\}$ is also in $\Pi$. As this holds for every integer~$d$, $\mathsf T_{\rm imm}\circ\mathsf C(\mathscr C')\in\mbox{\sf loc-}\Pi$. Hence, $\mathsf T(\mathscr C)\in\overline{\mbox{\sf loc-}\Pi}$.
\end{proof}

Since having bounded linear cliquewidth, having bounded cliquewidth, and having bounded twin-width are transduction {\dset}s (See \cite{colcombet2007combinatorial,twin-width1}), we have
the following corollaries.

\begin{cor}
	\label{cor:loc_lcw}
	Every transduction of (a perturbation of) a class with locally bounded linear cliquewidth is a perturbation of a class with locally bounded linear cliquewidth.
\end{cor}	
\begin{cor}
	\label{cor:loc_cw}
	Every transduction of (a perturbation of) a class with locally bounded cliquewidth is a perturbation of a class with locally bounded cliquewidth.
\end{cor}	
\begin{cor}
	\label{cor:loc_tww}
	Every transduction of (a perturbation of) a class with locally bounded twin-width is a perturbation of a class with locally bounded twin-width.
\end{cor}	

Assume that a transduction \dset $\Pi$ is included in a set of classes~$\Pi'$. Then it directly follows from 
\Cref{thm:local-prop} that every transduction of a class that is locally~$\Pi$ is a perturbation of a class that is locally~$\Pi'$. 
Recall that a class is \emph{$\chi$-bounded} if there is a function $f$ such that every graph $G\in\mathscr C$ with clique number $\omega$ has chromatic number at most $f(\omega)$.
It is known that  every transduction of a class with bounded expansion is $\chi$-bounded~\cite[Corollary 4.1]{nevsetril2020rankwidth}. Hence, we have

\begin{cor}
	\label{cor:loc_SBE}
	Every transduction of a  locally bounded expansion class is a perturbation of a locally $\chi$-bounded class.
\end{cor}

We now turn to the local versions of some the most important model-theoretic properties in the quasi-order, namely (monadic) stability and (monadic) dependence. Recall (\cref{sec:dividing-lines}) that dependence and stability correspond to the impossibility to define arbitrarily large \emph{power set graphs} (that is, bipartite graphs with vertex set $[n]\cup 2^{[n]}$ where $i$ is adjacent to $I$ whenever $i\in I$) and arbitrarily large half-graphs (see \Cref{def:halfgraph}), respectively.
Further model-theoretic dividing lines may be derived by considering monadic expansions, i.e. expansions by unary predicates. We say that a class $\Cc^+$ of 
colored graphs is a \emph{monadic expansion} of a class $\Cc$ of graphs if each colored graph in $\Cc^+$ is a monadic expansion of some graph in $\Cc$. 
\begin{itemize}
	\item A class $\Cc$ of graphs is \emph{monadically dependent} if every monadic expansion of $\Cc$ is dependent;
	\item a class $\Cc$ of graphs is \emph{monadically stable} if every monadic expansion of $\Cc$ is stable.
\end{itemize}

These dividing lines are characterized by the impossibility to transduce some classes.

\begin{thmC}[\cite{baldwin1985second}]
	\label{thm:BS}
	A class $\Cc$ is \emph{monadically dependent} (or \emph{monadically NIP}) if the class of all graphs is not a transduction of $\Cc$; it is \emph{monadically stable} if the class of all half-graphs is not a transduction of $\Cc$.
\end{thmC}

The distinction between the classical versions of these properties and their monadic variants is actually not so important in the setting of hereditary classes, as witnessed by the following theorem.

\begin{thmC}[\cite{braunfeld2022existential}]
	Let $\Cc$ be a hereditary class of graphs. Then,
	\begin{itemize}
		\item $\Cc$ is monadically dependent iff $\Cc$ is dependent;
		\item $\Cc$ is monadically stable iff $\Cc$ is stable.
	\end{itemize}
\end{thmC}

We also introduce monadically straight classes. Recall the definitions of trivially perfect graphs and forest orders from \cref{sec:GT}.

\begin{defi}
	A class $\Cc$ is {\em monadically straight} if the class of trivially perfect graphs (equivalently, the class of forest orders) is not transducible from $\Cc$.
	\end{defi}

 We believe that monadic straightness could well be an important dividing line in the transduction quasi-order (and maybe from a general model-theoretic point of view).

\pagebreak
As above, we may consider the local versions of these properties, which we will prove to be equivalent to their non-local versions.

\begin{thm}
\label{thm:local}
For a class $\Cc$ of graphs we have the following equivalences:
\begin{enumerate}
\item $\Cc$ is locally monadically dependent if and only if $\Cc$ is
  monadically dependent;
\item $\Cc$ is locally monadically straight if and only if $\Cc$ is
  monadically straight;
\item $\Cc$ is locally monadically stable if and only if $\Cc$ is
  monadically stable.
\end{enumerate}
\end{thm}

\begin{proof}
	The proof will follow from the following claim. 

		\begin{claim}
\label{cl:golocal}
	Let $\Cc$ be a class such that the class $\Cc'=\{n(G\join K_1)\mid n\in\mathbb N, G\in\Cc\}$ is a transduction of $\Cc$. Then, for every class $\Dd$ we have
	$\Cc\sqsubseteq_\FO \Dd$ if and only if there exists some integer $r$ with $\Cc\sqsubseteq_\FO \rloc{r}{\Dd}$.
\end{claim}
\begin{claimproof}
	Let $\mathsf H$ be the hereditary transduction; which closes a class by taking induced subgraphs (see \Cref{ex:hered}).
	
	As $\rloc{r}{\Dd}\sqsubseteq_\FO \Dd$ (as witnessed by $\mathsf H$), if there exists some integer $r$ with $\Cc\sqsubseteq_\FO \rloc{r}{\Dd}$, then $\Cc\sqsubseteq_\FO \Dd$.
	Now assume  $\Cc\sqsubseteq_\FO\Dd$. 
		First, note that the class $\Cc'$ is addable (as $n(G+K_1)\union m(G+K_1)=(m+n)(G+K_1)$). 
	As $\Cc'\sqsubseteq_\FO\Cc$ (by assumption), we have 
	$\Cc'\sqsubseteq_\FO\Dd$. 
Let  $\mathsf T$ be a transduction such that $\Cc'\subseteq\mathsf T(\Dd)$.
	According to 	 
		  \Cref{lem:slunion},  there exist  of a copy operation $\mathsf C$ and an immersive transduction $\mathsf T_{\rm imm}$, such that $\mathsf T$ is subsumed by $\mathsf T_{\rm imm}\circ\mathsf C$. Let $\Dd'=\mathsf C(\Dd)$. According to \Cref{cor:apex}, there is an integer $r$ such that $\Cc\sqsubseteq_\FO^\circ \rloc{r}{\Dd'}$.
		  Note that, for every graph $G$ and every positive integer $r$, every ball of radius $r$ of $\mathsf C(G)$ is an induced subgraph of $\mathsf C(B)$, for some ball $B$ of radius $r$ of $G$. Consequently,
		 $\rloc{r}{\Dd'}\subseteq \mathsf H\circ \mathsf C(\rloc{r}{\Dd})$. Thus, we have	$\Cc\sqsubseteq_\FO \rloc{r}{\Dd'}\sqsubseteq_\FO \rloc{r}{\Dd}$.
\end{claimproof}

		The class $\{n(G\join K_1)\mid n\in\mathbb N, G\in\Gg\}$ is obviously a transduction of~$\Gg$. Hence, according to \Cref{cl:golocal}, 
		a class $\Cc$ is locally monadically dependent if and only if it is monadically dependent.
	The class $\{n(G\join K_1)\mid n\in\mathbb N, G\in\TP\}$ is a  transduction of~$\TP$. Hence, according to \Cref{cl:golocal}, 
		a class $\Cc$ is locally monadically straight if and only if it is monadically straight.
		The class $\{n(G\join K_1)\mid n\in\mathbb N, G\in\Hh\}$ is a  transduction of~$\Hh$.
	Hence, according to \Cref{cl:golocal}, 
		a class $\Cc$ is locally monadically stable if and only if it is monadically stable.	
\end{proof}

\begin{exa}
Although the class of unit interval graphs has unbounded clique-width,
	every proper hereditary subclass of unit interval graphs
	has bounded clique-width~\cite[Theorem 3]{UIcw}. 
	This is in particular the case for the class of unit interval graphs with bounded radius. 
		
	As having bounded cliquewidth is preserved by transductions and as the class of all graphs has unbounded cliquewidth,  classes with bounded cliquewidth are monadically dependent.
	Hence, the class of unit interval graphs, being locally monadically dependent, is monadically dependent. 
\end{exa}

We next prove the analogue of \Cref{thm:local} for stability and dependence, after recalling their definitions.

	For a formula $\phi(\bar{x},\bar{y})$ and a bipartite graph $H=(I,J,E)$, we say a graph $G$ {\em encodes~$H$ via~$\phi$} if
	there are sets $A=\set{\bar{a}_i | i \in I} \subseteq V(G)^{|x|}, B=\set{\bar{b}_j | j \in J} \subseteq V(G)^{|y|}$ such that 
	$G\models\phi(\bar{a}_i,\bar{b}_j)\Leftrightarrow H\models E(i,j)$ for all $i,j\in I\times J$.  
	
	Given a class $\Cc$ of graphs, {\em $\phi$ encodes $H$ in $\Cc$} if there is some $G \in \Cc$ encoding $H$ via $\phi$.

Recall that a formula $\phi(\bar{x}, \bar{y})$ with its free variables partitioned into two parts is {\em dependent} over a class $\Cc$ of graphs if there is some finite bipartite graph $H$ such that $\phi$ does not encode~$H$ in $\Cc$, while $\phi$ is {\em stable} over $\Cc$ if there is some half-graph $H$ such that $\phi$ does not encode~$H$ in $\Cc$. The class $\Cc$ is {\em dependent, resp.\ stable}, if every partitioned formula is dependent, resp.\ stable, over $\Cc$ (cf  \Cref{sec:dividing-lines}).

\begin{lem}
	If a class $\Cc$ of graphs is independent (resp.\ unstable), then there is a strongly local
	formula that is independent (resp.\ unstable). 
\end{lem}
\begin{proof}
	We make use of the standard fact that dependence of formulas is preserved by Boolean combinations \cite[Lemma 2.9]{simon2015guide}.
	 Suppose $\phi(\bar{x}, \bar{y})$ is independent. Using Gaifman normal form and \Cref{lem:sl}, we may rewrite $\phi$ as a Boolean combination of basic local sentences and strongly local formulas. Since no sentence can be independent, some strongly local formula must be. 
	
	The argument for stability is identical,  using the fact that stability of formulas is preserved by Boolean combinations \cite[Remark 3.4]{palacin2018introduction}.
\end{proof}

\begin{cor}
	A class of graphs is dependent (resp.\ stable) if and only if it is locally dependent (resp.\ locally stable).
\end{cor}

%% file: extended-dual.tex
\section{Maximum classes, minimum classes, and transduction dualities}\label{sec:dual}
For some properties, there is a unique (up to transduction-equivalence) maximum class with that property, which can be thought of as essentially exhibiting all the complexity of classes with the property. For example, the maximum class of bounded cliquewidth is the class of trivially perfect graphs (or forest orders; or trees, if considering \MSO-transductions) and the maximum class of bounded linear cliquewidth is the class of half-graphs (or linear orders; or paths, if considering \MSO-transductions), and these statements capture the ``tree-like'' and ``order-like'' behavior of these classes.

Dually, for some properties there is a unique minimum class without that property, which can be considered a canonical obstruction to that property. Model-theoretic properties are often defined by such canonical obstructions, such as (monadic) stability and the class of linear orders or (monadic) dependence and the class of all graphs. 

Some properties may admit both types of characterization, which we will term a \emph{(singleton) transduction duality}; this is in analogy with the well-established theory of homomorphism dualities \cite{nesetril2000duality}, which suggests many further lines of inquiry. More generally, a \emph{transduction duality} $(\Pi, \Psi)$ consists of two sets $\Pi$ and $\Psi$ of graph classes such that for any graph class~$\Cc$, either there is some $\Ff \in \Pi$ such that $\Cc \sqsubseteq_\FO \Ff$, or  there is some $\Dd \in \Psi$ such that $\Dd \sqsubseteq_\FO \Cc$
This formalizes a structural dichotomy stating that the obstructions to being in $\Pi$ are caputred by $\Psi$.  For example, while the classes of shrubdepth $\leq d$ form a strictly increasing chain in the transduction quasi-order \cite[Theorem 4.5]{ganian2019shrub} and so the bounded shrubdepth property does not admit a maximum class, \cite{shrubbery} shows a transduction duality where $\Pi = \{ \mathscr T_n \mid n \in \N\}$ where $\mathscr T_n$ is the class of trees of height $n$, and $\Psi$ is the singleton class of paths.

\subsection{A transduction duality for structurally bounded degree}\label{sec:sbddual}
Given a class property~$\Pi$, the class property {\em structurally $\Pi$} is the set of  graph classes that are transducible from a class in $\Pi$. From \Cref{prop:transcubic}, we see that the class of cubic graphs is the maximum class of structurally bounded degree. We now establish a transduction duality by showing the class of star forests is the minimum class not of structurally bounded degree. The proof will largely follow from known results on the model-theoretic property of \emph{mutual algebraicity}.

\begin{defi}
	Given a structure $M$, an $n$-ary relation $R(\xbar)$ is {\em mutually algebraic} if there is a constant $K$ such that for each $m \in M$, the number of tuples  $\bbar \in M^n$ such that $M \models R(\bbar)$ and $m \in \bbar$ is at most $K$. Note a unary relation is always mutually algebraic. Given a structure $M$, a formula $\phi(\xbar)$ is {\em mutually algebraic} if it defines a mutually algebraic relation in $M^{|\xbar|}$.
	
	A relational structure $M$ is {\em bounded degree} if every atomic relation is mutually algebraic.
	
	A relational structure $M$ is {\em mutually algebraic} if every formula with parameters from~$M$ is equivalent to a Boolean combination of mutually algebraic formulas with parameters from~$M$.
	
	A hereditary class $\Cc$ is {\em mutually algebraic} if every model of $\Th(\Cc) := \bigcap_{M \in \Cc} \Th(M)$ is mutually algebraic.
\end{defi}

We remark that \cite[Theorem 3.3]{MA} shows that mutual algebraicity of a theory is equivalent to {\em monadic {\NFCP}}. {\NFCP} is a model-theoretical property that is stronger than stability at the level of theories, although not at the level of individual formulas. At the level of formulas, {\NFCP} has appeared in \cite{DomInd} in the context of graph algorithms.

There are two main points in what follows. If $M$ is not mutually algebraic, then (some elementary extension of) $M$ transduces an equivalence relation with infinitely many infinite classes by \cite[Theorem 3.2]{Worst}. If $M$ is mutually algebraic, then some expansion naming finitely many constants is quantifier-free interdefinable with a bounded degree structure $M'$ by \cite[Lemma 4.3]{LT2}, which is thus transduction-equivalent to $M$. There is then some work to transfer these results to hereditary graph classes.

\begin{prop} \label{p:MA}
	Let $\Cc$ be a hereditary class of relational structures in a finite language. Then the following hold.
	\begin{enumerate}
		\item $\Cc$ is mutually algebraic if and only if $\Cc$ is (non-copying, quantifier-free) transduction-equivalent to a hereditary class of bounded degree structures,
		\item If $\Cc$ is mutually algebraic, then so is every transduction of $\Cc$,
		\item $\Cc$ is mutually algebraic if and only if $\Cc$ does not transduce the class of all equivalence relations.
	\end{enumerate}
\end{prop}
\begin{proof}
	(1) $(\Ra)$ By \cite[Proposition 4.12]{LT2}, if $\Cc$ is mutually algebraic then there are bounded degree hereditary classes $\tCC_1, \dots, \tCC_m$ such that each $\tCC_i$ is (non-copying, quantifier-free) transduction-equivalent to some $\Cc_i \subset \Cc$ and $\bigcup_{i \in [m]} \Cc_i = \Cc$. Thus $\Cc$ is (non-copying, quantifier-free) transduction-equivalent to $\bigcup_{i \in [m]} \tCC_i$	
	
	(2) Mutual algebraicity of $\Cc$ is clearly preserved by simple interpretations. By \cite[Theorem 3.3]{MA}, mutual algebraicity is preserved by coloring. We now show mutual algebraicity is preserved by copying. Suppose $\Cc$ is mutually algebraic, and by (1) let $T \colon \Dd \to \Cc$ be a non-copying transduction with $\Dd$ bounded degree. Then $T$ also induces a non-copying transduction from $C_k(\Dd)$ to $C_k(\Cc)$.  Taking copies preserves being bounded degree, bounded degree classes are mutually algebraic by \cite[Lemma 4.3]{LT2}, and we have already shown that non-copying transductions preserve mutual algebraicity.
	
	(1) $(\La)$ By \cite[Lemma 4.3]{LT2}, bounded degree classes are mutually algebraic, and we have shown mutual algebraicity is preserved by transductions.
	
	(3) $(\Ra)$ The class of all equivalence relations is not mutually algebraic, and mutual algebraicity is preserved by transductions.
	
	$(\La)$ From the proof of \cite[Theorem 3.2(4)]{Worst}, if $\Cc$ is not mutually algebraic then there is a formula $\phi(x,y)$ and $M \models Th(\Cc)$ containing elements $\set{b_i | i \in \N} \cup \set{a_{i,j} | i,j \in \N}$ such that $M \models \phi(a_{i,j}, b_k) \iff i = k$. By compactness, for every $n \in \N$ there is some $M_n \in \Cc$ containing elements $\set{b_i | i \in [n]} \cup \set{a_{i,j} | i,j \in [n]}$ such that $M_n \models \phi(a_{i,j}, b_k) \iff i = k$. Adding unary predicates $A  := \set{a_{i,j} | i,j \in [n]}$ and $B := \set{b_i | i \in [n]}$, we have that~$M_n$ transduces an equivalence relation with $n$ classes of size $n$ by the formula $\phi(x,y) := A(x) \wedge A(y) \wedge \exists z(B(z) \wedge \phi(x, z) \wedge \phi(y, z))$, and similarly can transduce any substructure of that equivalence relation.
\end{proof}

	When $\Cc$ is a class of graphs rather than relational structures, we may strengthen the result.

\begin{thm} \label{thm:dual}
	Let $\Cc$ be a hereditary graph class. Then the following are equivalent.
	\begin{enumerate}
		\item $\Cc$ is mutually algebraic;
		\item $\Cc$ is transduction-equivalent to (in particular, is a perturbation of) a hereditary class of bounded degree graphs;
		\item $\Cc$ is a transduction of the class $\Cubic$ of all cubic graphs,
		\item $\Cc$ does not (non-copying) transduce the class $\Tt_2$ of all star forests.
	\end{enumerate}
\end{thm}
\begin{proof}
	
	\Cref{p:MA} immediately gives $(2) \Rightarrow (1)$, \Cref{prop:transcubic} gives $(2)\Leftrightarrow(3)$,
	and from the fact that the class of star forests is transduction-equivalent to the class of all equivalence relations also gives $(1) \Leftrightarrow (4)$. It nearly gives $(1) \Rightarrow (2)$, but a priori only produces a class $\Dd_0$ of bounded degree relational structures, rather than graphs, such that $\Dd_0 \equiv^\circ_{\FO} \Cc$.
	
	From \cite[Definition 4.10]{LT2}, we see the relations of $\Dd_0$ are given by relations of arity no greater than those of $\Cc$, so they are unary and binary. The class $\Dd_0^-$ obtained by forgetting the unary relations still transduces $\Cc$ since the transduction can add the unary relations back.  We may view $\Dd_0^-$ as a class of directed edge-colored graphs, and we let $\Dd_1$ be the class of 2-subdivisions of the class of graphs obtained from~$\Dd_0^-$ by symmetrizing the edge relations and forgetting the edge-colors. Then $\Dd_1$ still has bounded degree, and~$\Dd_1$ transduces $\Dd_0^-$ as follows: we define a colored directed edge by taking a 2-subdivided edge, coloring the subdivision vertex closer to what will be the out-vertex of the directed edge with a special color, and then coloring the other subdivision vertex with the color of the desired directed edge. Thus $\Dd_1$ transduces $\Cc$. As every transduction of a class with bounded degree is a perturbation of a class with bounded degree, we are finished.
\end{proof}

In particular, we have the following singleton transduction-duality:
\[
\Tt_2\not\sqsubseteq_\FO \Cc\qquad\iff\qquad \Cc\sqsubseteq_\FO\Cubic.
\]

There is a correspondence between dualities and covers in the homomorphism quasi-order \cite{nesetril2000duality}, and we see a continued connection here.

\begin{cor}
	The only two $\sqsubseteq_\FO$-covers of the class of edgeless graphs are star forests and paths. 
\end{cor}
\begin{proof}
	Suppose $\Cc$ is a hereditary graph class that does not transduce star forests. By \Cref{thm:dual}, $\Cc$ is transduction-equivalent to a class $\Dd$ of graphs with bounded degree. 
	Suppose there is no bound on the size of connected components in $\Dd$. Since $\Dd$ has bounded degree, we can find arbitrarily long induced paths in the graphs in $\Dd$, hence we can transduce~$\Pp$ in $\Dd$.
	
	If there is a bound on the size of connected components, then $\Dd$ is \FO-transduction equivalent to  the class of edgeless graphs by \Cref{lem:Ee}.
\end{proof}

\subsection{The absence of maximum classes}
We prove a strong negative result, implying that no property containing all bounded expansion classes admits a maximum class. The idea is that we may ``diagonalize'' against all the countably many transductions from any candidate maximum class. On the other hand, we prove a positive result for countable sets of nowhere dense classes or classes of bounded expansion.

\begin{prop} \label{prop:nomax}
	Up to transduction equivalence, the class of all graphs is the only class that is an $\sqsubseteq_\FO$-upper bound for all classes with bounded expansion.
\end{prop}
\begin{proof}
	Assume $\mathscr B$ is a monadically dependent bound for all classes with bounded expansion.
	To a monadically dependent hereditary class $\mathscr C$ we associate the mapping $h_{\mathscr C}\colon\mathbb N^+\rightarrow\mathbb N$ as follows: $h_{\mathscr C}(i)$ is the largest integer $n$ such that the exact $i$-subdivision of $K_n$ belongs to $\mathscr C$ (this is well-defined by the assumption that $\Cc$ is monadically dependent).  
	We now consider an enumeration $\mathsf T_1,\mathsf  T_2,\dots$ of all first-order transductions (there are countably many) and let $\mathscr C_i=\mathsf T_i(\mathscr B)$. We further define the function  $h \colon \mathbb N^+\rightarrow\mathbb{N}$  by $h(i)=h_{\mathscr C_i}(i)+1$. 
	By taking $\Dd$ to be the hereditary closure of the set $\set{i\text{-subdivided } K_{h(i)} \colon i \in \N^+}$ we get a bounded expansion class with $h_{\mathscr D}(i)\geq h(i)$ for every integer $i$. 
	As $\mathscr B$ is a bound of all bounded expansion classes, there exists a transduction, say $\mathsf T_k$, such that 
	$\mathscr D\subseteq \mathsf T_k(\mathscr B)=\Cc_k$. However, this implies $h_{\mathscr D}(i)\leq h_{\Cc_k}(i)$ for every integer $i$, which contradicts 
	$h_{\mathscr D}(k)\geq h(k)\geq h_{\Cc_k}(k)+1$.
\end{proof}

However, we show that for countable sets of classes with bounded expansion, we can find an upper bound with bounded expansion.

\begin{lem}
	\label{lem:countBE}
	Let $\mathscr C_1,\dots,\mathscr C_k,\dots$ be countably many classes of graphs.
	
	Then, there exists a  class $\mathscr B$ such that $\mathscr C_i\sqsubseteq_\FO^\circ \mathscr B$ holds for all $i\in\mathbb N$, and such that
\begin{itemize}
	\item if every  $\Cc_i$ has bounded expansion, then $\mathscr B$ has bounded expansion;
\item if every $\Cc_i$ is nowhere dense, then $\mathscr B$ is nowhere dense.
\end{itemize}
\end{lem}
\begin{proof}
	Let $\mathscr B$ contain, for each $i\in\mathbb N$ the $i$-subdivision of all the graphs in $\mathscr C_i$. 
	
	By construction, the non-copying transduction $\mathsf T_i$ which connects vertices at distance exactly $i+1$ and keep only marked vertices is such that $\mathsf T_i(\mathscr B)\supseteq \mathscr C_i$. Thus, $\mathscr C_i\sqsubseteq_\FO^\circ \mathscr B$ holds for all $i\in\mathbb N$.
	
	Assume that the $k$-th subdivision $H^{(k)}$ of a graph $H$ belongs to $\mathscr B$, where $H$ has minimum degree at least $3$. Then, $H^{(k)}$ is the $i$-subdivision of some graph $H'\in\mathscr C_i$ for some $i\leq k$. It follows that 
	$\mathscr B\,\widetilde\nabla\,k\subseteq \bigcup_{i=1}^k \mathscr C_i\,\widetilde\nabla\,k$.
	Consequently, if all the $\mathscr C_i$'s  have bounded expansion then $\mathscr B$ has bounded expansion, and if all the $\mathscr C_i$'s  are nowhere dense then $\mathscr B$ is nowhere dense.
\end{proof}

\subsection{Bounded twin-width}
Graph classes of bounded twin-width were introduced in \cite{twin-width1}, and have been intensively studied since. While we will not need their definition for this section, we note that the property of bounded twin-width is preserved by FO-transductions, and any further required facts will be referenced. We provide a counterexample to a conjecture, stated in different terminology, that the cubic graphs are the minimum class of unbounded twin-width.

\begin{defi}
	Let $\Cc$ be a hereditary graph class. Then $\Cc$ is \emph{delineated} if for every hereditary subclass $\Dd \subseteq \Cc$, $\Dd$ has bounded twin-width if and only if $\Dd$ is monadically dependent.
	\end{defi}

In \cite[Conjecture 66]{bonnet2022twin8}, it was conjectured that if $\Cc$ is not delineated then $\Cc$ transduces the class $\Ss$ of subcubic graphs. This is equivalent to conjecturing that $\Ss$ is the minimum class of unbounded twin-width. In one direction, suppose $\Ss$ is the minimum such class; if $\Cc$ is not delineated then it has unbounded-twin width, and so $\Ss \sqsubseteq_\FO \Cc$. In the other direction, suppose the original conjecture holds. If $\Cc$ has unbounded twin-width then either it is not monadically dependent or it is not delineated (witnessed by considering $\Cc$ as a subclass of itself); in either case, $\Ss \sqsubseteq_\FO \Cc$.

\begin{defi}
	A graph class $\Cc$ is {\em small} (resp.\ {\em unlabeled-small)} if there is $c \in \R$ such that the number of labeled structures of size $n$ in $\Cc$ is $O(c^n \cdot n!)$ (resp.\ the number of unlabeled structures --- that is the number of structures up to isomorphism --- of size $n$ in $\Cc$ is $O(c^n)$).
\end{defi}

\begin{lem} \label{lem:bdsmall}
	Let $\Cc$ be an unlabeled-small class of graphs with bounded degree. Then every transduction $\Dd$ of $\Cc$ is also unlabeled-small.
\end{lem}
\begin{proof}
	Let $\mathsf T$ be a transduction so that $\Dd \subseteq \mathsf T(\Cc)$. By Theorem \ref{thm:normal}, we may decompose~$\mathsf T$ into the composition of a copy operation, an immersion, and a perturbation. Note that copying and perturbations both preserve being unlabeled-small since every structure in the image has a pre-image that is no larger, and that copying also preserves being bounded-degree; thus these operations have no effect and we may assume $\mathsf T$ is an immersion. Fix a target graph $D \in \Dd$ and let $C \in \mathsf T^{-1}(D)$. If $\mathsf T$ is strongly $r$-local, then let $C' \subseteq C \in \Cc$ be obtained by keeping the vertices of $C$ that remain in $\mathsf T(C)$, as well as their $r$-neighborhoods. By strong $r$-locality, $C' \in \mathsf T^{-1}(D)$. Since $\Cc$ has bounded degree, there is some $k$ such that $|C'| \leq k |D|$.
	
	Let $\Cc_n$ denote the number of unlabeled graphs of size $n$ in $\Cc$, and similarly $\Dd_n$. Suppose $\Cc_n = O(c^n)$. Then if the transduction uses $\ell$ colors, we have $\Dd_n \leq 2^{\ell kn} O(c^{kn})$ since every graph in $\Dd_n$ is the $\mathsf T$-image of a graph in $\Cc_{kn}$, so $\Dd$ is unlabeled-small.
\end{proof}

\begin{cor}
The class of subcubic graphs is not the $\sqsubseteq_\FO$-minimum class of unbounded twin-width.
\end{cor}
\begin{proof}
		Let $\Cc$ be a hereditary unlabeled-small class of graphs of unbounded twin-width and bounded degree (so $\Cc$ is monadically dependent), as constructed in \cite{bonnet2022twin7}, so $\Cc$ is not delineated. Since the class of subcubic graphs is not unlabeled-small \cite[Formula 6.6]{noyHandbook}, \Cref{lem:bdsmall} shows it cannot be transduced from $\Cc$.
		
		We remark that the class $\Cc$ is only shown to be small rather than unlabeled-small in \cite{bonnet2022twin7}, using \cite[Lemma 41]{bonnet2020twin2} that the class of induced subgraphs of the Cayley graph of a finitely-generated group is small. However, the same proof shows such a class is unlabeled-small; the only change needed is that Cayley's formula for the number of labeled rooted forests on $n$ vertices should be replaced by an exponential upper bound on the number of unlabeled rooted forests on $n$ vertices, which follows from such a bound on unlabeled rooted trees \cite[Theorem 4.4.6]{drmotaHandbook}.
\end{proof}

\begin{prob}
	Is (unlabeled-)smallness preserved by transductions?
	\end{prob}

	However, we have the following positive results concerning maximum classes.
	
\begin{thm}[{follows from \cite[Theorem 1.5]{TWW-factor}}]
	\label{thm:max_TWW}
	There exists a constant $c$ such that every class with bounded twin-width is a transduction of the class of all graphs with twin-width at most $c$. In other words, the bounded twin-width property has a maximum.
\end{thm}
	
	The following is also a direct consequence of results in \cite{TWW-factor}.
	
\begin{thm}
		\label{thm:max_sparse_TWW}
	The property of being monadically stable and having bounded twin-width  has a maximum, which can be chosen to be the class of all graphs with twin-width at most $c$ and girth at least $6$ (for some constant $c$).
\end{thm}
\begin{proof}
	It has been proved in \cite{gajarsky2022stable} that every monadically stable class of graphs with bounded twin-width is a transduction of a weakly sparse class with bounded twin-width.
	For $k,t\in\mathbb N$, let $\Cc_{k,t}$ be the class of all $K_{t,t}$-free graphs with twin-width at most $k$.
	According to \cite[Theorem~1.4]{TWW-factor},  there exists a constant $c$ and a function $f$ such that, for every graph in $G\in \Cc_{k,t}$, the $f(k,t)$-subdivision of $G$ has twin-width at most $c$ (and girth at least $6$). It follows that $\Cc_{k,t}$ is an $\FO$-transduction of the class $\mathscr B$ of all graphs with girth at least $6$ and twin-width at most $c$. Consequently, $\mathscr B$ is a maximum for the property of being monadically stable and having bounded twin-width.
\end{proof}

\subsection{Not a lattice}

We now show the transduction quasi-order is not a lattice, by showing the classes of planar graphs and of half-graphs have no greatest lower bound, i.e. there is no maximal class for the downset below both classes. This also shows that the join of infinitely many classes is not generally defined, as we could then define the meet of two classes as the join of all classes below both. We will make use of \cref{lem:pl2pw} showing that the class of planar graphs transduces every class of bounded pathwidth. We first show that the pathwidth hierarchy is strict in the transduction quasi-order, and thus has no maximum class.

\begin{thm}
	\label{thm:PW}
	For $n\geq 1$ we have $\Tt_{n+2}\subseteq_\FO\PW_{n+1}$ but
	\mbox{$\Tt_{n+2}\not\sqsubseteq_\FO\PW_n$}.
	  Consequently, $\PW_{n+1}\not\sqsubseteq_\FO \PW_{n}$.
\end{thm}

\begin{proof}
	For convenience, we define $\Tt_1=\PW_0=\Ee$, which is consistent with our definitions.
 We now prove $\Tt_{n+1}\not\sqsubseteq_\FO\PW_{n-1}$
	by induction on $n$. For $n=1$, this follows from \Cref{lem:Ee} (or from the known fact $\Tt_2\sqsupset\Ee$). 
	Now assume that we have proved the statement for~$n\geq 1$ and assume towards a contradiction that $\Tt_{n+2}\sqsubseteq_\FO\PW_{n}$.
	The class~$\PW_{n}$ is the monotone closure of the class $\mathcal{I}_{n+1}$ of interval graphs with clique number at most~$n+1$. As the class~$\mathcal{I}_{n+1}$ has bounded star chromatic number, it follows from \Cref{lem:monotone} that $\PW_{n}\sqsubseteq_\FO\mathcal{I}_{n+1}$. Let~$\mathsf T$ be the composition of the respective transductions, so that  $\Tt_{n+2}\subseteq\mathsf T(\mathcal{I}_{n+1})$.
	As $\Tt_{n+2}$ is additive, 
	it follows from \Cref{crl:slunion} that there is a copy operation $\mathsf C$ and an immersive transduction~$\mathsf{T}_0$ such that $\mathsf T_0\circ \mathsf C$ subsumes~$\mathsf T$.
	As $\{G+K_1: G\in\Tt_{n+1}\}\subseteq \Tt_{n+2}$ if follows from \Cref{cor:apex} that there exists an integer~$r$ such that $\Tt_{n+1}\sqsubseteq^\circ_\FO \rloc{r}{\mathsf C(\mathcal I_{n+1})}\subseteq \mathsf C(\rloc{r}{\mathcal I_{n+1}})$. According to \Cref{lem:slunion}, as $\Tt_{n+1}$ is additive, there exists an immersive transduction $\mathsf{T}_1$ with $\Tt_{n+1}\subseteq \mathsf T_1\circ\mathsf C(\rloc{r}{\mathcal I_{n+1}})$.
	
	Let $G\in \rloc{r}{\mathcal I_{n+1}}$ and let $P$ be a shortest path connecting the 
	minimal and maximal vertices in an interval representation of $G$. Then $P$ has length at most $2r$, $P$ dominates $G$, and  $G-P\in\mathcal I_n$. 
	By encoding the adjacencies in~$G$ to the vertices of $P$ by a monadic expansion, we get that that there exists a transduction~$\mathsf T_2$ (independent of our choice of $G$) such that $G$ is a $\mathsf T_2$-transduction of $H$, where $H$ is obtained from $G$ by deleting all the edges incident to a  vertex in $P$. 
	In particular, $\rloc{r}{\mathcal I_{n+1}}\subseteq\mathsf T_2(\mathcal I_n)$ thus
	$\Tt_{n+1}\subseteq\mathsf T_1\circ\mathsf C\circ\mathsf T_2(\mathcal I_n)$, which contradicts our induction hypothesis. 
\end{proof}

\begin{thm}
	\label{thm:nolattice}
The class $\Hh$  of half-graphs and the  class $\Pl$ of planar graphs have no greatest lower bound in $\sqsubseteq_\FO$ (or in $\sqsubseteq_\FO^\circ$).
\end{thm}
\begin{proof}
	Consider one of the quasi-orders $\sqsubseteq_\FO$,  $\sqsubseteq_\FO^\circ$, 
 which we will denote by $(\mathfrak{X},\sqsubseteq)$.
	Assume~$\Cc$ is the greatest lower bound of $\Hh$ and $\Pl$ in~$(\mathfrak{X},\sqsubseteq)$.
	According to~\Cref{thm:adler}, the class $\Pl$, being nowhere dense, is monadically stable. Thus, as $\Cc$ is (by assumption) a transduction of $\Pl$, it is also monadically stable. Hence, according to \cite{SODA_msrw}, there exists an integer $k$ such that $\Cc\sqsubseteq_\FO^\circ\PW_k$ (hence $\Cc\sqsubseteq\PW_k$). According to \Cref{thm:PW}, 
	$\PW_{k+1}\not\sqsubseteq_\FO\PW_k$ (hence $\PW_{k+1}\not\sqsubseteq\PW_k$). Thus,
	$\Cc\sqsubset\PW_{k+1}$. However, $\PW_{k+1}$ is addable (hence additive), $\PW_{k+1}\sqsubseteq_\FO^\circ\Hh$ and, according to \Cref{lem:pl2pw},
	$\PW_{k+1}\sqsubseteq^\circ\Pl$.
	In particular, $\PW_{k+1}\in\mathfrak{X}$, $\PW_{k+1}\sqsubseteq\Hh$, and $\PW_{k+1}\sqsubseteq\Pl$, 
	what contradicts the assumption that $\Cc$ is the greatest lower bound of $\Hh$ and $\Pl$ in $(\mathfrak{X},\sqsubseteq)$.
\end{proof}

On the other hand, it is easy to see that the least upper bound of two classes is always defined.

\begin{prop} \label{prop:sup}
	Let $\Cc_1$ and $\Cc_2$ be graph classes. Then $\Cc_1 \cup \Cc_2$ is the least upper bound of $\Cc_1$ and $\Cc_2$ in $\sqsubseteq_\FO$ (and in $\sqsubseteq_\FO^\circ$).
	\end{prop}
\begin{proof}
	We only do the proof for $\sqsubseteq_\FO$, since that for $\sqsubseteq_\FO^\circ$ is essentially identical. It is immediate that $\Cc_1 \cup \Cc_2$ transduces each of $\Cc_1$ and $\Cc_2$. So suppose $\Dd$ transduces $\Cc_1$ and $\Cc_2$ via transductions $\mathsf T_1$ and $\mathsf T_2$. We wish to encode these both in a single transduction (and for the proof to carry over to $\sqsubseteq_\FO^\circ$, we should not achieve this via copying). We do this by introducing a new color $U(x)$. We first perform the copy operation the max of the number of times it was performed in $\mathsf T_1$ and $\mathsf T_2$. Then, if any point is colored by $U$, we perform the interpretation from $\mathsf T_1$, and otherwise perform the interpretation from $\mathsf T_2$.
\end{proof}

%% file: dense-analog.tex
\section{Dense analogues and open problems}
\label{sec:analogue}
In this section, we discuss the notion of \emph{dense analogue} of a class property, as introduced in \cite{gajarsky2022stable}.
In the following, we consider only infinite hereditary classes.
Let $\Sigma$
 be the class property of being weakly sparse (that is the property of a class that there is some integer $s$ such that no graph in the class includes $K_{s,s}$ as a subgraph).
Unlike the model-theoretic dividing lines of stability and dependence for hereditary classes, the notion of weakly sparse  is not preserved by transductions, but finds its motivation in the study of sparse classes.
Recall that a class property $\Pi$ is a \emph{transduction \dset} if it is preserved by transductions. 
In other words, a transduction \dset consists of a \dset for the transduction order $(\mathfrak C,\sqsubseteq_\FO)$.  
The \emph{transduction closure}~$\Pi^{\downarrow}$ of a class property $\Pi$ is the property 
to be a transduction of a class in~$\Pi$, that is the smallest transduction \dset (for inclusion) that includes $\Pi$. (This is the same as the class property structurally $\Pi$.)
A \emph{transduction ideal} is a transduction \dset $\Pi$ such that every two elements of $\Pi$ have an upper bound in $\Pi$, i.e.\ if the union of any two classes in $\Pi$ is in $\Pi$ (by \cref{prop:sup}).  
A class property $\Pi_s$ is a \emph{sparse transduction \dset} if it contains only weakly sparse classes and is preserved by transductions to weakly sparse classes. In other words, a sparse transduction \dset is the trace on $\Sigma$ of a transduction \dset. A sparse transduction \dset $\Pi_s$ is a \emph{sparse transduction ideal} if the union of any two classes in~$\Pi_s$ is in $\Pi_s$.

The \emph{dense analogue}  $\Pi$ of a sparse transduction \dset $\Pi_s$ is the largest transduction \dset~$\Pi$ (for inclusion) with $\Pi\cap\Sigma=\Pi_s$.  This definition is valid as $\Pi_s^{\downarrow}\cap\Sigma=\Pi_s$ for every sparse transduction \dset, and the union of all transduction {\dset}s $\Pi$ with
$\Pi\cap\Sigma=\Pi_s$ is a transduction \dset with the same property.

\begin{prop}
	The dense analogue of a sparse transduction ideal is a transduction ideal.
\end{prop}
\begin{proof}
	Let $\Pi$ be the dense analogue of a sparse transduction ideal $\Pi_s$.
	Assume $\Cc_1,\Cc_2\in\Pi$, and let $\Dd$ be a weakly sparse transduction of $\Cc_1\union\Cc_2$. Then, 
	there is a partition $\Dd_1\union\Dd_2$ of~$\Dd$ with $\Dd_1\sqsubseteq\Cc_1$ and $\Dd_2\sqsubseteq\Cc_2$. As $\Dd$ is weakly sparse, so are $\Dd_1$ and~$\Dd_2$. Thus, both $\Dd_1$ and $\Dd_2$ belong to $\Pi_s$. As $\Pi_s$ is a sparse transduction ideal, $\Dd\in\Pi$. It follows that every weakly sparse transduction of $\Cc_1\union\Cc_2$ belongs to $\Pi_s$, thus $\Cc_1\union\Cc_2\in\Pi$.
	Therefore, $\Pi$ is a transduction ideal.
\end{proof}

\begin{prop}
	Let $\Pi$ be a transduction \dset.
	Then, $\Pi$ is the dense analogue of $\Pi\cap\Sigma$ if and only if 
	the complement $\Pi^*$ of $\Pi$ is such that 
	\[
	\Cc\in\Pi^*\quad\Longrightarrow\quad (\exists \Dd\in\Sigma\cap\Pi^*)\ \Dd\sqsubseteq_\FO\Cc\qquad\text{(\/property $\Pi^*$ is grounded in $\Sigma$)}.
	\]
\end{prop}
\begin{proof}
	Assume $\Pi$ is the dense analogue of $\Pi\cap\Sigma$ and let $\Pi^*$ be the complement of $\Pi$. 
	As~$\Pi$ is the dense analogue of $\Pi\cap\Sigma$, for every class $\Cc\notin\Pi$ there exists a weakly sparse class $\Dd\notin\Pi$ with $\Dd\sqsubseteq_\FO\Cc$. 
	
	Conversely,  let $\Pi_s=\Pi\cap\Sigma$. 
	For every $\Cc\notin\Pi$, we have $\Cc\in\Pi^*$.
	Thus, there exists $\Dd\sqsubseteq_\FO \Cc$ with $\Dd\in\Sigma\cap\Pi^*$. Hence, $\Dd$ is weakly sparse and $\Dd\notin\Pi$. It follows that $\Pi$ is the dense analogue of $\Pi_s$.
\end{proof}

\begin{exa}
	Monadic dependence
	is the dense analogue of nowhere dense.
	This can be expressed as follows.
	If a class is nowhere dense, then it is monadically dependent~\cite{adler2014interpreting}. 
	Conversely, assume that $\mathscr C$ is a hereditary weakly sparse monadically dependent class.  Assume for contradiction that $\mathscr C$ is not nowhere dense. Then, according to  \cite[Theorem 6]{DVORAK2018143}, there is some integer $k$ such that~$\mathscr C$ contains a $\leq k$-subdivision of every complete graph. It follows that the class of all graphs can be obtained as a transduction of $\mathscr C$, contradicting the hypothesis that $\mathscr C$ is 
	monadically dependent.
\end{exa}

Note that this shows that the dense analogue of $\Pi$ can be strictly larger than $\Pi^{\downarrow}$, since nowhere dense classes are monadically stable and thus their transduction closure is contained within monadic stability.

\begin{exa}
	Bounded shrubdepth is the dense analogue of bounded treedepth. Indeed, every weakly sparse class with bounded shrubdepth has bounded treedepth \cite[Proposition~6.4]{Sparsity}, while the class of paths (which has  unbounded treedepth) can be transduced from any class with unbounded shrubdepth \cite[Theorem 1.1]{shrubbery}.
\end{exa}

It would be tempting to think that the prominent role played by weakly sparse classes in our understanding of the transduction quasi-order is due to the fact that the complexity of a set of classes is determined by its relation to $\Sigma$. A strong formalization of this is given by the following order-theoretic notion: a subset $\Upsilon$ of classes is sup-dense in the transduction quasi-order $\sqsubseteq_\FO$  if, for every two classes $\Cc$ and $\Dd$ we have
\[
\Cc\equiv_\FO\Dd\qquad\iff\qquad\forall \Ff\in\Upsilon\quad (\Ff\sqsubseteq_\FO\Cc)\iff(\Ff\sqsubseteq_\FO\Dd).
\]

\begin{prop}
	\label{prop:non-sup-dense}
	The  set $\Sigma$ is not sup-dense in the FO-transduction quasi-order.
\end{prop}
\begin{proof}
	Let $\Cc$ be the class of all  planar graphs, and let $\Dd$ be the union of $\Cc$ and of the class~$\Hh$ of all half-graphs.
	As every transduction of $\Cc$ is monadically stable and $\Dd$ is not (as it includes $\Hh$), we have that
	$\Dd\sqsubseteq_\FO\Cc$ but 
	$\Dd\not\equiv_\FO\Cc$. 
		 However, according to \cite{SODA_msrw}, every weakly sparse transduction of half-graphs has bounded pathwidth, and thus by  \Cref{lem:pl2pw} is a transduction of $\Cc$. Thus $\Sigma$ is not sup-dense in the transduction quasi-order.
\end{proof}

While the classes $\Cc$ and $\Dd$ above had the same weakly sparse classes $\sqsubseteq_\FO$-below them, they had different weakly sparse classes $\sqsubseteq_\FO$-above them, since the only weakly sparse classes $\sqsubseteq_\FO$-above $\Dd$ are somewhere dense. We now show that even requiring that $\Cc$ and $\Dd$ have the same relation to every weakly sparse class does not force them to be $\equiv_\FO$-equivalent.

\begin{prop}
	There are hereditary graph classes $\Cc \not \equiv_\FO \Dd$ such that for every weakly sparse class $\Ff$, we have $\Ff \sqsubseteq_\FO \Cc \iff \Ff \sqsubseteq_\FO \Dd$ and $\Ff \sqsupseteq_\FO \Cc \iff \Ff \sqsupseteq_\FO \Dd$.
	\end{prop}
\begin{proof}
	Let $\Cc_0$ be a maximum monadically stable class of bounded twin-width as given by \cref{thm:max_sparse_TWW}, let $\Hh$ be the class of half-graphs, let $\mathscr{T\!\!P}$ be the class of trivially perfect graphs, let $\Cc = \Cc_0 \cup \Hh$, and let $\Dd = \Cc_0 \cup \mathscr{T\!\!P}$. Then,
	$\Dd\sqsubseteq_\FO\Cc$ but 
	$\Dd\not\equiv_\FO\Cc$.
	  By \cite{Gurski2000}, every weakly sparse transduction of $\mathscr{T\!\!P}$ or $\Hh$ has bounded treewidth, and thus is contained in $\Cc_0$. On the other hand, the only weakly sparse classes above either $\Cc$ or $\Dd$ are somewhere dense.
\end{proof}

We now turn to various conjectures concerning dense analogues, and prove some equivalences between them, for which the following proposition will be helpful.

\begin{prop}
	\label{prop:topmin}
	Let $\Pi_s$ be a sparse transduction \dset.
	Assume that $\Pi_s$ has the following characterization, where
	$\Ff\in\Sigma$  is a  non-empty weakly sparse class of graphs: 
	\begin{equation}
		\label{eq:topm}
	\Cc\in\Pi_s\iff \Cc\in\Sigma\text{ and }\exists F\in\Ff\text{ s.t. }\Cc\text{ excludes $F$ as a topological minor.}
\end{equation}

	Then, a weakly sparse class $\Cc$ belongs to $\Pi_s$ if and only if it has no transduction to a class $\Dd$ containing, for each $F\in\Ff$, some subdivision of $F$.
\end{prop}
\begin{proof}
	Let $\Cc$ be a weakly sparse class. We prove by contraposition that if $\Cc$ belongs to~$\Pi_s$, then $\Cc$ has no transduction to a class~$\Dd$ containing, for each $F\in\Ff$, some subdivision of $F$.
	Assume that $\Cc$ 
	has a transduction to a class $\Dd$ containing, for each $F\in\Ff$, some subdivision of $F$.
	By considering a subclass, we can assume that $\Dd$ contains no other graphs than these subdivisions. Then, $\Dd$ is weakly sparse (as $\Ff$ is weakly sparse) and, according to \eqref{eq:topm}, $\Dd\notin\Pi_s$. As $\Pi_s$ is a \dset, $\Cc\notin\Pi_s$.

	Conversely,	assume that $\Cc$ has  no transduction to a class containing, for each $F\in\Ff$, some subdivision of $F$.
	In particular, the hereditary closure $\mathsf H(\Cc)$ of $\Cc$ does not contain any subdivision of some $F\in\Ff$.
	In other words, no graph in $\Cc$ contains a subdivision of $F$ as an induced subgraph.
	As $\Cc$ is weakly sparse, it follows from~\cite{DVORAK2018143} that $\Cc$ has bounded expansion.  In particular, $\Cc$ has bounded star chromatic number. According to \Cref{lem:monotone}  the
	monotone closure $\Cc'$ of $\Cc$ is a transduction of $\Cc$. Considering this time the monotone closure transduction of $\Cc$, we get that there is a graph $F'\in\Ff$ such that 
	no graph in $\Cc'$ contains a subdivision of $\Ff'$ as a subgraph. By~\eqref{eq:topm}, it follows that $\Cc'\in\Pi_s$. As $\Cc\subseteq \Cc'$, we get $\Cc\in\Pi_s$ as well.
\end{proof}

Let us give a pair of examples of applications of \Cref{prop:topmin}.

\begin{exa}
	\label{exa:tw}
	Every weakly  sparse class with unbounded treewidth has a transduction   containing a subdivision of each wall.
\end{exa}
\begin{proof}
	Let $\Pi_s$ be the property of having bounded treewidth.
	It is well known that a weakly sparse class has bounded treewidth if and only if it excludes some wall as a topological minor.
	Hence, according to \Cref{prop:topmin}, a class has unbounded twin-width if and only if it has a transduction to a class  containing a subdivision of each wall.
\end{proof}

\begin{exa}
	\label{exa:pw}
		Every weakly  sparse class with unbounded pathwidth has a transduction   containing a subdivision of each cubic tree.
\end{exa}
\begin{proof}
	Let $\Pi_s$ be the property of having bounded pathwidth.
	It is well known that a weakly sparse class has bounded pathwidth if and only if it excludes some cubic tree as a topological minor.
	Hence, according to \Cref{prop:topmin}, a class has unbounded pathwidth if and only if  it has a transduction to a class  containing a subdivision of each cubic tree.
\end{proof}

It was conjectured in \cite{gajarsky2022stable} that cliquewidth is the dense analog of treewidth and linear cliquewidth is the dense analog of pathwidth. Every weakly sparse transduction of a class with bounded cliquewidth (resp. linear cliquewidth) has bounded treewidth (resp. bounded pathwidth). Hence, according to the characterization given in \Cref{exa:tw,exa:pw}, these conjectures can be restated as follows.

\begin{conj}
	\label{conj:lcw}
	Bounded linear cliquewidth is the dense analogue of bounded pathwidth.
\end{conj}
\begin{conj}
	\label{conj:cw}
	Bounded cliquewidth is the dense analogue of bounded treewidth
\end{conj}

A particular case concerns sparse transduction \dset $\Pi_s$ that do not include bounded pathwidth. Indeed, if some class $\Cc$ with bounded pathwidth does not belong to $\Pi_s$, this means that the class of half-graphs does not belong to the dense analogue of $\Pi_s$. In such a case, the dense analogue of $\Pi_s$ is stable and the conjecture restates as:
\begin{conj}
	Let $\Pi_s$ be a  sparse transduction \dset that does not include all classes with bounded pathwidth.
	Then, the  largest transduction \dset whose trace on $\Sigma$ is  $\Pi_s$ is the smallest transduction \dset including $\Pi_s$.
\end{conj}

An example where this conjecture holds is  the property of having bounded degree:

\begin{prop}
	Let $\Pi_s$ be the property of nearly having bounded degree.
	Then, $\Pi_s$ is a sparse transduction \dset and  the largest  transduction \dset $\Pi$ whose trace on $\Sigma$ is  $\Pi_s$ is the smallest transduction \dset including $\Pi_s$.
	
	Furthermore, $\Pi$ consists of the mutually algebraic (equiv.\ monadically NFCP) classes.
\end{prop}
\begin{proof}
	First note that, according to \Cref{prop:transcubic}, $\Pi_s$ is equivalent to the property of being a weakly sparse perturbation of a class with bounded degree and is a sparse transduction \dset.
	
	By \Cref{thm:dual} every mutually algebraic class of graphs is transduction-equivalent to a class of bounded degree graphs. Hence, as the property of being mutually algebraic is a transduction \dset, it is the smallest transduction \dset including $\Pi_s$.
	
	On the other hand, if $\Cc$ is a class that is not mutually algebraic, it  transduces star forests by \Cref{thm:dual}, hence it does not belong to $\Pi$. Thus, the property of being mutually algebraic is also the largest 
	transduction \dset whose trace on $\Sigma$ is  $\Pi_s$.
\end{proof}

A related problem is the relation between sparse and stable properties.

\begin{conj}
	\label{conj:stable}
	Let $\Pi_s$ be a sparse transduction \dset of bounded expansion classes with dense analogue~$\Pi$.
	Then,  the smallest transduction \dset $\Pi_s^{\downarrow}$ including $\Pi_s$  is the set of all  (monadically) stable classes in $\Pi$.
\end{conj}

This conjecture is deeply related to the studies in  \cite{SODA_msrw,nevsetril2020rankwidth,gajarsky2022stable}.
However, it is plausible that one cannot relax the condition that the classes in $\Pi_s$ have bounded expansion, because of the following equivalence.

\begin{prop}
	The following statements are equivalent:
	\begin{enumerate}
		\item 	For every sparse transduction \dset $\Pi_s$ with dense analogue $\Pi$, the  transduction \dset $\Pi_s^{\downarrow}$ is the set of all stable classes in $\Pi$;
		\item All monadically stable classes are sparsifiable (i.e.\ transduction-equivalent to some weakly sparse class).
	\end{enumerate}
\end{prop}
\begin{proof}
	(1)$\Rightarrow$(2): Let $\Cc$ be a monadically stable class, let $\Cc^{\downarrow}$ be the set of all transductions of $\Cc$ and let $\Pi_s=\Cc^{\downarrow}\cap\Sigma$. Obviously, $\Cc^{\downarrow}$ is included in the dense analogue of $\Pi_s$. Thus, (1) implies that $\Cc$ is a transduction of a class $\Dd\in\Pi_s$, which is itself a transduction of $\Cc$. Hence, $\Cc$ is sparsifiable.
	
	(2)$\Rightarrow$(1):  	Let $\Pi_s$ be a monotone sparse transduction \dset with dense analogue $\Pi$ and let $\Cc$ be a (monadically) stable class in $\Pi$.  By (2) $\Cc$ is transduction-equivalent to a weakly sparse class $\Dd$. As $\Pi\cap\Sigma=\Pi_s$, we have $\Dd\in\Pi_s$.  Hence, $\Cc\in\Pi_s^{\downarrow}$.
\end{proof}

For the case where $\Pi_s$ is the property of having bounded expansion, \Cref{conj:stable} restates as follows:
\begin{conj}
	\label{conj:BE}
	If a monadically stable class does not have structurally bounded expansion, then it has a transduction to a class that is nowhere dense but not bounded expansion.
\end{conj}

\begin{prop}
	\Cref{conj:stable,conj:BE} are equivalent.
\end{prop}
\begin{proof}
 The only direction to prove is that \Cref{conj:BE} implies \Cref{conj:stable}. Assume \Cref{conj:BE} holds and let $\Pi_s$ be a sparse transduction \dset of bounded expansion classes with dense analogue $\Pi$.
 Let~$\Cc$ be a monadically stable class. If $\Cc$ does not have structurally bounded expansion, then there exists a nowhere dense class $\Dd\sqsubseteq\Cc$ that does not have bounded expansion. Thus, $\Dd\notin\Pi_s$ and $\Cc\notin\Pi$.
 Otherwise, $\Cc$ is transduction-equivalent  to a class $\Dd$ with bounded expansion. Then, 
 \begin{align*}
 \Cc\in\Pi\iff\Dd\in\Pi\iff\Dd\in\Pi_s\iff\Cc\in\Pi_s^{\downarrow}. \tag*{\qedhere}
\end{align*}
\end{proof}

One main difficulty in proving all these conjectures is the non-monotonicity of the involved classes. The following, if true, would provide a great tool. Recall that a graph $G$ is \emph{$k$-degenerate} if every induced subgraph of $G$ has a minimum degree at most $k$. A class of graphs $\mathscr C$ is \emph{degenerate} if there exists $k\in\mathbb{N}$ such that all the graphs in $\mathscr C$ are $k$-degenerate. 

\begin{conj}
	\label{conj:nondeg}
	Every weakly sparse non-degenerate class has a transduction to a monotone non-degenerate  class.
\end{conj}

\pagebreak
\begin{conj}
	\label{conj:nondeg2}
	For every weakly sparse class of graphs $\Cc$,
	\begin{itemize}
		\item either $\Cc$ has a transduction to its monotone closure,
		\item or $\Cc$ has a transduction to a monotone nowhere dense class with unbounded fractional chromatic number.
	\end{itemize}
\end{conj}

Though \Cref{conj:nondeg2} seems to be stronger than \Cref{conj:nondeg}, this is not the case.

\begin{prop}
	\Cref{conj:nondeg,conj:nondeg2} are equivalent.
\end{prop}
\begin{proof}
	On the one hand, \Cref{conj:nondeg2} obviously implies \Cref{conj:nondeg}. On the other hand, assume \Cref{conj:nondeg} holds and let $\Cc$ be a weakly sparse class of graphs. If $\Cc$ has bounded expansion or is not monadically dependent, then $\Cc$ has a transduction to its monotone closure. Otherwise, $\Cc$ is nowhere dense but not bounded expansion. If $\Cc$ is degenerate, then it follows from \cite{DVORAK2018143} that $\Cc$ contains an shallow induced subdivision of graphs with arbitrarily large minimum degree, thus transduces onto a non-degenerate monotone (monadically dependent) class. On the other hand, if $\Cc$ is non-degenerate, we have the same property according to \Cref{conj:nondeg}.
	Thus, we can reduce to the case where $\Cc$ is a non-degenerate monotone  monadically dependent  class. In particular, $\Cc$ is nowhere dense. 
	Then it follows from \cite{chi_f} that $\Cc$ contains $1$-subdivisions of graphs with unbounded fractional chromatic number. 
\end{proof}